\documentclass[a4paper,dvipsnames,oneeqnum]{siamart250211}
\usepackage[utf8]{inputenc}

\usepackage{lipsum}
\usepackage{amsfonts}
\usepackage{amsmath,amssymb}
\usepackage{graphicx}

\usepackage{epstopdf}
\usepackage{algorithmic}
\usepackage[caption=false]{subfig}
\usepackage{mathtools,mathdots,mathrsfs}
\usepackage{cleveref}
\usepackage{bm}
\usepackage{mathtools,kbordermatrix}

\graphicspath{{./figures-arxiv/}}

\usepackage{tikz}
\usepackage{textcomp}

\ifpdf
  \DeclareGraphicsExtensions{.eps,.pdf,.png,.jpg}
\else
  \DeclareGraphicsExtensions{.eps}
\fi

\numberwithin{theorem}{section}

\newcommand{\TheTitle}{Spectral properties of kernel matrices \\ in the flat limit}
\newcommand{\TheShortTitle}{Spectral properties of kernel matrices in the flat limit}

\newcommand{\TheAuthors}{S.Barthelm\'{e} and K. Usevich}

\headers{\TheShortTitle}{\TheAuthors}

\title{{\TheTitle}\thanks{Submitted to the editors on DATE.
    \funding{This work was supported by ANR project GenGP (ANR-16-CE23-0008) and ANR project LeaFleT (ANR-19-CE23-0021-01).}}}

\author{
  Simon Barthelm\'{e}\thanks{CNRS, Univ. Grenoble Alpes,  Grenoble INP, GIPSA-lab, 38000 Grenoble, France (\email{simon.barthelme@gipsa-lab.fr}).}
   \and
  Konstantin Usevich\thanks{Universit\'{e} de Lorraine and CNRS, CRAN (Centre de Recherche en Automatique en Nancy), UMR 7039, Campus Sciences, BP 70239, 54506 Vand\oe{}uvre-l\`{e}s-Nancy cedex, France (\email{konstantin.usevich@univ-lorraine.fr}).}%
}

\ifpdf
\hypersetup{
  pdftitle={\TheTitle},
  pdfauthor={\TheAuthors}
}
\fi

\usepackage[mathscr]{euscript}
\usepackage{cleveref}

\newtheorem{remark}[theorem]{Remark}

\newtheorem{conjecture}{Conjecture}
\newtheorem{example}{Example}

\newcommand\numberthis{\addtocounter{equation}{1}\tag{\theequation}}

\newcommand{\ultriangle}{\blacktriangleright}

\newcommand{\CC}{\mathbb{C}}
\newcommand{\RR}{\mathbb{R}}
\newcommand{\HH}{\mathbb{H}}
\newcommand{\PP}{\mathbb{P}}
\newcommand{\ZZp}{\mathbb{Z}_{+}}
\newcommand{\norm}[1]{\ensuremath{\left\| #1\right\|}}

\newcommand{\calJ}{\mathcal{J}}
\newcommand{\calC}{\mathcal{C}}
\newcommand{\calB}{\mathcal{B}}
\newcommand{\calA}{\mathcal{A}}

\newcommand{\calY}{\mathcal{Y}}

\newcommand{\vect}[1]{\boldsymbol{#1}}
\newcommand{\matr}[1]{\boldsymbol{#1}}
\newcommand{\eqdef}{\stackrel{\textrm{def}}{=}}
\newcommand{\T}{{\sf T}}        
\newcommand{\stair}[1]{\mathscr{S}_{#1}}

\newcommand{\Dscale}[2]{\matr{\Delta}_{#1}(#2)}

\newcommand{\va}{\vect{\alpha}}
\newcommand{\vb}{\vect{\beta}}

\newcommand{\bA}{\matr{A}}

\newcommand{\bI}{\matr{I}}

\newcommand{\bK}{\matr{K}}       
\newcommand{\bD}{\matr{D}}

\newcommand{\X}{\mathcal{X}}       

\renewcommand{\O}{\mathcal{O}}       
\DeclareMathOperator{\rank}{rank}
\DeclareMathOperator{\mspan}{span}

\DeclareMathOperator{\diag}{diag}
\DeclareMathOperator{\Tr}{Tr}

\DeclareMathOperator{\blkdiag}{blkdiag}

\newcommand{\Qfull}{\matr{Q}_{\rm full}}    
\newcommand{\Rfull}{\matr{R}_{\rm full}}  
\newcommand{\Qthin}{\matr{Q}}    
\newcommand{\Rthin}{\matr{R}}    
\newcommand{\Qort}{\matr{Q}_{\bot}}    
\newcommand{\QortSub}[1]{\matr{Q}_{\bot,#1}}    
\newcommand{\Wronsk}[1]{\matr{W}_{\le #1 }}

\newcommand{\tikzoverset}[2]{%
  \tikz[baseline=(X.base),inner sep=0pt,outer sep=0pt]{%
    \node[inner sep=0pt,outer sep=0pt] (X) {$#2$}; 
    \node[yshift=3pt] at (X.north) {$#1$};
}}

\newcommand{\arrowAn}[1]{
   \hbox{\;\!\!\scalebox{0.6}{\normalfont  {#1}\!\!\!\, {\Large\raisebox{0.6\height}{\Huge \texttildelow}}\,}}}

\newcommand\tildeN[2]{\tikzoverset{\arrowAn{#1}}{#2}}

\newcommand{\coefAtMonomial}[2]{[#2]\left\{#1\right\}}

\definecolor{darkgreen}{rgb}{0,0.6,0}

\Crefname{figure}{Fig.}{Figs.}

\begin{document}

\maketitle
\begin{abstract}
Kernel matrices are of central importance to many applied fields. 
In this manuscript, we focus on spectral properties of kernel matrices in the so-called ``flat limit'', which occurs when points are close together relative to the scale of the kernel.
We establish asymptotic expressions for the determinants of the kernel matrices, which we then leverage to obtain asymptotic expressions for the main terms of the eigenvalues.
Analyticity of the eigenprojectors  yields expressions for limiting eigenvectors, which are strongly tied to discrete orthogonal polynomials.
Both smooth and finitely smooth kernels are covered, with stronger results available in the finite smoothness case.
\end{abstract}



\vspace{4em}

\section{Introduction}
For an ordered set of points $\X = (\vect{x}_1, \ldots, \vect{x}_n)$,  $\vect{x}_k \in \Omega \subset \RR^{d}$, not lying in general on a regular grid, and a kernel $K: \Omega  \times \Omega \to \RR$, the  kernel matrix $\matr{K}$ is defined as
\[
\matr{K}  = \matr{K}(\X) = \left[ K(\vect{x}_i,\vect{x}_j)\right]_{i,j=1}^{n,n}.
\]
These matrices occur in approximation theory (kernel-based approximation and
interpolation, \cite{schaback2006kernel,wendland2004book}), statistics and machine learning
(Gaussian process models \cite{williams2006gaussian}, Support Vector Machines
and kernel PCA \cite{scholkopf2002learning}).
Often, a  positive  scaling parameter is introduced\footnote{In this paper, for simplicity, we consider only the case of isotropic scaling (i.e., all the variables are scaled with the same parameter $\varepsilon$).
The results should hold for the constant anisotropic case, by rescaling the set of points $\X$ in advance.}, and the scaled kernel matrix becomes
\begin{equation}\label{eq:scaled_kernel}
\matr{K}_{\varepsilon} = \matr{K}_{\varepsilon} (\X) =  \left[ K_{\varepsilon}(\vect{x}_i,\vect{x}_j) \right]_{i,j=1}^{n,n}, 
\end{equation}
where typically $K_{\varepsilon}(\vect{x},\vect{y}) = K(\varepsilon\vect{x},\varepsilon\vect{y})$.
If the kernel is radial (the most common case), then its value depends only on the Euclidean distance between   $\vect{x}$ and $\vect{y}$,  and $\varepsilon$ determines how quickly the kernel decays with  distance. 

Understanding  spectral properties of  kernel matrices is essential in statistical applications (e.g., for  selecting hyperparameters), as well as in scientific computing (e.g., for preconditioning \cite{fornberg2011stable,wathen1987preconditioning}). Because the spectral properties of kernel matrices are not directly tractable in the general case, one usually needs to resort to asymptotic results.
The most common form of asymptotic analysis takes $n \rightarrow \infty$.
Three cases are typically considered: (a) when the distribution of points in $\X$
converges to some continuous measure on $\Omega$, the kernel matrix tends in some sense to a
linear operator in a Hilbert space, whose spectrum is then studied
\cite{von2008consistency}; (b) recently, some authors have obtained asymptotic results in a regime where both $n$
and the dimension $d$ tend to infinity
\cite{elkaroui2010spectrum,couillet2016kernel,wang2018numerical}, using the tools of random matrix theory; 
(c) in a special case of $\X$ lying on a regular grid,  stationary kernel
matrices become  Toeplitz or more generally  multilevel  Toeplitz,  whose
asymptotic spectral distribution is determined by their symbol (or the Fourier
transform of the sequence)
\cite{grenander1984toeplitz,baxter1994norm,tyrtyshnikov1996toeplitz,miranda2000blocktoeplitz,baxter2002preconditioned}. 

Driscoll \& Fornberg \cite{driscoll2002interpolation} pioneered a new form of
asymptotic analysis for kernel methods, in the context of Radial Basis Function (RBF) interpolation. The point set $\X$ is considered fixed,
with arbitrary geometry  (i.e., not lying in general on a regular grid),  and the scaling parameter $\varepsilon$  approaches  $0$.
Driscoll \& Fornberg called this the ``flat limit'', as kernel functions become
flat over the range of $\X$ as $\varepsilon \rightarrow 0$. Very surprisingly, they
showed that for certain kernels the RBF interpolant stays well-defined in the
flat limit, and tends to the Lagrange polynomial  interpolant. 
Later, a series of papers extended their results to the multivariate case \cite{schaback2005multivariate,larsson2005theoretical,schaback2008limit}
and established similar convergence results for various types of radial
functions (see \cite{song2012multivariate,lee2015flatkernel} and references therein).
In particular, \cite{song2012multivariate} showed that for kernels of finite smoothness the limiting interpolant is a spline rather than a polynomial. 

The flat limit is interesting for several reasons. In contrast to other asymptotic analyses, it is deterministic ($\X$ is fixed), and makes very few assumptions on the geometry of the point set. In addition, kernel methods are plagued by the problem of picking a scale parameter \cite{scholkopf2002learning}. One either uses burdensome procedures like cross-validation or maximum likelihood \cite{williams2006gaussian} or sub-optimal but cheap heuristics like the median distance heuristic \cite{garreau2017large}. The flat limit analysis may shed some light on the problem. Finally, the results derived here can be thought of as perturbation results, in the sense that they are formally exact in the limit, but useful approximations when the scale is not too small. 

Despite its importance, little was known until recently about the eigenstructure of kernel matrices in the flat limit.
The difficulty comes from the fact that $\matr{K}_{\varepsilon} = K(0,0) \vect{1}\vect{1}^{\T}  + \mathcal{O}(\varepsilon)$, where $\vect{1} = \left[\begin{smallmatrix} 1 & \cdots & 1\end{smallmatrix}\right]^{\T}$, i.e.,  we are dealing with a singular perturbation problem\footnote{ Seen from
the point of view of the characteristic polynomial, the equation $\det(\matr{K}_\varepsilon - \lambda \matr{I}) = 0$
has a solution of multiplicity $n-1$ at $\varepsilon=0$, but these roots immediately
separate when $\varepsilon>0$.}.

Schaback \cite[Theorem~6]{schaback2005multivariate}, and, more explicitly,
  Wathen \& Zhu \cite{wathen2015eigenvalues} obtained results on the orders of
eigenvalues of kernel matrices for smooth analytic radial basis kernels, based
on the Courant-Fischer minimax principle. A heuristic analysis of the behaviour
of the eigenvalues in the flat limit was also performed in
\cite{fornberg2007runge}.  However, the main terms in the expansion of the
eigenvalues have never been obtained, and the results in
\cite{schaback2005multivariate, wathen2015eigenvalues} apply only to smooth
kernels. In addition, they hold no direct information on the limiting
eigenvectors.

In this paper, we try filling this gap by characterising both the
eigenvalues and eigenvectors of kernel matrices in the flat limit. 
We consider both completely smooth kernels and finitely-smooth kernels. The latter
(Mat\'ern-type kernels) are very popular in spatial statistics.
For establishing asymptotic properties of eigenvalues, we use the expression for the limiting determinants of $\matr{K}_{\varepsilon}$ (obtained only for the smooth case), and Binet-Cauchy formulae.
As a special case, we recover the results of Schaback
\cite{schaback2005multivariate} and Wathen \& Zhu \cite{wathen2015eigenvalues}, but for a wider class of kernels.

\subsection{Overview of the results}
\label{sec:overview}
Some of the results are quite technical, so the goal of this section is to serve as a reader-friendly summary of the contents. 

\subsubsection{Types of kernels}
\label{sec:types-of-kernels}

We begin with some definitions. 

A kernel is called \emph{translation invariant} if
\begin{equation}\label{eq:kernel_trans_general}
K(\vect{x},\vect{y}) = \varphi(\vect{x}-\vect{y})
\end{equation}
 for some function $\varphi$. A kernel is \emph{radial}, or \emph{stationary} if, in addition, we have:
 \begin{equation}\label{eq:kernel_rbf_general}
 K(\vect{x},\vect{y})  = f(\norm{\vect{x}-\vect{y}}_2)
 \end{equation}
 i.e., $\varphi(t) = f(|t|)$ and the value of the kernel depends only on the Euclidean distance between $\vect{x}$ and $\vect{y}$. The function $f$ is an \emph{RBF}. 
 Finally, $K(\vect{x},\vect{y})$ may be positive (semi)  definite, in which case
the kernel matrix is positive (semi) definite \cite{demarchi2009nonstandard} for all sets of distinct points $\X = (\vect{x}_1, \ldots, \vect{x}_n)$, and all $n>0$. 

 All of our results are valid for stationary, positive semi-definite kernels. In addition, some are also valid for translation-invariant kernels, or even general, non-radial kernels. For simplicity, we focus on radial kernels in this introductory section. 

 An important property of a radial kernel is its order of smoothness, which we call $r$ throughout this paper. The definition is at first glance not very enlightening: formally, if the first $p$ odd-order derivatives of the RBF $g$ are zero, and the $(p+1)$-th is nonzero,  then $r=p+1$. To understand the definition, some Fourier analysis is required \cite{stein1999interpolation}, but for the purposes of this article we will just note two consequences. When interpolating using a radial kernel of smoothness $r$, the resulting interpolant is $r-1$ times differentiable. When sampling from a Gaussian process with covariance function of smoothness order $r$, the sampled process is also $r-1$ times differentiable (almost surely). $r$ may equal $\infty$, which is the case we call \emph{infinitely smooth}. If $r$ is finite we talk about a \emph{finitely smooth} kernel. We treat the two cases separately due to importance  of infinitely smooth kernels, and because proofs are simpler in that case. 

 Finally, the points are assumed to lie in some subset of $\RR^d$, and if $d=1$ we call this the \emph{univariate} case, as opposed to the \emph{multivariate} case ($d>1$).

\subsubsection{Univariate results}
\label{sec:univariate-summary}

 In the univariate case, we can give simple closed-form expressions for the eigenvalues and
 eigenvectors of kernel matrices as $\varepsilon \rightarrow 0$. What form these expressions take depends essentially on  the order of smoothness. 
 
We shall contrast two kernels that are at opposite ends of the smoothness spectrum. 
One,
the Gaussian kernel, is infinitely smooth, and is defined as:
\[ K_{\varepsilon}(x,y) = \exp \left( - \varepsilon (x-y)^2 \right).\]
The other has smoothness order 1, and is known as the ``exponential'' kernel
(and is also a Mat\'ern kernel):
\[ K_{\varepsilon}(x,y) = \exp \left( -\varepsilon |x-y| \right).\]
Both kernels are radial and positive definite. However, the small-$\varepsilon$ asymptotics of these two kernels are strikingly different.

In the case of the Gaussian kernel, the eigenvalues go to $0$ extremely
fast, except for the first one, which goes to $n$. Specifically, the first
eigenvalue is $\O(1)$, the second is $\O(\varepsilon^2)$, the third is
$\O(\varepsilon^4)$, etc\footnote{ In fact,  \Cref{thm:1d_smooth_ev} guarantees  the same behaviour for general kernels of sufficient smoothness. }. 
\Cref{fig:eval-gaussian-kernel-1d} shows the
eigenvalues of the Gaussian kernel for a fixed set $\X$ of
randomly-chosen nodes in the unit interval ($n=10$ here). The eigenvalues are
shown as a function of $\varepsilon$, under log-log scaling. As expected from
\Cref{thm:1d_smooth_ev} (see also \cite{schaback2005multivariate,wathen2015eigenvalues}), for each $i$, $\log \lambda_i$ is approximately linear as a
function of $\log \varepsilon$. In addition, the main term in the scaling of $\log
\lambda_i$ in $\varepsilon$ (i.e., the offsets of the various lines) is also given by 
\Cref{thm:1d_smooth_ev}, and the corresponding asymptotic approximations
are plotted in red. They show very good agreement with the exact eigenvalues, up
to $\varepsilon \approx 1$.

\begin{figure}[t!]
  \centering
  \includegraphics[width=7cm]{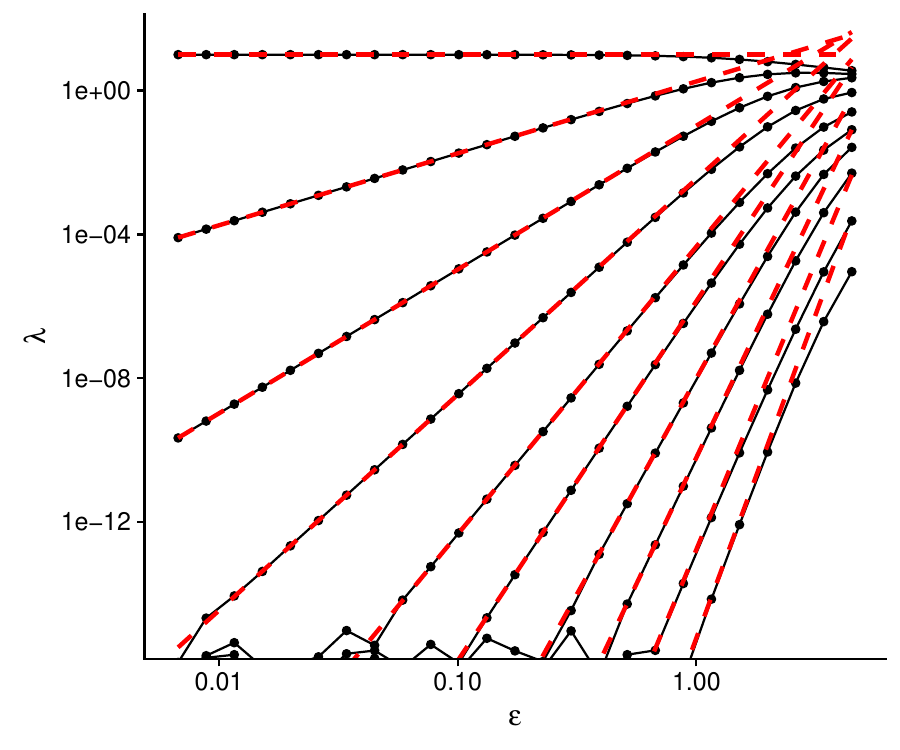}
  \caption{Eigenvalues of the Gaussian kernel ($d=1$). The set of $n=10$
    nodes was drawn uniformly from the unit interval. In black, eigenvalues of the Gaussian kernel, for different values of $\varepsilon$. The dashed red
    curves are our small-$\varepsilon$ expansions. Note that both axes are scaled
    logarithmically. 
    The noise apparent for small
    $\varepsilon$ values  in the low range is due to loss of precision in the
    numerical computations.}
  \label{fig:eval-gaussian-kernel-1d}
\end{figure}

Contrast that behaviour with the one exhibited by the eigenvalues of the
exponential kernel. \Cref{thm:1d_finite_smoothness_ev} describes the
expected behaviour: the top eigenvalue is again $\O(1)$ and goes to $n$, while
all remaining eigenvalues are $\O(\varepsilon)$. \Cref{fig:eval-exp-kernel-1d} is the counterpart of
the previous figure, and shows clearly that all eigenvalues except for the top
one go to 0 at unit rate. 
The  main term in the expansions of eigenvalues  determines again the
 offsets shown in  \Cref{fig:eval-exp-kernel-1d},  which   can be computed
from the eigenvalues of the centered distance matrix  as shown in \Cref{thm:1d_finite_smoothness_ev}.  

\begin{figure}[t!]
  \centering
  \includegraphics[width=7cm]{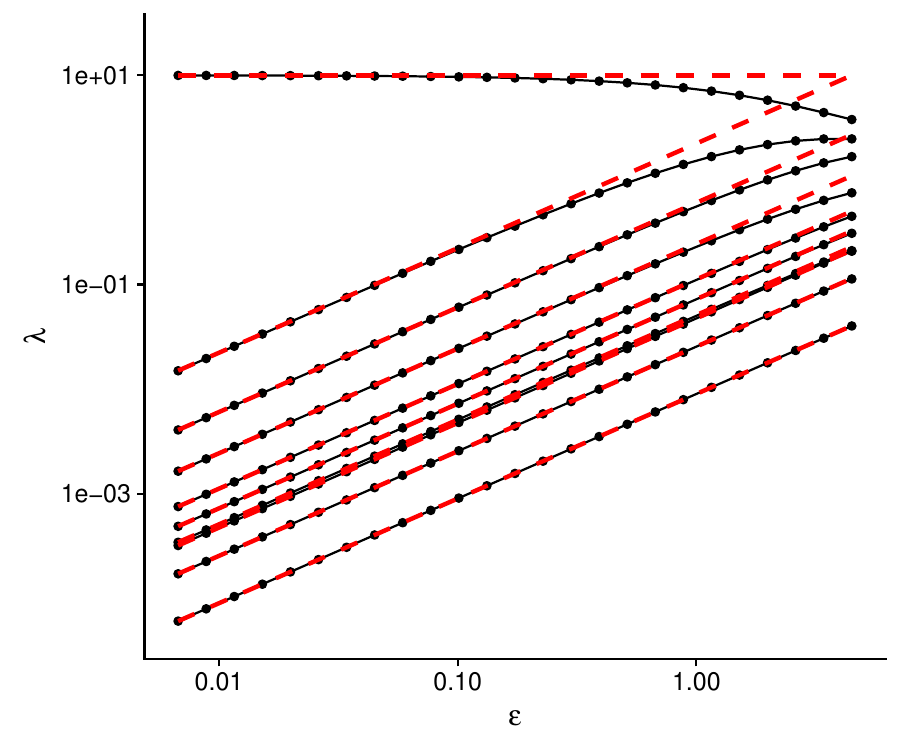}
  \caption{Eigenvalues of the exponential kernel ($d=1$), using the same set of
    points as in \Cref{fig:eval-gaussian-kernel-1d}. The largest eigenvalue
    has a slope of 0 for small $\varepsilon$, the others have unit slope, as
    in \Cref{thm:det_1d_finite_smoothness}.    }
  \label{fig:eval-exp-kernel-1d}
\end{figure}

\newpage

To sum up: except for the top eigenvalue, which behaves in the same way for both
kernels, the rest scale quite differently. More generally, 
\Cref{thm:1d_finite_smoothness_ev} states that for kernels of smoothness order
$r < n$ ($r=1$ for the exponential, $r=\infty$ for the Gaussian), the eigenvalues
are divided into two groups. The first group of eigenvalues is of size $r$, and have orders
$1$, $\varepsilon^2$, $\varepsilon^4$, etc. The second group is of size $n-r$, and all
have the same order, $\varepsilon^{2r-1}$. 

The difference between the two kernels is also reflected in how the eigenvectors
behave. For the Gaussian kernel, the limiting eigenvectors (shown in 
\Cref{fig:evec-gaussian-kernel-1d}) are columns of the $\matr{Q}$ matrix of the QR factorization of the Vandermonde matrix (i.e., the orthogonal polynomials with respect to the discrete uniform measure on $\X$). For instance, the top eigenvector equals the constant vector $\frac{1}{\sqrt{n}}\mathbf{1}$, and the second eigenvector equals $\begin{bmatrix}x_1 & \cdots & x_n \end{bmatrix}^{\T} - \frac{\sum x_i}{n}$ (up to normalisation). Each successive eigenvector depends on the geometry of $\X$ via the moments $m_p(\X) = \sum_{i=1}^n x^p_i$. 
In fact, this result is valid for any positive definite smooth analytic in $\varepsilon$ kernel as shown by \Cref{cor:1d_smooth_ev_ratio}.

\begin{figure}[t!]
  \centering
  \includegraphics[width=10cm]{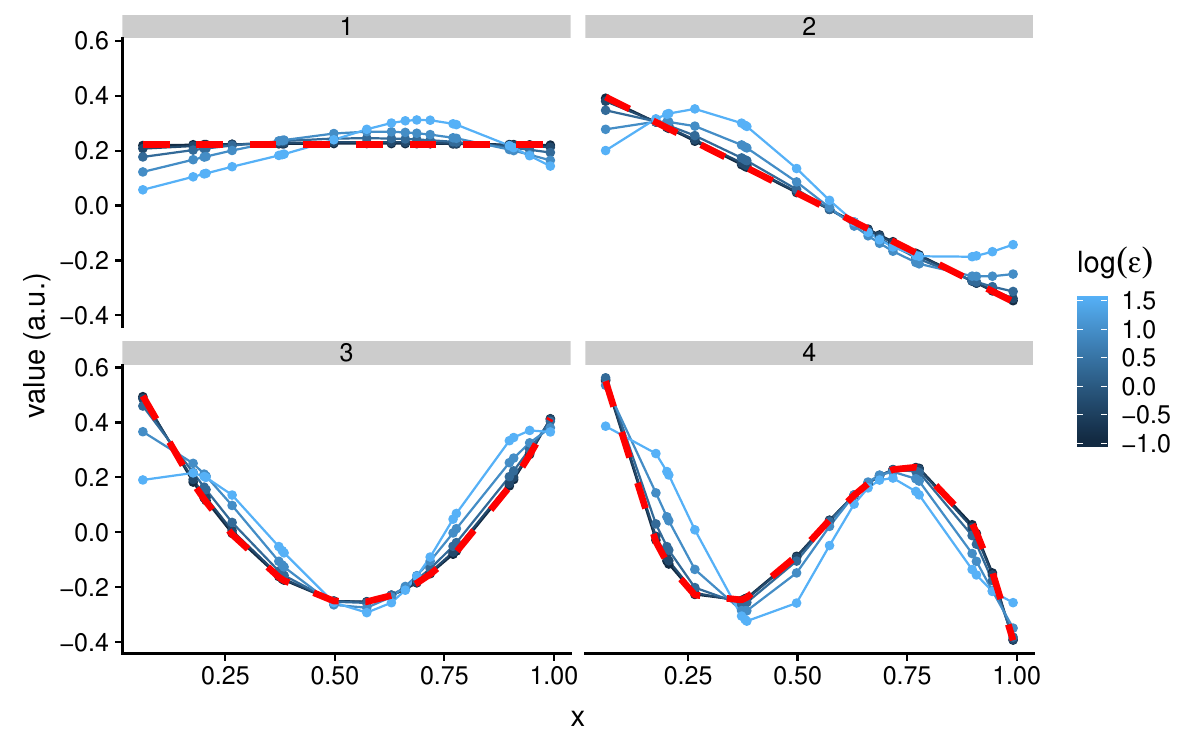}
  \caption{First four eigenvectors of the Gaussian kernel. We used the same set of points as in the previous two figures. In blue, eigenvectors of
    the Gaussian kernel, for different values of $\varepsilon$. The dashed red
    curve shows the theoretical limit as $\varepsilon \rightarrow 0$ (i.e., the first four orthogonal polynomials of the discrete measure) }
  \label{fig:evec-gaussian-kernel-1d}
\end{figure}

In the case of finite smoothness, the  two groups are associated with different groups of eigenvectors. 
The first group of $r$ eigenvectors are again orthogonal polynomials. 
The second group are  splines of order $2r-1$. 
Convergence  of eigenvectors is shown in  \Cref{fig:evec-exp-kernel-1d}.
This general result for eigenvectors is  shown in \Cref{thm:finite_smoothness_eigenvectors}.

\begin{figure}[t!]
  \centering
  \includegraphics[width=10cm]{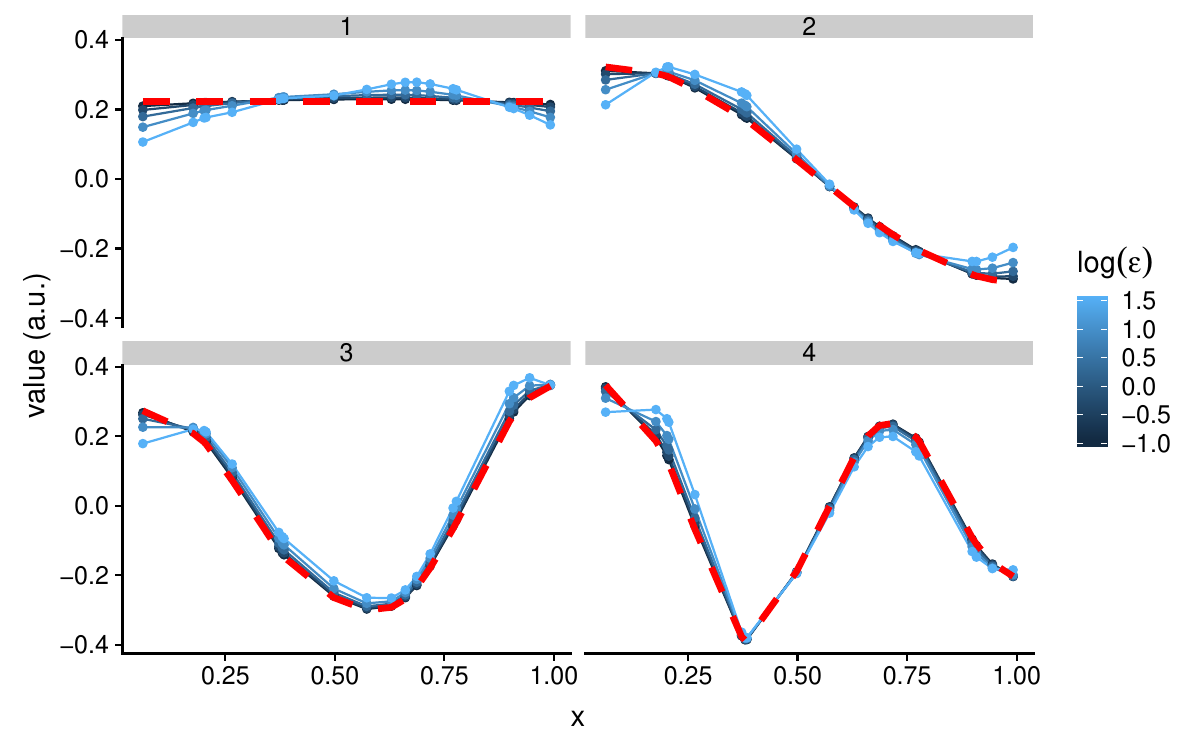}
  \caption{First four eigenvectors of the exponential kernel. The same set of 
    nodes is used as in \Cref{fig:evec-gaussian-kernel-1d}. In blue, the
    eigenvectors of the kernel, for different values of $\varepsilon$. The dashed red
    curve shows the theoretical limit as $\varepsilon \rightarrow 0$. From \Cref{thm:finite_smoothness_eigenvectors},
  these are (1) the vector $\frac{\mathbf{1}}{\sqrt{n}}$ and (2-4) the first three eigenvectors
of $\left(\bI-\frac{\mathbf{1}\mathbf{1}^{\T}}{n}\right)\bD_{(1)}\left(\bI-\frac{\mathbf{1}\mathbf{1}^{\T}}{n}\right)$, where $\bD_{(1)}$ is the Euclidean distance matrix.}
  \label{fig:evec-exp-kernel-1d}
\end{figure}

\begin{figure}[ht!]
  \centering
  \includegraphics[width=7cm]{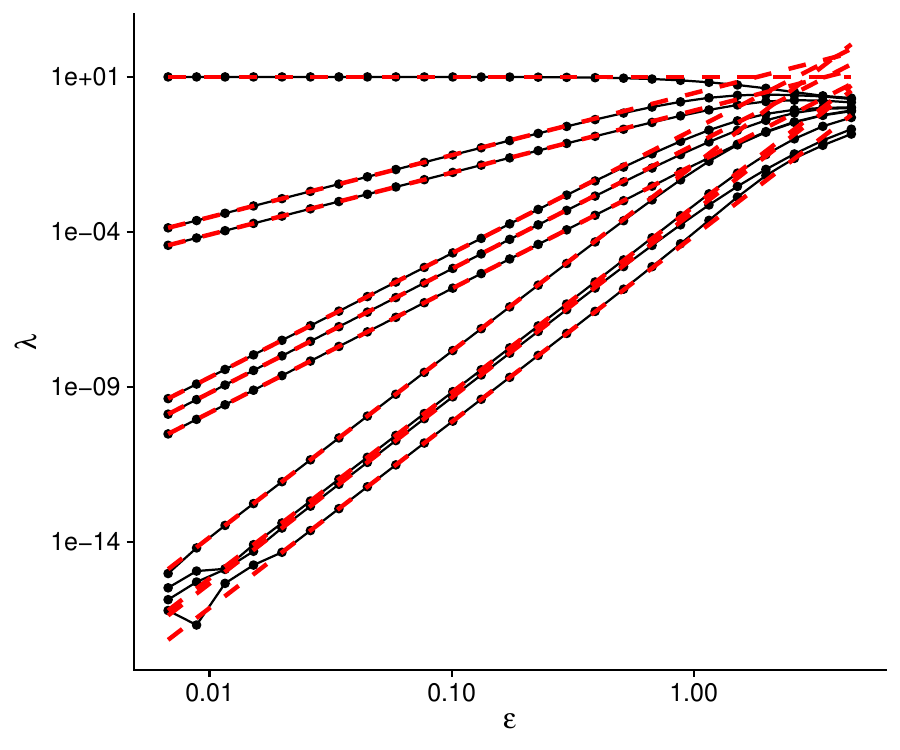}
  \caption{Eigenvalues of the Gaussian kernel in the multivariate case ($d=2$).
    The set of $n=10$    nodes was drawn uniformly from the unit square. In black, eigenvalues of the Gaussian kernel, for different values of $\varepsilon$. The dashed red
    curves are our small-$\varepsilon$ expansions. Eigenvalues appear in groups of
    different orders, recognizable here as having the same slopes as $\varepsilon
    \rightarrow 0$. A single eigenvalue is of order $\O(1)$. Two eigenvalues are
  of order $\O(\varepsilon^2)$, as many as there are monomials of degree 1 in two
  dimensions. Three eigenvalues are
  of order $\O(\varepsilon^4)$, as many as there are monomials of degree 2 in two
  dimensions, etc. }
  \label{fig:eval-gaussian-kernel-2d}
\end{figure}

\begin{figure}[ht!]
  \centering
  \includegraphics[width=7cm]{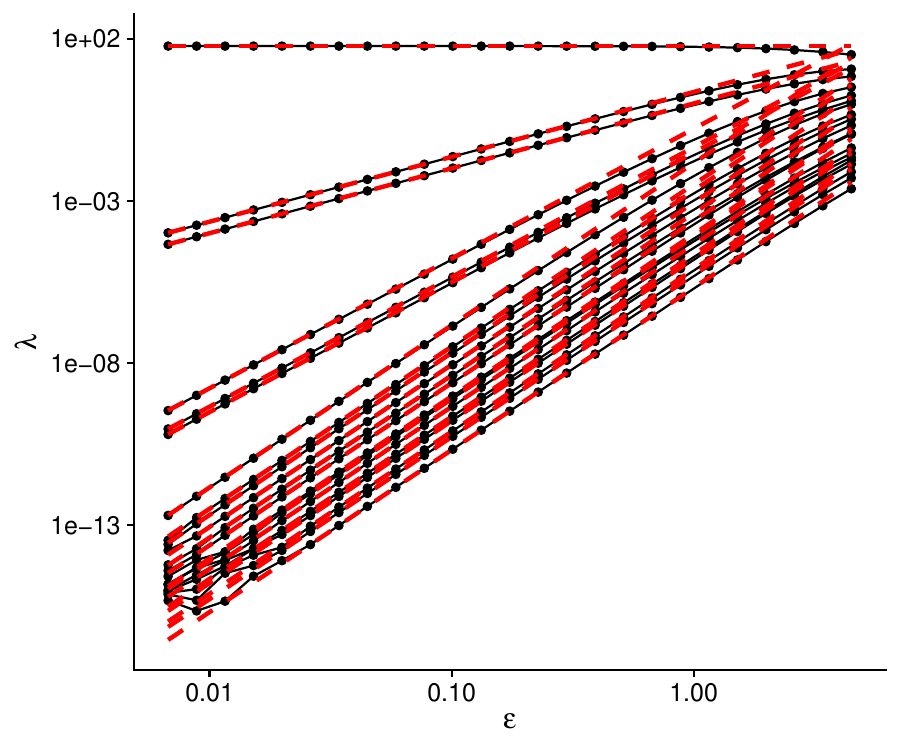}
  \caption{Eigenvalues of a finitely smooth kernel in $d=2$. Here we picked a
    kernel with smoothness index $r=3$ (the exponential kernel has order $r=1$). 
    The set of $n=20$ nodes was drawn uniformly from the unit square. 
    The first $r=3$ groups of eigenvalues have the same behaviour as in the
    Gaussian kernel: one is $\O(1)$, two are $\O(\varepsilon^2)$, three
    are $\O(\varepsilon^4)$. All the rest are $\O(\varepsilon^5)$. }
  \label{fig:eval-fsmooth-kernel-2d}
\end{figure}

\subsubsection{The multivariate case}
\label{sec:multivariate-case}

The general, multivariate case requires more care. Polynomials continue to play a central role in the flat limit, but when $d>1$ they appear naturally in groups of equal degree. For instance, in $d=2$, we may write  $\vect{x}=\left[\begin{smallmatrix}y & z \end{smallmatrix}\right]^{\T}$ and the first few monomials are as follows:
\begin{itemize}
\item Degree 0: $1$
\item Degree 1: $y$, $z$
\item Degree 2: $y^2$, $yz$, $z^2$
\item Degree 3: $y^3$, $y^2z$, $yz^2$, $z^3$
\end{itemize}
etc. Note that there is one monomial of degree $0$, two degree-$1$ monomials, three degree-$2$ monomials, and so on. If $d=1$ there is a single monomial in each group, 

\noindent and here lies the essence of the difference between the univariate and multivariate cases.

In infinitely smooth kernels like the Gaussian kernel, as shown in \cite{schaback2005multivariate,wathen2015eigenvalues}, there are as many
eigenvalues of order $\O(\varepsilon^{2k})$ as there are monomials of degree $k$ in
dimension $d$: for instance, there are 4 monomials of degree 3 in dimension 2,
and so 4 eigenvalues of order $\O(\varepsilon^{8})$. 
An example is shown in 
\Cref{fig:eval-gaussian-kernel-2d}. In finitely smooth kernels
like the exponential kernel, there are $r$ successive groups of eigenvalues with
order $\O(1)$, $\O(\varepsilon^2)$, ..., up to $\O(\varepsilon^{2r-2})$. Following that, all remaining eigenvalues have order $\O(\varepsilon^{2r-1})$, just like in the one-dimensional case. 

We show in \Cref{thm:nd_smooth_ev}, that the main terms of these
eigenvalues in each group can be computed from the QR factorization of the
Vandermonde matrix and a Schur complement associated  with the Wronskian matrix
of the kernel. In the finite smoothness case (see \Cref{fig:eval-fsmooth-kernel-2d}), the same expansions are valid until the smoothness order, and the last group of eigenvalues is given by the eigenvalues of the projected distance matrix, as in the one-dimensional case.

Finally, in the  multivariate case, \Cref{thm:nd_smooth_ev} characterises the
eigenprojectors.
In a nutshell, the invariant subspaces associated with each group of
eigenvalues are spanned by orthogonal polynomials of a certain order. 
The eigenvectors are the subject of a conjecture given in section \ref{sec:eigenvectors}, which we believe to be quite solid.

\subsection{Overview of tools used in the paper}
For finding the orders of the eigenvalues, as in \cite{schaback2005multivariate,wathen2015eigenvalues}, we use the Courant-Fischer minimax principle (see  \Cref{lem:courant_fischer}).
However, unlike \cite{schaback2005multivariate,wathen2015eigenvalues}, we do not use directly  the results of Micchelli \cite{micchelli1986interpolation}, but rather rotate the kernel matrices using the $Q$ factor in the QR factorization of the Vandermonde matrix, and find the expansion of rotated matrices from the Taylor expansion of the kernel.

The key results are the expansions for the determinants of $\matr{K}_{\varepsilon}$, which use the expansions of rotated kernel matrices.
Our results on determinants (Theorems \ref{thm:det_1d_smooth}, \ref{thm:det_1d_finite_smoothness}, \ref{thm:det_nd_smooth}, and \ref{thm:det_nd_finite_smoothness}) generalize those of Lee \& Micchelli \cite{lee2013collocation}.
The next key observation is that principal submatrices of $\matr{K}_{\varepsilon}$ are also kernel matrices, hence the results on determinants imply the results on expansions of elementary symmetric polynomials of eigenvalues (via the correspondence between elementary symmetric polynomials, see \Cref{lem:esp-sum-minors}, and the Binet-Cauchy formula).
Finally, the main terms of the eigenvalues can be retrieved from the main terms of the elementary symmetric polynomials, as shown in \Cref{lem:orders_lambda_to_esp}.
An important tool for the multivariate and finite smoothness case is \Cref{lem:esp_coef_perturbation} on low-rank perturbation of elementary symmetric polynomials that we could not find elsewhere in the literature

To study the properties of the eigenvectors, we 
use analyticity of  the eigenvalues and eigenprojectors for Hermitian analytic matrix-valued functions \cite{kato1995perturbation}. 
By using an extension of the Courant-Fischer principle (see \Cref{lem:courant_fischer_ev}),
we can fully characterise the limiting eigenvectors in the univariate case, and obtain the limiting invariant subspaces for the groups of the limiting eigenvectors in the multivariate case.
Moreover, by using the perturbation expansions from  \cite{kato1995perturbation},
we can find the last individual eigenvectors in the finitely smooth case.

We note that the multivariate case requires a number of technical assumptions on the arrangement of points, which are typical for multivariate polynomial interpolation.
For the results on determinants, no assumptions are required.
However, for getting the results on the eigenvalues or eigenvectors at a certain order of $\varepsilon$, we need a relaxed version to the well-known unisolvency condition \cite{schaback2005multivariate,fornberg2004observations}, namely  the multivariate Vandermonde  matrix up to the corresponding degree to be full rank.

\subsection{Organisation of the paper}
\label{sec:organisation}
In an attempt to make the paper reader-friendly, it is organized as follows.
In \Cref{sec:1dkernels} we recall the main terminology for  1D kernels.
In \Cref{sec:known_results} we gather well-known (or not so well-known) results on eigenvalues, determinants, elementary symmetric polynomials and their perturbations, which are key tools used in the paper.
\Cref{sec:1d_results} contains the main results on determinants, eigenvalues, and eigenvectors  in the univariate ($d=1$) case.
While these results are special cases of the multivariate case ($d>1$), the latter  is burdened with heavier notation due to the
complexity of dealing with multivariate polynomials.
 To get a gist of the results
and techniques, the reader is advised to  first consult the case $d=1$. 
In \Cref{sec:nd_notations}, we introduce all the needed notation to handle the multivariate case.
\Cref{sec:nd_results} contains the main results of the paper on determinants, eigenvalues, and groups of eigenvectors in the multivariate case.
Thus \Crefrange{sec:1dkernels}{sec:nd_results}  are quite self-contained and contain most of the results in the paper, except the result on precise locations of the last group of eigenvectors in the finite smoothness case. 
In \Cref{sec:kato}, we provide a brief summary on analytic perturbation theory, needed only for proving the stronger result on eigenvectors for finitely smooth kernels in \Cref{sec:eigenvectors}. 

\section{Background and main notation}\label{sec:1dkernels}
This section contains main definitions and examples of kernels, mostly in the 1D case (with multivariate extensions given in \Cref{sec:nd_notations}).
We assume that  the kernel $K: \Omega \times \Omega \to \RR$ is in the class $\calC^{(\ell,\ell)} (\Omega)$, $\Omega = (-a; a)$, i.e. all the partial derivatives $\frac{\partial^2}{\partial x^i \partial y^j} K $ exist and are continuous for $0 \le i,j \le \ell$ on $\Omega \times \Omega$.

We will often need the following short-hand notation for partial derivatives
\[
K^{(i,j)} \eqdef \frac{\partial^{i+j}}{\partial x^i \partial y^j} K,
\]
and we introduce the so-called Wronskian matrix
\begin{equation}\label{eq:Wronskian}
\Wronsk{k} \eqdef
\begin{bmatrix}
\frac{K^{(0,0)}(0,0)}{0!0!}  & \cdots & \frac{K^{(0,k)}(0,0)}{0!k!} \\
\vdots &  & \vdots \\
\frac{K^{(k,0)}(0,0)}{k!0!}  & \cdots & \frac{K^{(k,k)}(0,0)}{k!k!} 
\end{bmatrix}.
\end{equation}
\subsection{Translation invariant  and radial kernels}
Let us consider an important example of a translation  invariant \eqref{eq:kernel_trans_general}   kernel, which in the univariate case becomes 
$K(x,y) = \varphi(x-y)$.
We assume that $\varphi\in \calC^{2r}(-2a,2a)$; hence, $\varphi$ has a Taylor expansion around $0$
\[
\varphi(t) = \alpha_0 + \alpha_1 t+  \alpha_2 t^2  + \cdots +\alpha_{2r} t^{2r}  + o(t^{2r}), \quad \text{where} \; \alpha_k \eqdef \frac{\varphi^{(k)}(0)}{k!}.
\]
Therefore $K \in \calC^{(r,r)}(\Omega\times \Omega)$ and its derivatives at $0$ are
\begin{equation}\label{eq:wronskian_rbf}
\frac{K^{(i,j)}(0,0)}{i!j!} = \frac{(-1)^{j} \varphi^{(i+j)} (0)}{i!j!} = (-1)^{j} {i+j \choose j} \alpha_{i+j},
\end{equation}
i.e.,  the  Wronskian matrix has the form:
\[
\Wronsk{k} =
\begin{bmatrix}
\alpha_0 & -\alpha_1 & \alpha_{2} & -\alpha_3   & \cdots \\
\alpha_1 & -2\alpha_2 & 3\alpha_3 & -4\alpha_4  &  \cdots  \\
\alpha_2 & -3\alpha_3 & 6 \alpha_4 & -10\alpha_5  &\cdots  \\
\alpha_3 & -4\alpha_4 & 10\alpha_5 & -20 \alpha_6  & \cdots  \\
\vdots  & \vdots & \vdots & \vdots   & \ddots \\
\end{bmatrix}.
\]
A special case of the translational kernels are  smooth and radial, which will be considered in the next subsection.
The simplest example is the Gaussian kernel with
\[
\varphi(t) = \exp(-t^2) = 1 - t^2 + \frac{t^4}{2!}  - \frac{t^6}{3!} + \cdots ,
\]
i.e. for all  integer $\nu \ge 0$
\[
\alpha_{2\nu} = \frac{(-1)^\nu}{\nu!}, \quad \alpha_{2\nu+1} = 0.
\]
In this case, the Wronskian matrix becomes
\[
\Wronsk{k} =
\begin{bmatrix}
1 & 0 & -1& 0   & \cdots \\
0 & 2 & 0 & -2  &  \cdots  \\
-1 &0 & 3 & 0  &\cdots  \\
0 & -2 & 0 & -\frac{10}{3}  & \cdots  \\
\vdots  & \vdots & \vdots & \vdots   & \ddots \\
\end{bmatrix}.
\]
An important subclass consists of radial kernels \eqref{eq:kernel_rbf_general}, i.e., $K(x,y) = f(|x-y|)$ in the univariate case.
We put   the following  assumptions on~$f$:
\begin{itemize}
\item $f\in \calC^{2r}(-2a,2a)$;
\item the highest derivative of odd order is not zero, i.e.,  $f^{(2r-1)}(0) \neq 0$;
\item the lower derivatives with odd order vanish, i.e., $f^{(2\ell-1)}(0) = 0$ for $\ell < r$.
\end{itemize}
If the function $f(t)$ satisfies these assumptions, then $r$ is called the \emph{order of smoothness} of $K$.
Note that $f$ admits a Taylor expansion
\begin{equation}
\label{eq:finite-smoothness-expansion}
f(t) = f_0  + f_2 t^2  + \cdots + f_{2r-2} t^{2r-2} + t^{2r-1} ( f_{2r-1} + \O(t)),
\end{equation}
where $f_k = f^{(k)}(0)/k!$ is a shorthand notation for the scaled derivative at $0$.
For example, for the exponential kernel $\exp(-|x-y|)$, we have $f_{0}=1$ and $f_{1}=-1$, so the smoothness order is $r=1$. For the $\mathcal{C}^{2}$ Mat\'ern kernel $(1+|x-y|)\exp(-|x-y|)$, we have $f_{0}=1,f_{1}=0,f_{2}=-1/2,f_{3}=1/3$, so the smoothness order is $r=2$. 

\subsection{Distance matrices, Vandermonde matrices,  and their properties}
In the general multivariate case, we define ${\matr{D}_{(k)}}$ to be  the $k$-th Hadamard power of the Euclidean distance matrix $\bD_{(1)}$, i.e.,
\[
\bD_{(k)}= \left[ \left\| \vect{x}_{i} - \vect{x}_{j} \right\|_2^{k} \right]^{n,n}_{i,j=1}.
\]
Next, we focus on the univariate case (the multivariate case will be discussed later in \Cref{sec:nd_notations})
and denote by $\matr{V}_{\le k}$ the univariate Vandermonde matrix up to degree $k$ 
\begin{equation}\label{eq:Vandermonde1D}
\matr{V}_{\le k} = \begin{bmatrix}
1 & x_1 & \cdots & x^{k}_1\\
\vdots &\vdots  & &\vdots  \\
1 & x_n & \cdots & x^{k}_n 
\end{bmatrix},
\end{equation}
which has rank $\min(n, k+1)$ if the nodes are distinct.
In particular, the matrix $\matr{V}_{\le n-1}$ is square and invertible for distinct nodes.

For  even $k$, the Hadamard powers  ${\matr{D}_{(k)}}$  of the distance matrix can be expressed via the columns of the Vandermonde matrix using the binomial expansion
\begin{equation}\label{eq:dist_binomial_expansion}
{\matr{D}_{(2\ell)}}  = \sum\limits_{j=0}^{2\ell} (-1)^{j} {2\ell \choose j} \vect{v}_{j} \vect{v}_{2\ell-j}^{\T},
\end{equation}
where $\vect{v}_k \eqdef \begin{bmatrix}x^{k}_1 & \cdots& x^{k}_n \end{bmatrix}^{\T}$ are columns of Vandermonde matrices.
Therefore, $\rank \matr{D}_{(2\ell)} = \min(2\ell+1,n)$ if all points $x_k$ are distinct.

On the other hand, for $k$ odd, the matrices $\bD_{(k)}$ exhibit an entirely different set of properties. Of most interest here is conditional positive-definitess, which guarantees that the distance matrices are positive definite when projected onto a certain subspace. The following result appears e.g., in \cite[Chapter 8]{fasshauer2007meshfree} but follows directly from an earlier paper by Micchelli \cite{micchelli1986interpolation}. 
\begin{lemma}\label{lem:cpd}
For a distinct node set $\X \subset \RR$ of size $n>r \ge 1$, we let $\matr{B}$ be a full column rank matrix such that $\matr{B}^{\T}\matr{V}_{\le r-1} = 0$. Then the matrix $(-1)^r \matr{B}^{\T} \matr{D}_{(2r-1)} \matr{B}$ is positive definite.
\end{lemma}
For instance, if $r=1$, we may pick any basis $\matr{B}$ orthogonal to the  vector $\mathbf{1}$, and the lemma implies that $(-\matr{B}^{\T} \matr{D}_{(1)} \matr{B})$ has $n-1$ positive eigenvalues. We note for future reference that the result generalises almost unchanged to the multivariate case. 

\subsection{Scaling and expansions of kernel matrices}
In this subsection, we consider the general multivariate case.
Given a general  kernel $K$, we define its scaled version as
\begin{equation}\label{eq:general_kernel_scaled}
K_{\varepsilon}(\vect{x},\vect{y}) = K(\varepsilon\vect{x},\varepsilon\vect{y}),
\end{equation}
while for the specific case of radial kernels we use the following form 
\begin{equation}\label{eq:rbf_kernel_scaled}
K_{\varepsilon}(\vect{x}, \vect{y})= f(\varepsilon\norm{\vect{x}-\vect{y}}_2).
\end{equation}
Note that the definitions \eqref{eq:general_kernel_scaled} and
\eqref{eq:rbf_kernel_scaled} coincide for nonnegative $\varepsilon$, but differ
if we formally take other values (e.g., complex) of $\varepsilon$.
Depending on context, we will use one or the other of the definitions later on, especially when we talk about analyticity of kernel matrices.

Using  \eqref{eq:finite-smoothness-expansion} and \eqref{eq:rbf_kernel_scaled}, we may write the scaled kernel matrix \eqref{eq:scaled_kernel}, as
\begin{equation}\label{eq:finite-smoothness-kernel-expansion}
\bK_\varepsilon = f_{0}\bD_{(0)} +\varepsilon^2   f_{2} \bD_{(2)} + \cdots +\varepsilon^{2r-2} f_{2r-2}   \bD_{(2r-2)}+ \varepsilon^{2r-1}f_{2r-1}   \bD_{(2r-1)}  + \O(\varepsilon^{2r}).
\end{equation}
In the univariate case, \eqref{eq:dist_binomial_expansion} gives a way to rewrite the expansion  \eqref{eq:finite-smoothness-kernel-expansion} as
\begin{equation}\label{eq:fs_kernel_wronskian_diag}
\matr{K}_{\varepsilon} =  \sum_{\ell=0}^{r-1} \varepsilon^{2j} \underbrace{\matr{V}_{\le 2\ell} \matr{W}_{\diagup 2\ell} \matr{V}^{\T}_{\le 2\ell}}_{f_{2\ell}\matr{D}_{(2\ell)}} +
\varepsilon^{2r-1} f_{2r-1} \matr{D}_{(2r-1)}+ 
 \O(\varepsilon^{2r}),
\end{equation}
where $\matr{W}_{\diagup s} \in \RR^{(s+1) \times (s+1)}$ has nonzero elements only on its antidiagonal:
\begin{equation}\label{eq:wdiag_def_1d}
\Big(\matr{W}_{\diagup s}\Big)_{i+1,s-i+1} \eqdef f_{s} (-1)^i {s \choose i}.
\end{equation}
For example,
\[
\matr{W}_{\diagup 2} =f_2 
\begin{bmatrix}
  0 & 0 & 1 \\
  0 & -2 & 0 \\
  1 & 0 & 0
\end{bmatrix}.
\]
In fact,  from \eqref{eq:wronskian_rbf}, non-zero elements of  $\matr{W}_{\diagup s}$ are scaled  derivatives of the kernel:
\[
\Big(\matr{W}_{\diagup s}\Big)_{i+1,s-i+1} =\frac{K^{(i,s-i)}(0,0)}{i!(s-i)!}; 
\]
this justifies the notation $\matr{W}_{\diagup s}$ (i.e.,  an antidiagonal of the Wronskian matrix).

\section{Determinants and elementary symmetric polynomials}\label{sec:known_results}
In this paper, we will heavily use the  elementary symmetric polynomials of eigenvalues,
and we collect in this section some useful (more or less known) facts about them.

\subsection{Eigenvalues, principal minors and elementary symmetric polynomials}
The $k$-th elementary symmetric polynomial of $\lambda_1, \ldots, \lambda_{n}$ is defined as:
\begin{equation}
\label{eq:esp}
e_k(\lambda_1,\ldots,\lambda_n) = \sum_{\substack{\#\calY = k\\ \calY \subset \{1,\ldots,n\}}} \, \prod_{i \in \calY} \lambda_i,
\end{equation}
i.e., the sum is running over all possible  subsets of $\{1,\ldots,n\}$ of size $k$.
In particular, $e_1(\lambda_1,\ldots,\lambda_n)=\sum_{i=1}^{n} \lambda_i$, $e_2(\lambda_1,\ldots,\lambda_n) =
\sum_{i<j} \lambda_i\lambda_j$, and $e_n(\lambda_1,\ldots,\lambda_n) = \prod_{i=1}^{n} \lambda_i$.

Next, we consider the elementary symmetric polynomials   applied to eigenvalues of matrices, and define (with some abuse of notation):
\[
e_k(\bA) \eqdef e_k(\lambda_1(\bA),\ldots,\lambda_n(\bA)).
\]
Then the first and the last polynomials are the trace and determinant of $\bA$:
\[
e_1(\bA)= \Tr \bA,\quad
e_n(\bA) = \det \bA.
\]
This fact is a special case of a more general result on sums of principal minors.
\begin{theorem}[{\cite[Theorem 1.2.12]{horn1990matrix}}]\label{lem:esp-sum-minors}
\begin{equation}
\label{eq:esp-sum-minors}
e_k(\bA) = \sum_{\substack{\calY \subset \{1,\ldots,n\} \\ \# \calY  = k}} \det (\bA_{\calY,\calY})
\end{equation}
where $\bA_{\calY,\calY}$ is a submatrix of $\bA$ with  rows and columns indexed by $\calY$, i.e.
the sum runs over all principal minors of size $k \times k$.
\end{theorem}

\begin{remark}\label{rem:esp_char_poly}
The scaled symmetric polynomials $(-1)^{k} e_k(\bA)$ are the coefficients of the characteristic polynomial $\det(\lambda \matr{I} -\matr{A})$ at the coefficient $\lambda^{n-k}$. 
\end{remark}

\subsection{Orders of elementary symmetric polynomials}
Next, we assume that  $\lambda_1(\varepsilon), \ldots, \lambda_n(\varepsilon)$ are functions of some small parameter $\varepsilon$ and we are interested in the orders  of the corresponding elementary symmetric polynomials 
\[
 e_s(\varepsilon) \eqdef e_s(\lambda_1(\varepsilon),\ldots,\lambda_n(\varepsilon)),\quad 1 \le s \le n,
\]
as $\varepsilon \to 0$.
The following obvious observation will be important.
\begin{lemma}\label{lem:orders_lambda_to_esp}
Assume that 
 \[
 \lambda_1(\varepsilon) =  \O(\varepsilon^{L_1}),\; \lambda_2(\varepsilon) =  \O(\varepsilon^{L_2}),\; \ldots,\;\lambda_n(\varepsilon) =  \O(\varepsilon^{L_n}),
 \]
as $\varepsilon\to 0$ and $0\le L_1 \le \cdots \le L_n$ are some integers.
Then it holds that
 \[
 e_1(\varepsilon) =  \O(\varepsilon^{L_1}),\;e_2(\varepsilon) =  \O(\varepsilon^{L_1+L_2}),\;\ldots,\;e_n(\varepsilon) =  \O(\varepsilon^{L_1+\cdots+ L_n}).
 \]
\end{lemma}
\begin{proof}
The proof follows from the definition of $e_s$, the fact that the product $f(\varepsilon) = \O(\varepsilon^a)$ and $g(\varepsilon) = \O(\varepsilon^b)$ is of the order$f(\varepsilon)g(\varepsilon) = \O(\varepsilon^{a+b})$.
\end{proof}
We will need a refinement of \Cref{lem:orders_lambda_to_esp} concerning the main terms of such functions. 
We distinguish two situations: when the orders of $\lambda_k$ are separated, and when they form groups. For example, if $\lambda_1(\varepsilon) = a + \O(\varepsilon)$, $\lambda_2(\varepsilon) = \varepsilon( b + \O(\varepsilon))$, $\lambda_3(\varepsilon) = \varepsilon^2( c + \O(\varepsilon))$, then  the main terms of elementary symmetric  polynomials are products of the main terms of $\lambda_k$:
\[
e_1(\varepsilon) = a + \O(\varepsilon), \quad e_2(\varepsilon) = \varepsilon(ab + \O(\varepsilon)), \quad e_3(\varepsilon) = \varepsilon^{3}(abc + \O(\varepsilon)).
\]
On the other hand, in  the case $\lambda_1(\varepsilon) = a + \O(\varepsilon)$, $\lambda_{k+1}(\varepsilon) =  \varepsilon(b_k + \O(\varepsilon))$, $k \in \{1,2,3\}$,
the behaviour of the main terms of elementary symmetric polynomials is different:
\begin{align*}
e_1(\varepsilon) &= a +\O(\varepsilon), \quad e_2(\varepsilon) = \varepsilon(a(b_1+b_2+b_3) + \O(\varepsilon)), \\
e_3(\varepsilon)  &= \varepsilon^2(a(b_1b_2+b_1b_3+b_2b_3) + \O(\varepsilon)),\quad
e_4(\varepsilon) = \varepsilon^3(ab_1b_2b_3 +  \O(\varepsilon)).
\end{align*}
The following two lemmas generalize these observations to arbitrary orders. 
\begin{lemma}\label{lem:main_terms_esp_separated}
Suppose that $\lambda_1(\varepsilon), \ldots, \lambda_n(\varepsilon)$ have the form 
\begin{equation}\label{eq:lambda_orders_differentiable}
\lambda_1(\varepsilon) =  \varepsilon^{L_1} (\widetilde{\lambda}_1 + \O(\varepsilon)), 
\ldots\ ,
\lambda_n(\varepsilon) =  \varepsilon^{L_n} (\widetilde{\lambda}_n + \O(\varepsilon)),
\end{equation}
for some integers  $0\le L_1 \le \cdots \le L_r$. Then
\begin{enumerate}
\item The elementary symmetric polynomials have the form
\begin{equation}\label{eq:esp_orders_differentiable}
\begin{split}
e_1(\varepsilon) &=  \varepsilon^{L_1} (\widetilde{e}_1 + \O(\varepsilon)) \\
e_2(\varepsilon) &=  \varepsilon^{L_1+L_2} (\widetilde{e}_2 + \O(\varepsilon)) \\
\vdots \\
e_n(\varepsilon) &=  \varepsilon^{L_1+ \cdots + L_n} (\widetilde{e}_n + \O(\varepsilon)).
\end{split}
\end{equation}
\item If either $s=n$ or $L_{s}< L_{s+1}$, the main term $\widetilde{e}_s$ can be expressed as
\[
\widetilde{e}_s = \widetilde{\lambda}_1 \cdots \widetilde{\lambda}_s.
\]
In particular, if $s>1$ and $\widetilde{e}_{s-1} \neq 0$, then the main term $\widetilde{\lambda}_s$ can be found as
\[
\widetilde{\lambda}_s = \frac{\widetilde{e}_s}{\widetilde{e}_{s-1}}.
\]
\end{enumerate}
\end{lemma}

We  also need a generalization of  \Cref{lem:main_terms_esp_separated} to the case of a group of equal $L_k$.
\begin{lemma}\label{lem:main_terms_esp_group}
Let $\lambda_k(\varepsilon)$  and $L_k$ be as in \Cref{lem:main_terms_esp_separated} (i.e., $\lambda_k(\varepsilon)$ have the form \eqref{eq:lambda_orders_differentiable} and the corresponding $e_k(\varepsilon)$ have the form \eqref{eq:esp_orders_differentiable}), and define $L_0 = -1$, $L_{n+1} = +\infty$, for an easier treatment of border cases. 

If  for $1 \le m \le n-s$, there is a separated group of $m$ functions
\[
\lambda_{s+1}(\varepsilon), \ldots, \lambda_{s+m} (\varepsilon)
\]
of repeating degree, i.e.,
\begin{equation}\label{eq:group_degrees}
L_s <L_{s+1} = \cdots =  L_{s+m} < L_{s+m+1},
\end{equation}
then the main terms $\widetilde{e}_{s+k}$, $1\le k \le m$ in \eqref{eq:esp_orders_differentiable},
are connected with ESPs of the main terms $\widetilde{\lambda}_{s+k}$, $1\le k \le m$ as follows:
\[
\widetilde{e}_{s+k} = 
\begin{cases}
e_k(\widetilde{\lambda}_{s+1},\ldots,\widetilde{\lambda}_{s+m}), & s =0, \\
\widetilde{\lambda}_1 \cdots \widetilde{\lambda}_s e_k(\widetilde{\lambda}_{s+1},\ldots,\widetilde{\lambda}_{s+m}) = \widetilde{e}_s  e_k(\widetilde{\lambda}_{s+1},\ldots,\widetilde{\lambda}_{s+m}) , & s > 0.
\end{cases}
\]
In particular, if $s > 1$ and $\widetilde{e}_{s} \neq 0$, the ESPs for the main terms are equal to
\[
e_k(\widetilde{\lambda}_{s+1},\ldots,\widetilde{\lambda}_{s+m})  = \frac{\widetilde{e}_{s+k}}{\widetilde{e}_{s}},
\]
hence $\widetilde{\lambda}_{s+1},\ldots,\widetilde{\lambda}_{s+m}$ are the roots of the polynomial (see \Cref{rem:esp_char_poly})
\[
q(\lambda) = \widetilde{e}_{s} \lambda^m - \widetilde{e}_{s+1} \lambda^{m-1} + \widetilde{e}_{s+2} \lambda^{m-2} + \cdots + (-1)^m\widetilde{e}_{s+m}. 
\]
\end{lemma}
Proofs of \Crefrange{lem:main_terms_esp_separated}{lem:main_terms_esp_group} are contained in \Cref{sec:proofs_esp}.
Note that for  analytic $\lambda_k(\varepsilon)$, \Cref {lem:main_terms_esp_separated} and
\Cref{lem:main_terms_esp_group} follow from the Newton-Puiseux theorem (see, for example, \cite[section 2.1]{sosa2016first}),
 but we prefer to keep a  more general formulation in this paper.

\begin{remark}
Assumptions in \Crefrange{lem:main_terms_esp_separated}{lem:main_terms_esp_group} can be relaxed (when expansions \eqref{eq:lambda_orders_differentiable} are valid up to a certain order), but we  keep the current statement for simplicity.
\end{remark}

\subsection{Orders of eigenvalues and eigenvectors}
{First, we recall a corollary of the Courant-Fischer ``min-max'' principle giving} a  bound on smallest eigenvalues.
\begin{theorem}[{\cite[Theorem 4.3.21]{horn1990matrix}}]\label{lem:courant_fischer}
Let $\matr{A} \in \RR^{n\times n}$ be symmetric, and its eigenvalues arranged in non-increasing order $\lambda_1 \ge \lambda_2 \ge \cdots \ge \lambda_n$.
If there exist an $m$-dimensional subspace $\mathcal{L} \subset \RR^{n}$ and a constant $c_1$ such that
\[
\frac{\vect{u}^{\T} \matr{A} \vect{u}}{\vect{u}^{\T} \vect{u}} \le c_1
\]
for all $\vect{u} \in\mathcal{L} \setminus \{ \vect{0}\}$, then the smallest  $m$ eigenvalues are bounded as
\[
\lambda_n \le \lambda_{n-1} \le \cdots \le \lambda_{n-m+1} \le c_1.
\]
\end{theorem}

For determining the orders of eigenvalues, we will need the following  corollary of \Cref{lem:courant_fischer} (see the proof of \cite[Theorem 8]{wathen2015eigenvalues}). We provide a short proof for completeness.


\begin{lemma}\label{lem:courant_fischer_order}
Suppose that $\matr{K}(\varepsilon) \in \RR^{n \times n}$ is symmetric positive semidefinite for  $\varepsilon  \in [0,\varepsilon_0]$ and its eigenvalues are ordered as 
\[
\lambda_1(\varepsilon) \ge \lambda_2(\varepsilon)  \ge \cdots \ge \lambda_n(\varepsilon) \ge 0. 
\]
Suppose that  there exists a matrix $\matr{U} \in \RR^{n\times m}$, $\matr{U}^{\T}\matr{U} = \matr{I}_m$, such that
\begin{equation}\label{eq:assumption_courant_fischer_order}
\matr{U}^{\T} \matr{K}(\varepsilon)\matr{U} = \O(\varepsilon^L)
\end{equation}
in $[0,\varepsilon_0]$,  $\varepsilon_0 > 0$.
Then  the last $m$ eigenvalues  of $\matr{K}(\varepsilon)$ are of  order at least $L$, i.e.,
\begin{equation}\label{eq:courant_fischer_order_last}
\lambda_j(\varepsilon) = \O(\varepsilon^L), \text{ for }  n-m < j \le n.
\end{equation}
\end{lemma}
\begin{proof}
Assumption \eqref{eq:assumption_courant_fischer_order} and equivalence of matrix norms implies that 
\[
\|\matr{U}^{\T} \matr{K}(\varepsilon)\matr{U}\|_2 \le C \varepsilon^{L}.
\]
for some constant ${C}$.
Hence, we have that for any $\vect{z} \in \RR^{m} \setminus \{0\}$,
\[
\frac{\vect{z}^{\T}\matr{U}^{\T} \matr{K}(\varepsilon)\matr{U} \vect{z}}{\vect{z}^{\T}\matr{U}^{\T} \matr{U} \vect{z}} \le\|\matr{U}^{\T} \matr{K}(\varepsilon)\matr{U}\|_2  \le  C \varepsilon^{L}.
\]
By choosing the range of $\matr{U}$ as  $\mathcal{L}$ and applying  \Cref{lem:courant_fischer}, we  complete the proof.
\end{proof}

Next, we recall a classic result on eigenvalues/eigenvectors for analytic perturbations.
\begin{theorem}[{\cite[Chapter II, Theorem 1.10]{kato1995perturbation}}]\label{lem:analytic_eigenvalues}
Let $\matr{K}(\varepsilon)$ be an  $n\times n$  matrix depending analytically on $\varepsilon$ in the neighborhood of $0$ and symmetric for real $\varepsilon$.
Then all the eigenvalues $\lambda_k(\varepsilon)$,  $1\le k\le n$  can be chosen  analytic; moreover, the  orthogonal projectors $\matr{P}_k(\varepsilon)$ on the corresponding rank-one eigenspaces can be also chosen analytic, so that 
\begin{equation}\label{eq:analytic_evd}
\matr{K}(\varepsilon) = \sum_{k=1}^n \lambda_k(\varepsilon)\matr{P}_k(\varepsilon)
\end{equation}
is the eigenvalue decomposition of $\matr{K}(\varepsilon)$ in a neighborhood of $0$.
\end{theorem}

We will be  interested in finding the limiting rank-one projectors in \Cref{lem:analytic_eigenvalues}, i.e.,
\[
\matr{P}_k = \lim_{\varepsilon \to 0}\matr{P}_k(\varepsilon) =  \matr{P}_k(0),
\]
where the last equality follows from analyticity of $\matr{P}_k(\varepsilon)$ at $0$.
 Note that, for kernel matrices, it is impossible to retrieve the information about  the limiting projectors just from $\matr{K}(0)$ (which is rank-one). 
In what follows, instead of $\matr{P}_k$ we will talk about limiting eigenvectors $\vect{p}_k$ (i.e., $\matr{P}_k =\vect{p}_k\vect{p}_k^{\T}$),  although the latter are defined only up to a change of sign.

Armed with \Cref{lem:analytic_eigenvalues}, we can obtain an extension of \Cref{lem:courant_fischer_order}, which also gives us information about  limiting eigenvectors. 
\begin{lemma}\label{lem:courant_fischer_ev}
Let $\matr{K}(\varepsilon)$ and $\matr{U} \in \mathbb{R}^{n \times m}$ be as in \Cref{lem:courant_fischer_order}, and, moreover, $\matr{K}(\varepsilon)$  be analytic in the neighborhood of $0$.
Then $\mspan(\matr{U})$ contains all the limiting eigenvectors corresponding to the eigenvalues of the order at least $L$ (i.e., $ \O(\varepsilon^L)$).
\end{lemma}
\begin{proof} 
Let $L_1 \le \cdots \le L_k \le +\infty$ be the exact orders of the eigenvalues, i.e., either
\[
\lambda_k(\varepsilon) = \varepsilon^{L_k} (\widetilde{\lambda}_k + \O(\varepsilon)), \quad \widetilde{\lambda}_k  > 0,
\]
or $L_k = +\infty$ in the case $\lambda_k(\varepsilon) \equiv 0$.
Then, from orthogonality of the eigenvectors, we only need to prove that, $\matr{p}_k^{\T} \matr{U} = 0$ for $L_k < L$.
Note that $\widetilde{\lambda}_k  > 0$  due to positive semidefiniteness of $\matr{K}(\varepsilon)$.

Suppose that, on the contrary, there exists $k$ with $L_k < L$ such that $\matr{p}_k^{\T}  \matr{U} \neq 0$, and let it be the minimal such $k$.
Then from the  factorization \eqref{eq:analytic_evd} we have that
\[
\matr{U}^{\T}\matr{K}(\varepsilon)\matr{U} = \varepsilon^{L_k} \Big( {\sum\limits_{\ell: L_\ell = L_k}\widetilde{\lambda}_{\ell}  \matr{U}^{\T} \vect{p}_k\vect{p}_k^{\T}  \matr{U}} + \O(\varepsilon)\Big),
\]
with a nonzero main term (from $\widetilde{\lambda}_k > 0$), which contradicts the assumption \eqref{eq:assumption_courant_fischer_order}.
\end{proof}

\subsection{Saddle point matrices and elementary symmetric polynomials}
In this subsection, we will be interested in determinants and elementary symmetric polynomials for so-called saddle point matrices \cite{benzi2006eigenvalues}.
Let  $\matr{V} \in \RR^{n\times r}$ be a full column rank matrix.
For a matrix $\matr{A} \in \RR^{n\times n}$, we define the  \emph{saddle-point matrix} as
\[
\begin{bmatrix}
\matr{A} &  \matr{V} \\
 \matr{V}^{\T} & 0 
\end{bmatrix}.
\]
Consider a full QR decomposition of $\matr{V} \in \RR^{n\times r}$, i.e.,
\begin{equation}\label{eq:qr_full}
\matr{V} = \Qfull\Rfull = \Qthin\Rthin, \quad
\Qfull = \begin{bmatrix}\Qthin & \Qort \end{bmatrix},
\quad
\Rfull = \begin{bmatrix}\Rthin \\ \matr{0}\end{bmatrix},
\end{equation}
where 
 $\Qfull \in \RR^{n\times n}$, $\Qort \in \RR^{n\times (n\text{-}r)}$,  $\matr{Q}_{\rm full}^{\T} \Qfull = \matr{I}$, $\Rthin$ is upper-triangular, and
 $\Qort^{\T} \matr{V} = 0$.
 
\begin{lemma}\label{lem:saddle_point_mat}
For any $\matr{A} \in \RR^{n \times n}$ and $\matr{V} \in \RR^{n\times r}$, it holds that 
\[
\det
\begin{bmatrix}
\matr{A} &  \matr{V} \\
 \matr{V}^{\T} & 0 
\end{bmatrix} 
 = (-1)^{r} \det (\matr{V}^{\T}  \matr{V} ) \det(\Qort^{\T}\matr{A} \Qort)  
  = (-1)^{r} (\det \Rthin)^2 \det(\Qort^{\T}\matr{A} \Qort),
\]
where the matrices $\Rthin$ and $\Qort$ are from  the full QR factorization given in \eqref{eq:qr_full}.
\end{lemma}
The proof of \Cref{lem:saddle_point_mat} is contained in \Cref{sec:proofs_saddle}.

We will also also need to evaluate the elementary symmetric polynomials of matrices of the
form $\Qort^{\T}\matr{A} \Qort$. 
For a power series (or a polynomial) 
\[
a(t) = a_0 + a_1t + \cdots +a_N t^{N} + \ldots,
\]
we use the following notation, standard in combinatorics, for its coefficients:
\[
  \coefAtMonomial{a(t)}{t^r} \eqdef a_r = \frac{a^{(r)}(0)}{r!}.
\]
With this notation, the following lemma on low-rank perturbations of $\matr{A}$ holds true.
\begin{lemma}\label{lem:esp_coef_perturbation}
  Let $\matr{A} \in \RR^{n \times n}$ and $\matr{V} \in \RR^{n\times r}$, with
  the QR decomposition of $\matr{V}$ given as before by eq. \eqref{eq:qr_full}.
Then for $k \ge r$ the polynomial $a(t) = e_k(\matr{A} + t \matr{V}\matr{V}^{\T})$ has degree at most $r$, and its leading coefficient is given by
\[
\coefAtMonomial{e_k(\matr{A} + t \matr{V}\matr{V}^{\T})}{t^{r}} = 
 \det(\matr{V}^{\T} \matr{V}) e_{k-r}(\Qort^{\T}\matr{A} \Qort);
\]
In particular, if $k=n$, we get
\[
\coefAtMonomial{\det(\matr{A} + t \matr{V}\matr{V}^{\T})}{t^{r}} = 
 \det(\matr{V}^{\T} \matr{V}) \det(\Qort^{\T}\matr{A} \Qort) =
(-1)^{r} \det\begin{bmatrix}
\matr{A} &  \matr{V} \\
 \matr{V}^{\T} & 0 
\end{bmatrix}.
\]
\end{lemma}
The proof of \Cref{lem:esp_coef_perturbation} is also contained in \Cref{sec:proofs_saddle}.

\begin{remark}
Alternative expressions for perturbations of elementary symmetric polynomials are available in \cite[Theorem 2.16]{ipsen2008perturbation}, 
and \cite[Corollary 3.3]{jain2011derivatives}, but they do not lead directly  to the compact expression in \Cref{lem:esp_coef_perturbation} that we need.
\end{remark}

\section{Results in the 1D case}\label{sec:1d_results}
\subsection{Smooth kernels}
We begin this section by generalizing the result of \cite[Corollary 2.9]{lee2013collocation} on
 determinants of scaled kernel matrix $\matr{K}_{\varepsilon}$ in the smooth case.
\begin{theorem}\label{thm:det_1d_smooth}
Let $K \in \calC^{(n,n)}(\Omega \times \Omega)$ be the kernel, and \eqref{eq:general_kernel_scaled} be its scaled version.  
Then for small $\varepsilon$, 
\begin{enumerate}
\item the determinant of $\matr{K}_{\varepsilon}$ from \eqref{eq:scaled_kernel} has the  expansion
\[
\det (\matr{K}_{\varepsilon}) = \varepsilon^{n(n-1)} (\det(\matr{V}_{\le n-1})^2 \det \Wronsk{n-1} + \O(\varepsilon)).
\]
\item if  $\matr{K}_{\varepsilon}$ is  positive semidefinite on $[0,\varepsilon_0], \varepsilon_0 > 0,$ then eigenvalues have orders
\[
{\lambda}_{k}(\varepsilon) = \O(\varepsilon^{2(k-1)}).
\]
\end{enumerate}
\end{theorem}

While the proof of  1) is given in \cite[Corollary 2.9]{lee2013collocation}, and the proof of 2) for the radial analytic kernels is contained in \cite{schaback2005multivariate,wathen2015eigenvalues},
we provide a short proof of Theorem~\ref{thm:det_1d_smooth} in \Cref{sec:1d_smooth}, which also illustrates the main ideas behind other proofs in the paper.

\Cref{thm:det_1d_smooth}, together  with \Cref{lem:esp-sum-minors} allows us to find the main terms of the eigenvalues for analytic-in-parameter $\matr{K}_{\varepsilon}$.
\begin{theorem}\label{thm:1d_smooth_ev}
Let $K \in \calC^{(n,n)}(\Omega)$ be the  kernel such that $\matr{K}_{\varepsilon}$ is symmetric positive semidefinite on $[0,\varepsilon_0], \varepsilon_0 > 0,$ and analytic in $\varepsilon$ in the neighborhood of $0$.
Then for $s \le n$ it holds that 
\[
{\lambda}_{s} = \varepsilon^{2(s-1)}(\widetilde{\lambda}_s + \O(\varepsilon)),
\]
where the main terms satisfy
\begin{equation}\label{eq:1d_smooth_ev_main_term}
\widetilde{\lambda}_{1} \ldots \widetilde{\lambda}_{s}  = \det(\matr{V}^{\T}_{\le s-1}\matr{V}_{\le s-1}) \det(\Wronsk{s-1}).
\end{equation}
\end{theorem}
\begin{proof}
First, due to analyticity and \Cref{thm:det_1d_smooth}, we have that expansions \eqref{eq:lambda_orders_differentiable} are valid for $L_s = 2(s-1)$.
Second,  the submatrices of $\matr{K}_{\varepsilon}$  are also  kernel matrices (of smaller size), which, in turn can be found from  \Cref{thm:det_1d_smooth}.
More precisely,
\begin{align}
e_s(\matr{K}_{\varepsilon})& = \sum\limits_{1\le k_1  < \cdots < k_s \le n} \det (\matr{K}_{\varepsilon}({x_{k_1},\ldots,x_{k_s}}))\nonumber \\
& = \varepsilon^{s(s-1)}\Big(\det\Wronsk{s-1} \sum\limits_{1\le  k_1< \cdots <k_s  \le n} (\det \matr{V}_{\le s-1}({x_{k_1},\ldots,x_{k_s}}))^{2} + \O(\varepsilon) \Big)\nonumber \\
& =  \varepsilon^{s(s-1)}  \left(\det\Wronsk{s-1}\det (\matr{V}_{\le s-1}^{\T}\matr{V}_{\le s-1})+ \O(\varepsilon)\right), \label{eq:esp_1d_smooth}
\end{align}
where the penultimate equality follows from \Cref{thm:det_1d_smooth}, and the last equality follows from the Binet-Cauchy formula.
\end{proof}

Finally, we employ \Cref{lem:main_terms_esp_separated}   on  relations between the main terms in \eqref{eq:lambda_orders_differentiable} and \eqref{eq:esp_orders_differentiable}.
\begin{corollary}\label{cor:1d_smooth_ev_ratio}
Let $\matr{K}_{\varepsilon}$ be as in  \Cref{thm:1d_smooth_ev}, and
 the  points in $\X$ be distinct.
\begin{enumerate}
\item If for $1 < s \le n$, $\det \Wronsk{s-2} \neq 0$, the main term of the $s$-th eigenvalue can be obtained as
\begin{equation}\label{eq:1d_smooth_ev_ratio}
\widetilde{\lambda}_{s}  = \frac{\det(\matr{V}^{\T}_{\le s-1}\matr{V}_{\le s-1})}{\det(\matr{V}^{\T}_{\le s-2}\matr{V}_{\le s-2})} \cdot\frac{\det(\Wronsk{s-1})}{ \det(\Wronsk{s-2})}.
\end{equation}
\item If  $1 \le s < n$, $\det \Wronsk{s-1} \neq 0$, then the limiting eigenvectors $\vect{p}_1, \ldots, \vect{p}_{s}$ are the first $s$ columns of the $\Qfull$ factor of the QR factorization \eqref{eq:qr_full} of $\matr{V}_{\le s-1}$.
\item In particular, if $\det \Wronsk{n-2} \neq 0$, then all the main terms of the eigenvalues are given by \eqref{eq:1d_smooth_ev_ratio}, and all the limiting eigenvectors are given by the columns of the $\matr{Q}$ matrix in the QR factorization of $\matr{V}_{\le n-1}$.
\end{enumerate}
\end{corollary}
The proof of \Cref{cor:1d_smooth_ev_ratio} is also contained in \Cref{sec:1d_smooth}.
\begin{remark}
In \Cref{cor:1d_smooth_ev_ratio}, by Cramer's rule, the individual ratios in \eqref{eq:1d_smooth_ev_ratio} can be computed in the following  way:
\[
\frac{\det(\matr{V}^{\T}_{\le s-1}\matr{V}_{\le s-1})}{\det(\matr{V}^{\T}_{\le s-2}\matr{V}_{\le s-2})} = ((\matr{V}^{\T}_{\le s-1}\matr{V}_{\le s-1})^{-1})^{-1}_{s,s}=  R^2_{s,s},
\]
where $R_{s,s}$ is the last ($(s,s)$-th) diagonal element of the $\Rthin \in \RR^{s\times s}$ matrix in the thin QR decomposition of $\matr{V}_{\le s-1}$.
Similarly, 
\[
\frac{ \det(\Wronsk{s-1})}{\det(\Wronsk{s-2})}  = ((\Wronsk{s-1})^{-1})^{-1}_{s,s}= C^2_{s,s},
\]
where $C_{s,s}$ is the last diagonal element of the Cholesky factor of $\Wronsk{s-1} = \matr{C}\matr{C}^{\T}$.
\end{remark}

\subsection{Finite smoothness}
Next, we provide an analogue of \Cref{thm:det_1d_smooth} for a radial kernel with the order of smoothness $r$, which is smaller or equal to the number of points (i.e., \Cref{thm:det_1d_smooth} cannot be applied.).
\begin{theorem}\label{thm:det_1d_finite_smoothness}
For a radial kernel  \eqref{eq:rbf_kernel_scaled} with order of smoothness $r \le n$,
\begin{enumerate}
\item the determinant can be expressed as
\[
\det (\matr{K}_{\varepsilon}) =\varepsilon^{n(2r-1)-r^2} \left( \widetilde{k} + \O(\varepsilon)\right),
\]
where the main term is given by
\begin{align}
\widetilde{k}&=
(-1)^{r} 
\det \Wronsk{r-1}  \det
\begin{bmatrix}
f_{2r-1}  \bD_{(2r-1)} &  \matr{V}_{\le r-1} \\
 \matr{V}_{\le r-1}^{\T} & 0 
\end{bmatrix}\label{eq:det_finite_smoothness}\\
 &=  
\det \Wronsk{r-1}   \det(\matr{V}^{\T}_{\le r-1}\matr{V}_{\le r-1})
\det (f_{2r-1}  \Qort^{\T}\bD_{(2r-1)}\Qort),\label{eq:det_finite_smoothness_Q}
\end{align}
where $\Qort \in \RR^{n\times (n-r)}$ is the semi-orthogonal matrix such that $\Qort^{\T}\matr{V}_{\le r-1} = 0$
(e.g., the matrix $\Qort$ in the full QR decomposition \eqref{eq:qr_full} of $\matr{V} = \matr{V}_{\le r-1}$).

\item If $\matr{K}_{\varepsilon}$  is positive semidefinite on $[0,\varepsilon_0], \varepsilon_0 > 0,$ the eigenvalues have  orders
\[
{\lambda}_{s}(\varepsilon) = 
\begin{cases}
\O(\varepsilon^{2(s-1)}), & s \le r,\\
\O(\varepsilon^{2r-1}), &  s > r. \\
\end{cases}
\]
\end{enumerate}
\end{theorem}
The proof of \Cref{thm:det_1d_finite_smoothness} again is postponed to  \Cref{sec:proof_1d_fs} in order to present a more straightforward corollary on eigenvalues.

As an example, we have for $r=1$ (exponential kernel):
\begin{equation}
\label{eq:det-exponential-kernel}
\det(\bK_{\varepsilon}) =- \varepsilon^{n-1} n K(0,0)  \det
\begin{bmatrix}
  \bD_{(1)} &  \mathbf{1} \\
  \mathbf{1}^{\T} & 0 
\end{bmatrix},
\end{equation}
where $\mathbf{1}$ denotes a vector with all entries equal to  $1$.

Combining \Cref{thm:det_1d_finite_smoothness} with \Cref{lem:main_terms_esp_group,lem:courant_fischer_ev}, we get the following result.
\begin{theorem}\label{thm:1d_finite_smoothness_ev}
Let $K$ be a kernel satisfying the assumptions of \Cref{thm:det_1d_finite_smoothness}, where $\matr{K}_{\varepsilon}$ is positive semidefinite on $[0,\varepsilon_0],\varepsilon_0 > 0,$ and analytic in $\varepsilon$.
Then it holds that 

\noindent 1. The main terms of first  $r$ eigenvalues $\widetilde{\lambda}_1,\ldots,\widetilde{\lambda}_r$ satisfy  \eqref{eq:1d_smooth_ev_main_term}, for $1\le s \le r$.
In particular, if  for $1< s\le r$, $\det \Wronsk{s-2} \neq 0$, then $\widetilde{\lambda}_s$ is given by \eqref{eq:1d_smooth_ev_ratio}.

\noindent 2. If the points in $\X$ are distinct and for $1\le s\le r$, $\det \Wronsk{s-1} \neq 0$, then the first $s$ limiting eigenvectors are as in \Cref{cor:1d_smooth_ev_ratio}.
In particular, for the case $\det \Wronsk{r-1}\neq 0$,  the last $n-r$ limiting eigenvectors span the column space of $\Qort$.

\noindent 3. If $\det \Wronsk{r-1}\neq 0$ and the points in $\X$ are distinct, then  $\widetilde{\lambda}_{r+1},\ldots, \widetilde{\lambda}_{n}$ are the eigenvalues of 
\[
f_{2r-1} (\Qort^{\T} \bD_{(2r-1)}  \Qort).
\]
\end{theorem}
The proof of \Cref{thm:1d_finite_smoothness_ev} is also contained in \Cref{sec:proof_1d_fs}.
Note that we obtain results on the precise locations of the last $n-r$ limiting eigenvectors in \Cref{sec:eigenvectors}.
\subsection{Proofs  for the 1D smooth case}\label{sec:1d_smooth}
We need the following technical lemma. 
\begin{lemma}\label{lem:scaling_1d}
For any upper triangular matrix  $\matr{R} \in \RR^{(k+1) \times (k+1)}$  it holds that
\[
\matr{\Delta}^{-1} \matr{R} \matr{\Delta}  =  \diag(\matr{R}) + \O(\varepsilon),
\]
where $\diag(\matr{R})$ is  the diagonal part of $\matr{R}$ and $\matr{\Delta} = \Dscale{k}{\varepsilon} \in \RR^{(k+1) \times (k+1)}$ is defined as
\begin{equation}\label{eq:diag_scaling}
\Dscale{k}{\varepsilon} \eqdef \begin{bmatrix}
1 & & & \\
& \varepsilon & & \\
&  & \ddots & \\
&  &  & \varepsilon^{k}
\end{bmatrix}
\end{equation}

\end{lemma}
\begin{proof}
A direct calculation gives
\[
\matr{\Delta}^{-1} \matr{R} \matr{\Delta}  = 
\begin{bmatrix}
R_{1,1} & \varepsilon R_{1,2} & \cdots & \varepsilon^{k} R_{1,k+1} \\
0 & R_{2,2} & \ddots & \vdots  \\
\vdots & \ddots & \ddots & \varepsilon R_{k,k+1}  \\
0&  \cdots & 0 &R_{k+1,k+1}
\end{bmatrix}.
\]
\end{proof}

\begin{proof}[Proof of {\Cref{thm:det_1d_smooth}}]
We will use a special form of the Maclaurin expansion (i.e., the Taylor expansion at $0$) for bivariate functions, differentiable  with respect to  a ``rectangular'' set of multi-indices. 
Let us take  $x, y \in  \Omega$  and apply first the Maclaurin expansion  with respect to $x$:
\begin{equation}\label{eq:maclaurin_x} 
K({x},{y}) = K(0,y) +  x   K^{(1,0)} (0,y)  +\cdots+   \frac{{x}^{n-1} K^{(n-1,0)} (0,y)}{(n-1)!}  +    r_{n} (x,y).
\end{equation}
Then the Maclaurin expansion of \eqref{eq:maclaurin_x} with respect to $y$ yields 
\[
\begin{array}{c@{}c@{}c@{}c@{}c@{}c@{}c@{}|c@{}}
 K({x},{y})  & = &  K^{(0,0)} &+ \cdots +& \frac{{y}^{n\text{-}1}K^{(0,n\text{-}1)}}{(n\text{-}1)!}&+&   \frac{y^nK^{(0,n)}(0, \theta_{y,1}y)}{n!}  & { = K(0,y)} \\
  & &  &  &  & & &  \\
  & +& x K^{(1,0)} & +\cdots+ &  \frac{x{y}^{n\text{-}1}K^{(1,n\text{-}1)}}{(n\text{-}1)!}&+&    \frac{x{y}^{n}K^{(1,n)}(0, \theta_{y,2}y)}{n!} &  {= x K^{(1,0)} (0,y)} \\
  & & \vdots &  & \vdots & & \vdots & \vdots \\
&+&   \frac{{x}^{n\text{-}1}K^{(n\text{-}1,0)} }{(n\text{-}1)!}&+ \cdots +&  \frac{{x}^{n\text{-}1}{y}^{n\text{-}1}K^{(n\text{-}1,n\text{-}1)}}{(n\text{-}1)!(n\text{-}1)!}&+&    \frac{x^{n\text{-}1}{y}^{n}K^{(n\text{-}1,n)}(0, \theta_{y,n}y)}{(n\text{-}1)!n!} &  ={\frac{{x}^{n\text{-}1}{K^{(n\text{-}1,0)}} (0,y)}{(n\text{-}1)!}}\\
  & &  & &  & &  & \\
  & +&  r_n(x,y),  & &  & &  & \\
\end{array}
\]
where $0 \le \theta_{y,1}, \ldots, \theta_{y,n} \le 1$ depend only on $y$,  the corresponding terms of \eqref{eq:maclaurin_x}  are given on the right, and  a shorthand notation $K^{(i,j)} = K^{(i,j)}(0,0)$ is used. 
From the integral form of  $r_{n}(x,y)$ and Taylor expansion for $g(y) \eqdef K^{(n,0)}(t,y)$ (as in \eqref{eq:maclaurin_x}), we get
\begin{align*}
&r_{n}(x,y)   = \int\limits_{0}^{x} \frac{ (x-t)^{n-1} K^{(n,0)}(t,y)}{(n-1)!} dt \\
& =
\sum\limits_{\ell=0}^{n-1}
\left(\int\limits_{0}^{x} \frac{(x-t)^{n-1}}{(n-1)!}
\frac{y^{\ell} K^{(n,\ell)}(t,0)}{\ell!} dt \right)
+  \int\limits_{0}^{x}\!\!\int\limits_{0}^{y} \frac{(x-t)^{n-1}(y-s)^{n-1} K^{(n,n)}(t,s)}{(n-1)!(n-1)!} dtds.
\end{align*}
By the mean value theorem (as in \cite[\S 6.3.3]{zorich2004analysis}), there exist   $0 \le \eta_{x,1}, \ldots, \eta_{x,n} \le 1$ (depending only on $x$) and $0 \le \zeta_{x,y}, \xi_{x,y} \le 1$, such that $r_n(x,y) =$
\[
     \frac{x^{n}K^{(n,0)} (\eta_{x, 1 }x,0)}{n!}   +\cdots+   \frac{x^{n}y^{n\text{-}1}{K^{(n,n\text{-}1)}}(\eta_{x,n}x,0)}{n! (n\text{-}1)!}+    \frac{x^{n}{y}^{n} K^{(n,n)} (\zeta_{x,y}x,\xi_{x,y}y)}{n!n!},
\]
Next, with some abuse of notation,  let $\varepsilon_0>0$  be such that   $\varepsilon_0 x,\varepsilon_0 y \in \Omega$. 
Then, after replacing $(x,y)$ with $(\varepsilon x, \varepsilon y)$, $\varepsilon \in [0, \varepsilon_0]$ in the  expansion of $K(x,y)$, we obtain
\begin{equation}\label{eq:kernel1d_expansion_compact}
\begin{split}
&K(\varepsilon {x}, \varepsilon {y}) = 
[1, \varepsilon x, \ldots,  (\varepsilon x)^{n-1}] \Wronsk{n-1}[1, \varepsilon y, \ldots,  (\varepsilon y)^{n-1}]^{\T} \\
&+  \varepsilon^{n} ([1,\varepsilon x, \ldots, (\varepsilon x)^{n-1}]  \vect{w}_{1,y}(\varepsilon) +
\vect{w}_{2,x}^{\T}(\varepsilon) [1, \varepsilon y, \ldots, (\varepsilon y)^{n-1}]^{\T}) + \varepsilon^{2n} w_{3,x,y}(\varepsilon),
\end{split}
\end{equation}
where $\Wronsk{n-1}$ is the Wronskian matrix defined in \eqref{eq:Wronskian},   
\[
w_{3,x,y}(\varepsilon)\eqdef\frac{x^{n}{y}^{n} K^{(n,n)} (\zeta_{\varepsilon x,\varepsilon y}(\varepsilon x),\xi_{\varepsilon x, \varepsilon y}(\varepsilon y))}{n!n!},
\]
such that $w_{3,x,y}: [0,\varepsilon_0] \to \RR^{n}$ is bounded
and $\vect{w}_{1,y},\vect{w}_{2,x}: [0,\varepsilon_0] \to \RR^{n}$ are bounded  vector functions (depending only on $y$ and $x$ respectively),  defined as 
\begin{align*}
\vect{w}_{1,y}(\varepsilon) & \eqdef 
\frac{y^n}{n!}
\begin{bmatrix}
 K^{(0,n)}(0, \theta_{\varepsilon y,1}(\varepsilon y)) & K^{(1,n)}(0, \theta_{\varepsilon y,2}(\varepsilon y))    \,\cdots \, 
\frac{K^{(n-1,n)}(0, \theta_{\varepsilon y,n}(\varepsilon y)) }{(n-1)!}\end{bmatrix}^{\T}, \\
  \vect{w}_{2,x}(\varepsilon) &\eqdef
   \frac{x^{n}}{n!}
 \begin{bmatrix}
 K^{(n,0)}(\eta_{\varepsilon x,1}(\varepsilon x),0) &  K^{(n,1)}(\eta_{\varepsilon x,2}(\varepsilon x),0)  \,\cdots \,
 \frac{K^{(n,n-1)}(\eta_{\varepsilon x,n}(\varepsilon x), 0) }{(n-1)!}\end{bmatrix}^{\T}.\end{align*}
Let  $\varepsilon_0 > 0$, such that $\{\varepsilon_0  {x}_1, \ldots, \varepsilon_0  {x}_n \} \subset \Omega$.
From \eqref{eq:kernel1d_expansion_compact},  the scaled kernel matrix 
admits for $\varepsilon \in [0, \varepsilon_0]$ the expansion
\begin{equation}\label{eq:macalurin_matrix_1d}
\matr{K}_{\varepsilon} = \matr{V} \matr{\Delta}  \matr{W} \matr{\Delta}  \matr{V}^{\T} + 
\varepsilon^{k+1} (\matr{V} \matr{\Delta} \matr{W}_1(\varepsilon)  +  \matr{W}_2(\varepsilon) \matr{\Delta}
\matr{V}^{\T}) + \varepsilon^{2(k+1)} \matr{W}_3(\varepsilon)
\end{equation}
with $k = n-1$,  $\matr{W} \!=\! \Wronsk{k}$, $\matr{V} \!=\! \matr{V}_{\le k}$,  $\matr{\Delta}\! =\! \Dscale{k}{\varepsilon}$   as in \eqref{eq:diag_scaling},  and $\matr{W}_1(\varepsilon) ,\matr{W}_2(\varepsilon),\matr{W}_3(\varepsilon)$  are $\O(1)$ matrices defined respectively as $\matr{W}_3(\varepsilon) \eqdef \left[{w}_{3,x_i,x_j}(\varepsilon)\right]_{i,j=1}^{n,n}$,
\begin{align*}
\matr{W}_1(\varepsilon) \eqdef \begin{bmatrix}\vect{w}_{1,x_1} (\varepsilon) & \cdots & \vect{w}_{1,x_n} (\varepsilon) \end{bmatrix},
\matr{W}_2(\varepsilon) \eqdef \begin{bmatrix}\vect{w}_{2,x_1} (\varepsilon)& \cdots & \vect{w}_{2,x_n} (\varepsilon) \end{bmatrix}^{\T}.
\end{align*} 
Let $\matr{V} = \matr{Q}\matr{R}$ be the (full) QR factorization.
Then from \Cref{lem:scaling_1d}, we have that
\[
\matr{\Delta}^{-1} \matr{Q}^{\T} \matr{V}  \matr{\Delta} = \widetilde{\matr{R}} + \O(\varepsilon), \quad\text{where }  \widetilde{\matr{R}} = \diag(\matr{R}).
\]
By pre-/post-multiplying \eqref{eq:macalurin_matrix_1d} by $\matr{\Delta}^{-1}\matr{Q}^{\T}$ and its transpose, we get 
\begin{align}
\matr{\Delta}^{-1} \matr{Q}^{\T}  \matr{K}_{\varepsilon} \matr{Q} \matr{\Delta}^{-1}  
&= (\widetilde{\matr{R}} +  \O(\varepsilon)) \matr{W}  (\widetilde{\matr{R}}^{\T} +  \O(\varepsilon)) \nonumber\\
& + \varepsilon^{k+1} (\widetilde{\matr{R}} +  \O(\varepsilon))  \matr{W}_1(\varepsilon) \matr{Q}   \matr{\Delta}^{-1}+ 
\varepsilon^{k+1} \matr{\Delta}^{-1} \matr{Q}^{\T}  \matr{W}_2(\varepsilon) (\widetilde{\matr{R}}^{\T} +  \O(\varepsilon)) \nonumber \\ \phantom{\matr{W}}
&+ (\varepsilon^{k+1} \matr{\Delta}^{-1}) \matr{Q}^{\T}  \matr{W}_3(\varepsilon)\matr{Q}  (\varepsilon^{k+1} \matr{\Delta}^{-1}) \nonumber\\
&= \widetilde{\matr{R}} {\matr{W}}\widetilde{\matr{R}}^{\T}  + \O(\varepsilon).\label{eq:1d_smooth_scaled_K}
\end{align}
where the last equality follows from  $\varepsilon^{k+1}\matr{\Delta}^{-1} = \O(\varepsilon)$.
Now we are ready to prove the statements of the theorem.
\begin{enumerate}
\item
From \eqref{eq:1d_smooth_scaled_K} we immediately get
\[
 \varepsilon^{-n(n-1)} \det\matr{K}_{\varepsilon}  = (\det \widetilde{\matr{R}})^2\det{\matr{W}} + \O(\varepsilon)   = (\det {\matr{R}})^2\det{\matr{W}} + \O(\varepsilon).
\] 
\item 
Since $k = n-1$,  we can also rewrite \eqref{eq:1d_smooth_scaled_K} as
\begin{equation}\label{eq:QtKQ_smooth}
\matr{Q}^{\T}  \matr{K}_{\varepsilon} \matr{Q} = 
\begin{bmatrix}
\O(1) & \O(\varepsilon) & \cdots & \O(\varepsilon^{k}) \\
\O(\varepsilon) & \O(\varepsilon^2) & \cdots & \O(\varepsilon^{k+1}) \\
\vdots  & \vdots &     & \vdots \\
\O(\varepsilon^{k}) & \O(\varepsilon^{k+1}) & \cdots & \O(\varepsilon^{2k})  
\end{bmatrix},
\end{equation}
whose lower right submatrices  have orders
\begin{equation}\label{eq:block_s_smooth}
(\matr{Q}_{s:k})^{\T}  \matr{K}_{\varepsilon} \matr{Q}_{s:k} = \O(\varepsilon^{2s}),
\end{equation}
where  $\matr{Q}_{s:k}$ denotes the matrix of the last $n-s$ columns of $\matr{Q}$. 
By \Cref{lem:courant_fischer_order}, \eqref{eq:block_s_smooth}   implies the required orders of the eigenvalues.
\end{enumerate}
\end{proof}
\begin{proof}[Proof of \Cref{cor:1d_smooth_ev_ratio}]
\noindent 1. Follows  from \eqref{eq:1d_smooth_ev_main_term}   and \Cref{lem:main_terms_esp_separated}.

\noindent 2. From \Cref{thm:1d_smooth_ev}, we have $\widetilde{\lambda}_{\ell} \neq 0$ for $1\le \ell \le s$.  
Let $\matr{Q} \in \RR^{n\times n}$ be  the factor of the  QR factorization of $\matr{V}_{\le n-1}$ (which can also be taken as $\Qfull$ factor in the full QR factorization for any $\matr{V}_{\le \ell-1}$, $\ell < n$).
From \eqref{eq:block_s_smooth}  and \Cref{lem:courant_fischer_ev},
$\vect{p}_1,\ldots, \vect{p}_{\ell}$ are orthogonal to $\mspan(\matr{Q}_{:,\ell:k})$, for each $1\le \ell \le s$.
Due to orthonormality of the columns of $\matr{Q}$,  the vectors $\vect{p}_1,\ldots, \vect{p}_{s}$ must coincide with its first $s$ columns.

\noindent 3. The result on eigenvalues follows from 1. For the result on eigenvectors, we note that if the first $n-1$ columns of  $\matr{Q}$ are limiting eigenvectors, then the last column should be the remaining limiting eigenvector.
\end{proof}

\subsection{Proofs for  the 1D finite smoothness case}\label{sec:proof_1d_fs}
\begin{proof}[Proof of {\Cref{thm:det_1d_finite_smoothness}}]
First, we will rewrite the expansion \eqref{eq:finite-smoothness-kernel-expansion} in a convenient form. 
We will group the elements in \eqref{eq:fs_kernel_wronskian_diag} to get
\[
\matr{K}_\varepsilon = \matr{V}_{\le 2r-2} \Dscale{2r-2}{\varepsilon}\matr{W}_{\ultriangle}\Dscale{2r-2}{\varepsilon}\matr{V}_{\le 2r-2}^{\T} + \varepsilon^{2r-1} (f_{2r-1} \matr{D}_{(2r-1)} + \O(\varepsilon)),
\]
where $\matr{W}_{\ultriangle} \in \RR^{(2r-1) \times (2r-1)}$  is the antitriangular matrix defined as
\begin{equation}\label{eq:sum_wronskian_diagonals}
\matr{W}_{\ultriangle} =  \matr{W}_{\diagup 0} +   \matr{W}_{\diagup 2} +  \cdots +  \matr{W}_{\diagup 2(r-2)},
\end{equation}
and  $\matr{W}_{\diagup s}$ are defined\footnote{In the sum \eqref{eq:sum_wronskian_diagonals},  $\matr{W}_{\diagup 2\ell}$ are padded by zeros to ${(2r-1) \times (2r-1)}$ matrices.} in \eqref{eq:wdiag_def_1d}.
For example, in case when $r=2$
\[
\matr{W}_{\ultriangle}= 
\begin{bmatrix}
  f_{0} & 0 & f_{1} \\
  0 & -2f_{1} & 0 \\
  f_1 & 0 & 0
\end{bmatrix}.
\] 
Next, we note that $\matr{W}_{\ultriangle}$ can be split as
\[
\matr{W}_{\ultriangle}  = 
\begin{bmatrix}\Wronsk{r-1} & \widetilde{\matr{W}}_1 \\
\widetilde{\matr{W}}_2 & \\
\end{bmatrix},
\] 
where  $\Wronsk{r-1}$ is exactly the Wronskian matrix defined in \eqref{eq:Wronskian}.
Therefore, since the matrices $\matr{V}_{\le 2r-2}$ and $\Dscale{2r-2}{\varepsilon}$ can be partitioned as
\[
\matr{V}_{\le 2r-2}  = \begin{bmatrix} \matr{V}_{\le r-1} &\matr{V}_{rest} \end{bmatrix}, \quad 
\Dscale{2r-2}{\varepsilon} = \begin{bmatrix} \Dscale{r-1}{\varepsilon} & \\
& \varepsilon^{r} \Dscale{r-2}{\varepsilon} \end{bmatrix},
\]
we get 
\begin{equation*}
\begin{split}
&\matr{V}_{\le 2r-2} \Dscale{2r-2}{\varepsilon} \widetilde{\matr{W}} \Dscale{2r-2}{\varepsilon}\matr{V}_{\le 2r-2}^{\T} 
= \matr{V}_{\le r-1}   \Dscale{r-1}{\varepsilon} \Wronsk{r-1}   \Dscale{r-1}{\varepsilon}\matr{V}_{\le r-1}^{\T}  \\
&+\varepsilon^r \matr{V}_{\le r-1}   \Dscale{r-1}{\varepsilon} \underbrace{\widetilde{\matr{W}}_{1}   \Dscale{r-2}{\varepsilon}\matr{V}_{rest}^{\T}}_{\matr{W}_1(\varepsilon)} + 
\varepsilon^r \underbrace{\matr{V}_{rest} \Dscale{r-2}{\varepsilon}\widetilde{\matr{W}}_{2}}_{\matr{W}_2(\varepsilon)} \Dscale{r-1}{\varepsilon}\matr{V}_{\le r-1}^{\T},
\end{split}
\end{equation*}
which, after denoting $\matr{W} = \Wronsk{r-1}$, $\matr{V} = \matr{V}_{\le r-1}$, $\matr{\Delta} = \Dscale{r-1}{\varepsilon}$, gives
\[
\matr{K}_{\varepsilon} = \matr{V} \matr{\Delta}  \matr{W} \matr{\Delta}  \matr{V}^{\T} + 
\varepsilon^{r} (\matr{V} \matr{\Delta} \matr{W}_1(\varepsilon)  +  \matr{W}_2(\varepsilon) \matr{\Delta}
\matr{V}^{\T}) + \varepsilon^{2r-1} (f_{2r-1} \matr{D}_{(2r-1)} + \O(\varepsilon)).
\]
Next, we take the QR decomposition   $\matr{V}$  \eqref{eq:qr_full}
and consider a  diagonal scaling matrix 
\[
\widetilde{\matr{\Delta}} = \begin{bmatrix} \matr{\Delta} & \\
& \varepsilon^{r-1} \matr{I}_{n-r} \end{bmatrix} \in \RR^{n\times n}.
\]
After pre-/post-multiplying $\matr{K}_{\varepsilon}$ by $\widetilde{\matr{\Delta}}^{-1}\Qfull^{\T}$ and its transpose, we get 
\begin{equation}\label{eq:finite_smoothness_scaled_K}
\widetilde{\matr{\Delta}}^{-1}\Qfull^{\T}\matr{K}_{\varepsilon}\Qfull\widetilde{\matr{\Delta}}^{-1} = 
 \begin{bmatrix}{\matr{\Delta}}^{-1}\Qthin^{\T}\matr{K}_{\varepsilon}\Qthin{\matr{\Delta}}^{-1} & \varepsilon^{1-r} {\matr{\Delta}}^{-1}\Qthin^{\T}\matr{K}_{\varepsilon}\Qort  \\
 \varepsilon^{1-r} \Qort^{\T}\matr{K}_{\varepsilon}\Qthin{\matr{\Delta}}^{-1} & \varepsilon^{2-2r} \Qort^{\T}\matr{K}_{\varepsilon}\Qort \end{bmatrix},
\end{equation}
where $\Qfull = \begin{bmatrix}\Qthin & \Qort \end{bmatrix}$ as in \eqref{eq:qr_full}. 
For the upper-left block we get,  by \Cref{lem:scaling_1d}
\[
{\matr{\Delta}}^{-1}\Qthin^{\T}\matr{K}_{\varepsilon}\Qthin{\matr{\Delta}}^{-1} = 
\diag(\Rthin)\matr{W}\diag(\Rthin) + \O(\varepsilon).
\]
The lower-left block (which is a transpose of the upper-right one) becomes
\[
 \varepsilon^{1-r} \Qort^{\T}\matr{K}_{\varepsilon}\Qthin{\matr{\Delta}}^{\text{-}1} 
\! = \varepsilon (\Qort^{\T}(\matr{W}_2(\varepsilon) \diag(\Rthin) +\varepsilon^{r-1}f_{2r-1} \matr{D}_{(2r-1)}\Qthin{\matr{\Delta}}^{\text{-}1}) + {\O}(\varepsilon)).
\]
Finally, the lower right block is
\[
\varepsilon^{2-2r} \Qort^{\T}\matr{K}_{\varepsilon}\Qort =
\varepsilon ( f_{2r-1} \Qort^{\T}\matr{D}_{(2r-1)} \Qort+ \O(\epsilon)).
\]
\noindent 1.
Combining the blocks in \eqref{eq:finite_smoothness_scaled_K} gives
\begin{align*}
& \varepsilon^{-r(r-1)-2(n-r)(r-1)} \det\matr{K}_{\varepsilon} = 
\det(\widetilde{\matr{\Delta}}^{-1}\Qfull^{\T}\matr{K}_{\varepsilon}\Qfull\widetilde{\matr{\Delta}}^{-1})  = \\
& =
\varepsilon^{n-r} ((\det\Rthin)^2 \det \matr{W}
\det( f_{2r-1} \Qort^{\T}\matr{D}_{(2r-1)} \Qort) + \O(\varepsilon)),\\
& =
(-1)^{r} \varepsilon^{n-r} \left( \det \matr{W}
\det
\begin{bmatrix}
f_{2r-1}  \bD_{(2r-1)} &  \matr{V}_{\le r-1} \\
 \matr{V}_{\le r-1}^{\T} & 0 
\end{bmatrix} + \O(\varepsilon)\right),
\end{align*}
where the last equality follows by \Cref{lem:saddle_point_mat}.

\noindent 2. From \eqref{eq:finite_smoothness_scaled_K} it follows that
\begin{equation}\label{eq:QtKQ_finite_smoothness}
\Qfull^{\T}  \matr{K}_{\varepsilon} \Qfull = 
\begin{bmatrix}
\O(1) & \cdots & \O(\varepsilon^{r-1}) & \O(\varepsilon^{r}) & \cdots &  \O(\varepsilon^{r})   \\
\vdots  &  &\vdots &   \vdots & & \vdots \\
\O(\varepsilon^{r-1}) & \cdots & \O(\varepsilon^{2(r-1)}) & \O(\varepsilon^{2r-1}) & \cdots &  \O(\varepsilon^{2r-1})  \\
\O(\varepsilon^{r}) & \cdots & \O(\varepsilon^{2r-1}) & \O(\varepsilon^{2r-1}) & \cdots &  \O(\varepsilon^{2r-1})  \\
\vdots  &  &\vdots &   \vdots & & \vdots \\
\O(\varepsilon^{r}) & \cdots & \O(\varepsilon^{2r-1}) & \O(\varepsilon^{2r-1}) & \cdots &  \O(\varepsilon^{2r-1})  \\
\end{bmatrix},
\end{equation}
thus \Cref{lem:courant_fischer_order} implies the  orders of the eigenvalues, as in the proof of \Cref{thm:det_1d_smooth}.
\end{proof}
\begin{proof}[{Proof of \Cref{thm:1d_finite_smoothness_ev}}]
1. Note that $K \in \calC^{(r,r)}$, hence for $s \le r$  we can proceed as in \Cref{thm:1d_smooth_ev}, taking into account the fact that the orders of the eigenvalues are given by \Cref{thm:det_1d_finite_smoothness}.
Therefore, \eqref{eq:1d_smooth_ev_main_term} holds true for $1\le s \le n$.

2. For $k=n-1$, from \eqref{eq:QtKQ_finite_smoothness}, we have that
\begin{equation}\label{eq:block_s_finite_smoothness}
(\matr{Q}_{s:k})^{\T}  \matr{K}_{\varepsilon} \matr{Q}_{s:k} = 
\begin{cases}\O(\varepsilon^{2s}), & s < r, \\
\O(\varepsilon^{2r-1}), & s \ge r.
\end{cases}
\end{equation}
Then the statement follows from \Cref{lem:courant_fischer_ev}, as  in the proof of \Cref{cor:1d_smooth_ev_ratio}.

3. Now let us consider the case $s > r$. 
In this case, we have
\begin{align*}
&\varepsilon^{-n(2r-1)+r^2} e_s(\matr{K}_{\varepsilon}) = \varepsilon^{-n(2r-1)+r^2} \sum\limits_{\substack{\calY \subset \{1,\ldots,n\}\\|\calY|=s}} \det (\matr{K}_{\varepsilon,\calY}) \\
& =  \det \Wronsk{r-1} \!\!\!\sum\limits_{|\calY|=s}  
\det(f_{2r-1} \QortSub{\calY}^{\T} \bD_{(2r-1),\calY}  \QortSub{\calY})
\det(  \matr{V}_{\le r-1,\calY}^{\T}\matr{V}_{\le r-1,\calY})
+ \O(\varepsilon)\\
 &= 
\det \Wronsk{r-1}
 \sum\limits_{|\calY|=s}   \coefAtMonomial{\det(f_{2r-1}  \bD_{(2r-1),\calY} + t  \matr{V}_{\le r-1,\calY}  \matr{V}_{\le r-1,\calY}^{\T})}{t^r}+ \O(\varepsilon) \\
&= \det \Wronsk{r-1}
 \coefAtMonomial{e_s(f_{2r-1}  \bD_{(2r-1)} + t  \matr{V}_{\le r-1}  \matr{V}_{\le r-1}^{\T})}{t^r} + \O(\varepsilon)\\
& = \det \Wronsk{r-1} \det(  \matr{V}_{\le r-1}^{\T}\matr{V}_{\le r-1})
e_{s-r}(f_{2r-1} \Qort^{\T} \bD_{(2r-1)} \Qort)+ \O(\varepsilon),
\end{align*}
where the individual steps follow from \Cref{thm:det_1d_finite_smoothness} and \Cref{lem:esp_coef_perturbation}.
Therefore, by \Cref{lem:main_terms_esp_group} we get that for all $1 \le j \le n-r$
\[
\widetilde{e}_j (\widetilde{\lambda}_{r+1}, \ldots, \widetilde{\lambda}_{n}) = e_{j}(f_{2r-1} \Qort^{\T} \bD_{(2r-1)} \Qort),
\]
which together with \Cref{rem:esp_char_poly} completes the proof.
\end{proof}

\section{Multidimensional case: Preliminary facts and notations}\label{sec:nd_notations}
The multidimensional case requires introducing heavier notation, which we  review in this section.

\subsection{Multi-indices and sets}
For a multi-index $\va = (\alpha_1, \ldots, \alpha_d) \in \ZZp^d$,   denote
\[
\va! \eqdef \alpha_1! \cdots \alpha_d!, \qquad 
|\va| \eqdef \sum\limits_{k=1}^d \alpha_k,
\]
where  $0! = 1$ by convention.
For example, $|(2,1,3)| = 6$ and $(2,1,3)! = 12$.

We will frequently use the following notations
\[
\PP_k \eqdef \{ \va \in \ZZp^d : |\va| \le k  \}, \qquad
\HH_k \eqdef \{ \va \in \ZZp^d : |\va| = k  \} = \PP_k \setminus \PP_{k-1}.
\]
The cardinalities of these sets are given by the following well-known formulas:
\[
\# \PP_k = {k+d \choose d}, \qquad \# \HH_k = {k+d-1 \choose d-1} = \# \PP_k - \# \PP_{k-1},
\]
and will be used throughout this paper.

\begin{example}
For $d=1$, we have $\HH_{k} = \{k\}$ and $\PP_{k} = \{0,1, \ldots, k\}$.
For $d=2$, an example is shown in \Cref{fig:multiindices}.
\end{example}

\begin{figure}[t!]
\centering
\begin{tikzpicture}[scale=0.7]
\begin{scope}
    \draw (0,0) node[anchor=north ]  {\small$0$};
    \draw (0.5,0) node[anchor=north ]  {\small$1$};
    \draw (1,0) node[anchor=north ]  {\small$2$};
    \draw (1.5,0) node[anchor=north ]  {\small$3$};
    \draw (0,0.5) node[anchor=east]    {\small$1$};
    \draw (0,1) node[anchor=east]    {\small$2$};
    \draw (0,1.5) node[anchor=east]    {\small$3$};
  \draw[color=black, help lines, line width=.1pt] (0,0) grid[xstep=0.5cm, ystep=0.5cm] (1.9,1.9);
    
     \fill[black] (0,0) circle (2.5pt);
     \fill[black] (0.5,0) circle (2.5pt);
     \fill[black] (0,0.5) circle (2.5pt);
     \fill[black] (1,0) circle (2.5pt);
     \fill[black] (0.5,0.5) circle (2.5pt);
     \fill[black] (0,1) circle (2.5pt);
     \fill[gray] (1.5,0) circle (2.5pt);
     \fill[gray] (1,0.5) circle (2.5pt);
     \fill[gray] (0.5,1) circle (2.5pt);
      \fill[gray] (0,1.5) circle (2.5pt);
   
    \draw [->] (0,0) -- (0,2);
    \draw [->] (0,0) -- (2,0);
\end{scope}
\end{tikzpicture}
\caption{Sets of multiindices, $d=2$. Black dots: $\PP_2$, grey dots: $\HH_3$.}%
\label{fig:multiindices}
\end{figure}

An important class of  multi-index sets is the  lower sets.
An  $\calA \subset \ZZp^d $ is called a \emph{lower set} \cite{gasca2000polynomial} if for any  $\va \in \calA$ all ``lower'' multi-indices are also in the set, i.e.,
\[
\va \in \calA,\; \vect{\beta} \le \va  \; \Rightarrow \; \vect{\beta} \in \calA,
\]
\[
\text{where} \quad (\beta_1, \ldots, \beta_d)  \le (\alpha_1, \ldots, \alpha_d)  \iff \beta_1 \le \alpha_1, \ldots, \beta_d \le \alpha_d.
\] 
Note that all $\PP_k$  are lower sets.

\subsection{Monomials and orderings}
For a vector of variables $\vect{x} = \begin{bmatrix}x_1 & \cdots & x_d\end{bmatrix}^{\T}$,
the monomial  $\vect{x}^{\va}$ is defined as 
\[
\vect{x}^{\va} \eqdef x_1^{\alpha_1} \cdots x_d^{\alpha_d}.
\]%
\begin{remark}
Note that $|\va|$ is the total degree of the monomial $\vect{x}^{\va}$.
The sets of multi-indices $\PP_k$ and $\HH_k$ therefore correspond to    the sets of monomials of degree  $\le k$  and $=k$ respectively.
\[
\{ \vect{x}^{\va} : |\va| \le k  \}, \qquad
 \{ \vect{x}^{\va}  : |\va| = k  \} .
\]
\end{remark}
In what follows, we assume that an ordering of multi-indices, i.e., all the elements  in $\ZZp$ are linearly ordered, i.e. the relation $\prec$ is defined for all pairs of multi-indices.
For example, an ordering for $d=2$ is given by
\begin{equation}\label{eq:ordering_2d}
(0,0) \prec (1,0) \prec (0,1) \prec (2,0) \prec (1,1) \prec(0,2) \prec  (3,0) \prec (2,1) \prec (1,2)  \prec \cdots.
\end{equation}
In this paper, the ordering will not be important, as the results will not depend on the ordering. 
The only requirement is that the order  is graded \cite[Ch.2 \S 2]{cox1997ideals}, i.e.,
\[
 |\va| <  |\vect{\beta}| \quad  \Rightarrow\quad \va \prec \vect{\beta}.
\]
\begin{remark}
For convenience, in case $d \ge 2$, we can use an ordering satisfying
\[
(1,0\ldots,0 )  \prec (0,1\ldots,0 ) \prec \cdots \prec (0,0,\ldots,1),
\]
such that the matrix $\matr{V}_1$ defined in the next subsection is equal to
\[
 \matr{V}_{1}  = \begin{bmatrix} \vect{x}_1 & \ldots & \vect{x}_n\end{bmatrix}^{\T}.
\]
This is not the case for graded lexicographic or  reverse lexicographic \cite[Ch.2 \S 2]{cox1997ideals} orderings.
Instead, a graded reflected lexicographic order \cite{oeisorderings} can be used (see \eqref{eq:ordering_2d}).
\end{remark}

\subsection{Multivariate Vandermonde matrices}
Next, for an ordered set of points $\X = \{\vect{x}_1, \ldots, \vect{x}_n\} \subset \RR^d$ and set of multi-indices  $\calA = \{\va_1, \ldots, \va_m\} \subset \ZZp^d$ ordered according to the chosen ordering, we define the multivariate Vandermonde matrix
\[
\matr{V}_{\cal{A}}  = \matr{V}_{\cal{A}} (\X) = 
\left[
(\vect{x}_i)^{\va_j}
\right]_{\substack{1 \le i \le n\\1 \le j \le m}}.
\]
We will introduce a special notation $\matr{V}_{\le k} \eqdef   \matr{V}_{\PP_k}$ and $\matr{V}_{k} \eqdef  \matr{V}_{ \HH_k}$.
Since the ordering is graded, the matrix  $\matr{V}_{\le k}$  can be split into blocks $\matr{V}_{k}$  arranged by  increasing degree:
\begin{equation}\label{eq:VandermondeND}
\matr{V}_{\le k} = 
\begin{bmatrix}
\matr{V}_{0} & \matr{V}_{1}  & \cdots & \matr{V}_{k} 
\end{bmatrix}.
\end{equation}
It is easy to see that in the case $d=1$ the definition coincides with the previous definition of the Vandermonde matrix \eqref{eq:Vandermonde1D}.

An example of the Vandermonde matrix for $d=2$   and the set of points
\[
\X = \{\left[\begin{smallmatrix}y_1 \\z_1 \end{smallmatrix}\right], \left[\begin{smallmatrix}y_2 \\z_2 \end{smallmatrix}\right], \left[\begin{smallmatrix}y_3 \\z_3 \end{smallmatrix}\right]\},
\]
with the ordering as in \eqref{eq:ordering_2d} is given below
\[
 \matr{V}_{\le 2} =
\left[\begin{array}{c|cc|ccc}
1 & y_1 & z_1 & y_1^2 & y_1z_1 & z_1^2 \\
1 & y_2 & z_2 & y_2^2 & y_2z_2 & z_2^2 \\
1 & y_3 & z_3 & y_3^2 & y_3z_3 & z_3^2 \\
\end{array}\right].
\]

A special case is when the Vandermonde matrix is square, i.e., the number of monomials of degree $\le k$ is equal to the number of points:
\begin{equation}\label{eq:magic_numbers}
n = {{k+d} \choose d} = \# \PP_k.
\end{equation}
For example, $n = k+1$ if $d=1$ and $n = \frac{(k+2)(k+1)}{2}$ if $d=2$.
\begin{remark}
Unlike  the 1D case, even if all the points are different, the Vandermonde matrix $\matr{V}_{\le k}$ is not necessarily invertible.
For example, take the set of points on one of the axes
\[
\X = \{\left[\begin{smallmatrix}-1 \\0 \end{smallmatrix}\right], \left[\begin{smallmatrix}0 \\ 0\end{smallmatrix}\right], \left[\begin{smallmatrix}1 \\0 \end{smallmatrix}\right]\},
\]
for which the Vandermonde matrix  is rank-deficient:
\[
 \matr{V}_{\le 1} =
 \begin{bmatrix}
  1 & - 1& 0 \\
  1 &  0 & 0 \\
    1 &  1 & 0 \\
  \end{bmatrix}.
\]
\end{remark}
This effect is well-known in approximation theory \cite{gasca2000polynomial}.
If the square Vandermonde matrix is nonsingular, then the set of points $\X$ is called unisolvent.
It is known \cite[Prop. 4]{sauer2006polynomial} that  a general configuration of points (e.g., $\X$ are drawn from an absolutely continuous probability distribution with respect to the Lebesgue measure), is unisolvent almost surely.

\subsection{Kernels and smoothness classes}
For a function $f: \mathcal{U} \to \RR$, $\mathcal{U}  \in \RR^{d}$ and a multi-index $\va = (\alpha_1,\ldots,\alpha_d) \in \ZZp^d$ we use a shorthand notation for its  partial derivatives (if they exist):
\[
f^{(\va)}
 (\vect{x}) =\frac{\partial f^{|\va|}}{\partial x_1^{\alpha_1} \cdots\partial x_d^{\alpha_d}} (\vect{x}).
\]
It makes sense to define the smoothness classes with respect to lower sets.
For a lower set $\calA \subset \ZZp^d$ we define the class of functions $\mathcal{U} \to \RR$ which have on $\mathcal{U} $ all continuous derivatives $f^{(\va)}$, $\alpha \in \calA$. 
This class is denoted by  $\calC^{\calA}(\mathcal{U})$.

We will consider   kernels $K: \Omega \times \Omega \to \RR$ in the class $\calC^{(k,k)} (\Omega) \eqdef \calC^{\PP_k \times \PP_k} (\Omega)$,
i.e., which has all partial derivatives up to order $k$ for $\vect{x}$  and $\vect{y}$ separately.

Next, assume that we are given a kernel  $K \in \calC^{\calA \times \calB}(\Omega)$ for lower sets $\calA$ and $\calB$.
We will define the Wronskian matrix for this function as
\begin{equation}\label{eq:wronskian_nd}
\matr{W}_{\calA,\calB}  = 
\left[
\frac{K^{(\va,\vect{\beta})} (\vect{0},\vect{0})}{\va!\vect{\beta}!}
\right]_{\va \in \calA, \vect{\beta} \in \calB},
\end{equation}
where the rows and columns are indexed by multi-indices in  $\calA$ and $\calB$, according to the chosen ordering.

As a special case, we will denote $\Wronsk{k} = \matr{W}_{\PP_k} \eqdef \matr{W}_{\PP_k,\PP_k}$.
For example, for $d=2$ and $k=2$ (Example in \Cref{fig:multiindices}), and ordering \eqref{eq:ordering_2d} we have $\Wronsk{2} = $
\[
\begin{bmatrix}
K^{((0,0),(0,0))} & K^{((0,0),(1,0))} & K^{((0,0),(0,1))} & \frac{K^{((0,0),(2,0))}}{2} & {K^{((0,0),(1,1))}} & \frac{K^{((0,0),(0,2))}}{2} \\
K^{((1,0),(0,0))} & K^{((1,0),(1,0))} & K^{((1,0),(0,1))} & \frac{K^{((1,0),(2,0))}}{2} & {K^{((1,0),(1,1))}} & \frac{K^{((1,0),(0,2))}}{2} \\
K^{((0,1),(0,0))} & K^{((0,1),(1,0))} & K^{((0,1),(0,1))} & \frac{K^{((0,1),(2,0))}}{2} & {K^{((0,1),(1,1))}} & \frac{K^{((0,1),(0,2))}}{2} \\
\frac{K^{((2,0),(0,0))}}{2} & \frac{K^{((2,0),(1,0))}}{2} & \frac{K^{((2,0),(0,1))}}{2} & \frac{K^{((2,0),(2,0))}}{4} & {\frac{K^{((2,0),(1,1))}}{2}} & \frac{K^{((2,0),(0,2))}}{4} \\
K^{((1,1),(0,0))} & K^{((1,1),(1,0))} & K^{((1,1),(0,1))} & \frac{K^{((1,1),(2,0))}}{2} & {K^{((1,1),(1,1))}} & \frac{K^{((1,1),(0,2))}}{2} \\
\frac{K^{((0,2),(0,0))}}{2} & \frac{K^{((0,2),(1,0))}}{2} & \frac{K^{((0,2),(0,1))}}{2} & \frac{K^{((0,2),(2,0))}}{4} & {\frac{K^{((0,2),(1,1))}}{2}} & \frac{K^{((0,2),(0,2))}}{4} \\
\end{bmatrix},
\]
where we omit the arguments of $K^{(\va,\vect{\beta})}$.
We will also need block-antidiagonal  matrices $\matr{W}_{\diagup s} \in \RR^{\# \PP_s \times \# \PP_s}$ defined as follows 
\begin{equation}\label{eq:wronskian_diag_nd}
\matr{W}_{\diagup s} = 
\begin{bmatrix}
& & \matr{W}_{\HH_0,\HH_s} \\ 
& \iddots& \\
\matr{W}_{\HH_s,\HH_0}   & &
\end{bmatrix},
\end{equation}
where  $\matr{W}_{\mathcal{A},\mathcal{B}}$ are  blocks of the Wronskian matrix defined in \eqref{eq:wronskian_nd}.
For example
\[
\matr{W}_{\diagup 0} = 
\begin{bmatrix}
{W}_{0,0}
\end{bmatrix}, 
\quad
\matr{W}_{\diagup 1} = 
\begin{bmatrix}
& \matr{W}_{\HH_0,\HH_1} \\ \matr{W}_{\HH_1,\HH_0}   &
\end{bmatrix}, 
\]
and in general  $\matr{W}_{\diagup s}$ contains the main block antidiagonal of $\Wronsk{s}$.

\subsection{Taylor expansions}\label{sec:taylor_nd}
The standard Taylor  expansion (at $\vect{0}$, i.e., Maclaurin expansion) in the multivariate case is as follows \cite[\S 8.4.4]{zorich2004analysis}.
Let $f \in \calC^{{k+1}}(\Omega)$, where $\Omega$ is an  open neighborhood of $\vect{0}$ containing a line segment from $0$ to $\vect{x}$, denoted as $[\vect{0},\vect{x}]$.
Then the following
the  following Taylor expansion  holds:
\begin{equation}\label{eq:taylor_nd}
f( \vect{x} ) = \sum\limits_{\va \in \PP_{k}} \frac{\vect{x}^{\va}}{\va!} f^{(\va)} ( \vect{0} ) + r_k(\vect{x}),
\end{equation}
where the remainder can be expressed in the Lagrange  or  integral forms:
\[
r_k(\vect{x}) = \int\limits_{0}^{1} (k+1)(1-t)^{k} \Big( \sum\limits_{\vect{\beta} \in \HH_{k+1}} \frac{\vect{x}^{\vect{\beta}}}{\vect{\beta}!} f^{(\vect{\beta})}(t\vect{x}) \Big) dt = 
\sum\limits_{\vect{\beta} \in \HH_{k+1}} \frac{\vect{x}^{\vect{\beta}}}{\vect{\beta}!} f^{(\vect{\beta})} (\theta \vect{x})
\]
with $\theta \in [0,1]$.
A more general  Taylor  (Maclaurin) expansion has remainder in the Peano form, and requires smoothness of order one less, i.e., if $f \in \calC^{{k}}(\Omega)$, we have:
\begin{equation}\label{eq:taylor_nd_peano}
f(\vect{x}) = \sum\limits_{\va \in \PP_{k}} \frac{\vect{x}^{\va}}{\va!} f^{(\va)} (\vect{0}) +o(\norm{\vect{x}}^k_2).
\end{equation}
We  also need a ``bivariate'' version for a  function $f: \Omega \times \Omega \to \RR$ (the arguments are split into two groups)  such that
$f \in \calC^{\PP_{k_1+1} \times \PP_{k_2+1}} (\Omega \times \Omega)$.
Then we can take $\vect{x},\vect{y}$  such that $[0,\vect{x}],[0,\vect{y}] \subset \Omega$ and
apply  the same steps as in the proof of {\Cref{thm:det_1d_smooth}} to  get
\begin{equation*}
\begin{split}
& f(\vect{x},\vect{y})  =  \sum\limits_{\va \in \PP_{k_1}, \vect{\beta} \in \PP_{k_2 }} \frac{\vect{x}^{\va}\vect{y}^{\vect{\beta}}}{\va!\vect{\beta}!}  f^{(\va,\vb)} (\vect{0}, \vect{0})  + \sum\limits_{\va \in \PP_{k_1}, \vect{\beta} \in \HH_{k_2+1}}  \frac{\vect{x}^{\va}\vect{y}^{\vect{\beta}}}{\va!\vect{\beta}!}  f^{(\vect{\alpha},\vect{\beta})} (\vect{0}, \theta_{\vect{y},\vect{\alpha}} \vect{y}) \\
&+   \sum\limits_{\va \in \HH_{k_1+1}, \vect{\beta} \in \PP_{k_2}} \frac{\vect{x}^{\va}\vect{y}^{\vect{\beta}}}{\va!\vb!}  f^{(\va,\vb)} (\eta_{\vect{x},\vb}\vect{x}, \vect{0})  + \sum\limits_{\va \in \HH_{k_1+1}, \vect{\beta} \in \HH_{k_2+1}}  \frac{\vect{x}^{\va}\vect{y}^{\vb}}{\va!\vb!}  f^{(\va,\vb)} (\zeta_{\vect{x},\vect{y}}\vect{x}, \xi_{\vect{x},\vect{y}}\vect{y}), \\
\end{split}
\end{equation*}
where $\{\eta_{\vect{x},\vb}\}_{\vb\in \PP_{k_2}} \subset [0,1]$ depend on $\vect{x}$,  $\{\theta_{\vect{y},\va}\}_{\va \in \PP_{k_1}} \subset  [0,1]$  depend on $\vect{y}$, and $\zeta_{x,y}, \xi_{x,y} \in [0,1]$ depend on both $\vect{x}$ and $\vect{y}$.

\subsection{Distance matrices and expansions of radial kernels}
Next, we consider the radial kernel \eqref{eq:rbf_kernel_scaled} with order of smoothness 
$r$.
For even $k$, as in the univariate case, we will need an expansion in the form
similar to \eqref{eq:fs_kernel_wronskian_diag}, which was obtained
in the univariate case  via the binomial expansion 
\eqref{eq:dist_binomial_expansion}. 
Although we can also use the same approach in the multivariate case, we prefer to
derive this expansion directly from
Taylor's formula.
Let $K \in \calC^{2r-2}(\Omega \times \Omega)$ (not necessarily radial).
Then the Taylor expansion in Peano's form \eqref{eq:taylor_nd_peano} yields an expansion in $\varepsilon$
\[
K(\varepsilon\vect{x},\varepsilon\vect{y}) = \sum_{k=0}^{2r-2} \varepsilon^k \sum\limits_{\substack{\va, \vb \in \PP_{k}\\ |\va|+|\vb| =k}} \frac{\vect{x}^{\va} \vect{y}^{\vb}}{\va!\vb!} K^{(\va,\vb)} (\vect{0},\vect{0}) +o(\varepsilon^{2r-2}),
\]
which in matrix form can be written as
\begin{equation}\label{eq:Keps_diag_expansion_nd}
\matr{K}_{\varepsilon} =  \sum_{k=0}^{2r-2} \varepsilon^{k} \matr{V}_{\le k} \matr{W}_{\diagup k} \matr{V}^{\T}_{\le k} + o(\varepsilon^{2r-2}).
\end{equation}
For a radial kernel, the two expansions \eqref{eq:Keps_diag_expansion_nd} and \eqref{eq:finite-smoothness-kernel-expansion} coincide on $[0,\varepsilon_0]$; therefore,
the

\noindent distance matrices of even order $\matr{D}_{(2\ell)}$ have a compact expression as:
\[
f_{2\ell}\matr{D}_{(2\ell)} = \matr{V}_{\le 2\ell} \matr{W}_{\diagup 2\ell} \matr{V}^{\T}_{\le 2\ell},
\]
and moreover the expansion of $\matr{K}_{\varepsilon}$ given in
\eqref{eq:fs_kernel_wronskian_diag}  is  also valid in the multivariate case
\footnote{Note that there is an equivalent way of obtaining the expansion of the
distance matrices in terms of monomials, simply by writing $\norm{\vect{x}-\vect{y}}^{2p} =
(\sum_{i=1}^d (x_i-2y_ix_i + y_i))^p$ and expanding. }.
\begin{remark}
For $k$ odd, the matrices $\bD_{(k)}$ in the multivariate case also have the conditional positive-definiteness property (as in \Cref{lem:cpd}), except that the number of points should be $n > \# \PP_{r-1}$.
\end{remark}

\section{Results in the multivariate case}\label{sec:nd_results}
\subsection{Determinants in the smooth case}
For a degree $k$, we will introduce a notation 
for the sum of all total degrees of monomials with degrees in $\PP_k$
\[
M  = M(k,d)  \eqdef \sum\limits_{\va \in \PP_{k}} |\va|, \quad \text{such that }
\prod\limits_{\va \in \PP_{k}} {\varepsilon}^{|\va|} = \varepsilon^{M(k,d)},
\]
which is given by \footnote{See \cite[eq. (3.19)--(3.20)]{lee2015flatkernel}, where $M(k,d)$ is given in a slightly different form.}
\[
M(k,d) = d  {{k+d} \choose {d+1}}.
\]
For example, if $d=1$, then $M(k,1) = {{k+2} \choose {2}}$. 
With this notation, we can formulate the result on determinants in the multivariate case.

\begin{theorem}\label{thm:det_nd_smooth}
Assume that the kernel is in $\calC^{(k+1,k+1)}$, the scaled kernel matrix is defined by \eqref{eq:general_kernel_scaled} and \eqref{eq:scaled_kernel},  and also
\[
\# \PP_{k-1} < n \le  \# \PP_{k}. 
\]
\noindent 1. If $n =   \# \PP_{k}$, then 
\[
 \det{\matr{K}_{\varepsilon}} = \varepsilon^{2M(k,d)}(\det{\Wronsk{k}} (\det(\matr{V}_{\le k}))^2 + \O(\varepsilon)).
\]
\noindent 2.  If $n <   \# \PP_{k}$, for $\ell = \# \PP_{k}- n$, we have
\[
\det{\matr{K}_{\varepsilon}} = \varepsilon^{2(M(k,d)-k\ell)}  (\det(\matr{Y}\Wronsk{k} \matr{Y}^{\T}) \det(\matr{V}_{\le k-1}^{\T} \matr{V}_{\le k-1}) + \O(\varepsilon)),
\]
where $\matr{Y} \in \RR^{n \times \# \PP_{k}}$ is defined as
\[
\matr{Y} = 
\begin{bmatrix}
\matr{I}_{\# \PP_{k-1}} &  \\ & \Qort^{\T} \matr{V}_k
\end{bmatrix},
\]
$\matr{V}_k$ is  the Vandermonde matrix  \eqref{eq:VandermondeND} for monomials of degree $k$,
and $\Qort \in \RR^{n \times \ell}$ comes from the full QR decomposition of $\matr{V} = \matr{V}_{\le k-1}$ (see \eqref{eq:qr_full}).

\noindent 3. If  $\matr{K}_{\varepsilon}$   is   positive semidefinite on $[0,\varepsilon_0]$, the eigenvalues split into $k+1$ groups
\begin{equation}\label{eq:nd_smooth_ev_group}
\underbrace{\widetilde{\lambda}_{0,0}}_{\O(1)}, \underbrace{\{\widetilde{\lambda}_{1,j}\}_{j=1}^d}_{\O(\varepsilon^2)}, \ldots, \underbrace{\{\widetilde{\lambda}_{s,j}\}_{j=1}^{\#\HH_s}}_{\O(\varepsilon^{2s})}, \ldots,\underbrace{\{\widetilde{\lambda}_{k,j}\}_{j=1}^{\#\HH_k-\ell}}_{\O(\varepsilon^{2k})}.
\end{equation}
\end{theorem}
The proof of \Cref{thm:det_nd_smooth} is postponed to \Cref{sec:nd_det_proofs}.

From \Cref{thm:det_nd_smooth}, we can also get a result on eigenvalues  and eigenvectors. For this, we
 partition the $\Qfull$ matrix in the full QR factorization of $\matr{V}_{\le k-1}$ as
\begin{equation}\label{eq:Q_partitioning}
\Qfull = \begin{bmatrix} \matr{Q}_{0} & \matr{Q}_{1} & \ldots & \matr{Q}_{k} \end{bmatrix}
\end{equation}
with $\matr{Q}_s\in \RR^{n \times \#\HH_s}$, $0 \le s < k$, and $\matr{Q}_k\in \RR^{n \times (\#\HH_k - \ell)}$.
\begin{theorem}\label{thm:nd_smooth_ev}
Let $K$ be as in \Cref{thm:det_nd_smooth}, such that $\matr{K}_{\varepsilon}$ is symmetric positive semidefinite  on $[0,\varepsilon_0]$ and  analytic  in $\varepsilon $ in a neighborhood of $0$.
Then  the eigenvalues in the groups have the form
\[
{\lambda}_{s,j} = \varepsilon^{2s}(\widetilde{\lambda}_{s,j} + \O(\varepsilon)),
\]
where $\widetilde{\lambda}_{0,0} = n K^{(0,0)}$ and the other main terms are given as follows.

\noindent 1. For $1 \le s <k$, if $\det {\Wronsk{s-1}} \neq 0$ and $\matr{V}_{\le s-1}$ is full rank, then
\begin{equation}\label{eq:nd_smooth_ev_main_term_product}
\widetilde{\lambda}_{s,1} \cdots \widetilde{\lambda}_{s,\#\HH_s}  = \frac{\det(\matr{V}^{\T}_{\le s}\matr{V}_{\le s}) \det(\Wronsk{s})}{\det(\matr{V}^{\T}_{\le s-1}\matr{V}_{\le s-1}) \det(\Wronsk{s-1})}.
\end{equation}
\noindent 2. For any $1 \le s \le k$, if $\det {\Wronsk{s-1}} \neq 0$ and $\matr{V}_{\le s-1}$ is full rank, 
the main terms 
$\widetilde{\lambda}_{s,1}, \ldots, \widetilde{\lambda}_{s,\#\HH_s}$  (or $\widetilde{\lambda}_{k,1}, \ldots, \widetilde{\lambda}_{k,\#\HH_k -\ell}$ if $s=k$)
are the eigenvalues of
\begin{equation}\label{eq:nd_smooth_ev_schur}
\matr{Q}^{\T}_{s} \matr{V}_{s}  \widetilde{\matr{W}} \matr{V}^{\T}_{s}\matr{Q}_{s},
\end{equation}
where $\widetilde{\matr{W}} = \matr{W}_{\lrcorner} - \matr{W}_{\llcorner}
(\Wronsk{s-1})^{-1}\matr{W}_{\urcorner}$ is the Schur complement coming from the
following partition of the Wronskian: 
\[
\Wronsk{s} =  \begin{bmatrix}\Wronsk{s-1} & \matr{W}_{\left\urcorner\right.} \\ \matr{W}_{\llcorner}  &\matr{W}_{\lrcorner}  \end{bmatrix}  =  \begin{bmatrix}\Wronsk{s-1} & \matr{W}_{\urcorner,s} \\ \matr{W}_{\llcorner,s}  &\matr{W}_{\lrcorner,s}  \end{bmatrix},
\]
\noindent 3. For $0 \le s <k$, if $\det {\Wronsk{s}} \neq 0$ and $\matr{V}_{\le s}$ is full rank, then the limiting eigenvectors from the $s$-th group $\vect{p}_{s,1}, \ldots, \vect{p}_{s,\#\HH_s}$  span the column space of $\matr{Q}_s$.
Moreover, if  $\det \Wronsk{k-1} \neq 0$, the remaining eigenvectors span the column space of $\matr{Q}_k$. 
\end{theorem}
The proof of \Cref{thm:nd_smooth_ev} is postponed to \Cref{sec:nd_proofs_ev}.

\Cref{thm:nd_smooth_ev} does not give information on the precise location of limiting eigenvectors in each group. 
We formulate the following conjecture, which we validated numerically.
\begin{conjecture}\label{conj:eigenvectors}
For $1 \le s \le k$, if $\det {\Wronsk{s}} \neq 0$ and $\matr{V}_{\le s}$ is full rank, the limiting eigenvectors in the $\O(\varepsilon^{2s})$ group  are the columns of $\matr{Q}_{s} \matr{U}_s$, where $\matr{U}_s$ contains the eigenvectors\footnote{There is a usual issue of ambiguous definition of $\matr{U}_s$ if the matrix has repeating eigenvalues.} of   the matrix  $\matr{Q}^{\T}_{s} \matr{V}_{s}   \widetilde{\matr{W}}  \matr{V}^{\T}_{s}\matr{Q}_{s}$  from \eqref{eq:nd_smooth_ev_schur}. 
\end{conjecture}

\subsection{Finite smoothness case}
We prove a generalization of \Cref{thm:det_1d_finite_smoothness} to the multivariate case.
\begin{theorem}\label{thm:det_nd_finite_smoothness}
For small $\varepsilon$ and  a radial kernel  \eqref{eq:rbf_kernel_scaled} with order of smoothness $r$:

\noindent 1.
the determinant of $\matr{K}_{\varepsilon}$ in the case $n = \#\PP_{r-1}+N$ given in \eqref{eq:scaled_kernel} has the  expansion
\[
\det (\matr{K}_{\varepsilon}) = \varepsilon^{2M(r-1,d) + (2r-1)N} \left( \widetilde{k} + \O(\varepsilon)\right),
\]
where $\widetilde{k}$ has exactly the same expression as in \eqref{eq:det_finite_smoothness} or \eqref{eq:det_finite_smoothness_Q} (with $\matr{V}_{\le r}$ and $\matr{D}_{(2r-1)}$ replaced with their multivariate counterparts).

\noindent 2.
If $\matr{K}_{\varepsilon}$   is  positive semidefinite on $[0,\varepsilon_0]$, the eigenvalues are split into $r+1$ groups
\[
\underbrace{\widetilde{\lambda}_{0,0}}_{\O(1)}, \underbrace{\{\widetilde{\lambda}_{1,j}\}_{j=1}^d}_{\O(\varepsilon^2)}, \ldots, \underbrace{\{\widetilde{\lambda}_{r-1,j}\}_{j=1}^{\#\HH_r-1}}_{\O(\varepsilon^{2(r-1)})}, \underbrace{\{\widetilde{\lambda}_{k,j}\}_{j=1}^{N}}_{\O(\varepsilon^{2r-1})}.
\]

\noindent 3.
In the analytic in $\varepsilon$ case, the main terms for the first $r$ groups are the same as in \Cref{thm:nd_smooth_ev}.
For the last group, if $\det \Wronsk{r-1} \neq 0$ and $\matr{V}_{\le r-1}$ is full rank, the main terms are the eigenvalues of 
\[
f_{2r-1} (\Qort^{\T} \bD_{(2r-1)}  \Qort),
\]
where $\Qort$ comes from the full QR factorization \eqref{eq:qr_full} of $\matr{V}_{\le r-1}$.

\noindent 4. For $0 \le s < r$, the subspace spanned by the limiting eigenvectors for the $\O(\varepsilon^{2s})$ group of eigenvalues are as in \Cref{thm:nd_smooth_ev}.
If $\det \Wronsk{r-1} \neq 0$ and $\matr{V}_{\le r-1}$ is full rank, the eigenvectors 
for the last group of $\O(\varepsilon^{2r-1})$ eigenvalues span the column space of  $\Qort$.
\end{theorem}
The proof of \Cref{thm:det_nd_finite_smoothness} is postponed to \Cref{sec:nd_proofs_ev}.
Note that we obtain a stronger result on the precise location of the last group of eigenvectors  in \Cref{sec:eigenvectors}. 

\subsection{Determinants in the smooth case}\label{sec:nd_det_proofs}
Before proving \Cref{thm:det_nd_smooth}, we again need a technical lemma, which is an analogue of \eqref{lem:scaling_1d}.

\begin{lemma}\label{lem:scaling_nd}
Let $\matr{R}$ be an upper-block-triangular  matrix 
\[
\matr{R} = 
\begin{bmatrix}
\matr{R}_{1,1} &  \matr{R}_{1,k} &\cdots  & \matr{R}_{1,k+1} \\
0 &\ddots&& \vdots \\
\vdots& \ddots& \ddots &\matr{R}_{k,k+1} \\
0 & \cdots& 0 & \matr{R}_{k+1,k+1} \\
\end{bmatrix},
\]
where the blocks $\matr{R}_{i,j} \in \RR^{M_i \times N_j}$  are not necessarily square.
Then it holds that
\[
\begin{bmatrix}
I_{M_1} & & & \\
& \varepsilon^{\text{-}1} I_{M_2} & & \\
&  & \ddots & \\
&  &  & \varepsilon^{\text{-}k}I_{M_{k+1}} 
\end{bmatrix} \matr{R} \begin{bmatrix}
I_{N_1} & & & \\
& \varepsilon I_{N_2} & & \\
&  & \ddots & \\
&  &  & \varepsilon^{k}I_{N_{k+1}} 
\end{bmatrix} =  \blkdiag(\matr{R}) + \O(\varepsilon),
\]
where $\blkdiag(\matr{R})$ is just the  block-diagonal part of $\matr{R}$:
\[
\blkdiag(\matr{R}) =  
\begin{bmatrix}
\matr{R}_{1,1} & &  \\
& \ddots &  \\
&  & \matr{R}_{k+1,k+1}  \\
\end{bmatrix}.
\]
\end{lemma}
\begin{proof}
The proof is analogous to that of \Cref{lem:scaling_1d}.
\end{proof}

\begin{proof}[Proof  of \Cref{thm:det_nd_smooth}]
First, we  fix  a degree-compatible ordering of multi-indices
\[
\PP_k = \{{\va_1},\ldots, {\va_{\# \PP_k}}\},
\]
and denote the $\#\PP_k \times \#\PP_k$ matrix (note that $\# \PP_k = n + \ell$)
\begin{equation}\label{eq:diag_scaling_nd}
\matr{\Delta} = \matr{\Delta}_k(\varepsilon)  \eqdef
\begin{bmatrix}
\varepsilon^{|{\va_1}|} & & \\
& \ddots & \\
&  & \varepsilon^{|{\va_{\# \PP_k}}|} \\
\end{bmatrix} = 
\begin{bmatrix}
1 & & & \\
& \varepsilon I_{\# \HH_1} & & \\
&  & \ddots & \\
&  &  & \varepsilon^{k}I_{\# \HH_k} 
\end{bmatrix},
\end{equation}
and  by $\matr{E}_n$ the principal $n\times n$ submatrix of $\matr{\Delta}$.
Note that  their determinants are
\begin{equation}\label{eq:det_delta}
\det(\matr{\Delta}) = \varepsilon^{M(k,d)} \text{ and }\det(\matr{E}_n) = \varepsilon^{M(k,d) - \ell k}.
\end{equation}
From the bimultivariate Taylor expansion, as in the proof of {\Cref{thm:det_1d_smooth}}, we get:
\begin{equation}\label{eq:kernelnd_expansion_compact}
\begin{split}
&K(\varepsilon \vect{x},\varepsilon \vect{y}) = 
[(\varepsilon \vect{x})^{\va_{1}},\ldots, (\varepsilon \vect{x})^{\va_{n+\ell}}]\matr{W}[(\varepsilon \vect{y})^{\va_{1}},\ldots, (\varepsilon \vect{y})^{\va_{n+\ell}}]^{\T} \\
 &+ \varepsilon^{k+1} [(\varepsilon \vect{x})^{\va_{1}},\ldots, (\varepsilon \vect{x})^{\va_{n+\ell}}] \vect{w}_{1, \vect{y}} (\varepsilon) + \varepsilon^{k+1} \vect{w}_{2, \vect{x}}(\varepsilon)^{\T} [(\varepsilon \vect{y})^{\va_{1}},\ldots, (\varepsilon \vect{y})^{\va_{n+\ell}}]^{\T} \\
& +  \varepsilon^{2(k+1)} w_{3,\vect{x},\vect{y}} (\varepsilon),
\end{split}
\end{equation}
where $\vect{w}_{1}$, $\vect{w}_{2}$ are bounded (and continuous)  $[0,\varepsilon_0] \to \RR^{n}$ vector functions depending on $\vect{y}$ and $\vect{x}$ respectively, and $w_{3}$ is a bounded (and continuous) function $ [0,\varepsilon_0]  \to \RR$.

Let  $\varepsilon_0 > 0$, such that $\{\varepsilon_0  \vect{x}_1, \ldots, \varepsilon_0  \vect{x}_n \} \in \Omega$ for all $i$.
From \eqref{eq:kernelnd_expansion_compact},  the scaled kernel matrix 
admits for $0 \le \varepsilon \le \varepsilon_0$ the expansion
\begin{equation}\label{eq:macalurin_matrix_nd}
\matr{K}_{\varepsilon} = \matr{V}_{\le k} \matr{\Delta} \matr{W} \matr{\Delta}  \matr{V}^{\T}_{\le k} + 
\varepsilon^{n} (\matr{V}_{\le k} \matr{\Delta} \matr{W}_1(\varepsilon)  +  \matr{W}_2(\varepsilon) \matr{\Delta}
\matr{V}^{\T}_{\le k}) + \varepsilon^{2n} \matr{W}_3(\varepsilon),
\end{equation}
where $\matr{W}_1(\varepsilon) ,\matr{W}_2(\varepsilon),\matr{W}_3(\varepsilon) =  \O(1)$,  $\matr{W} = \Wronsk{k}$, $\matr{V} = \matr{V}_{\le k}$.

Next, we take the full QR decomposition $\matr{V}  = \Qfull\Rfull$ of $\matr{V} = \matr{V}_{\le k-1}$,  partitioned  as in \eqref{eq:qr_full}, so that $\Rthin \in \RR^{\#\PP_{k-1}\times\#\PP_{k-1}} $ and $\Qort \in \RR^{n \times (n-\#\PP_{k-1})}$.
Note that
\[
\Qfull^{\T} \matr{V}_{\le k} =
\begin{bmatrix}
\Rthin & \Qthin^{\T} \matr{V}_{k}  \\
\matr{O} &  \Qort^{\T} \matr{V}_{k} 
\end{bmatrix},
\]
and by \Cref{lem:scaling_nd} we have
\[
\matr{E}^{-1}_n \Qfull^{\T} \matr{V}_{\le k} \matr{\Delta} =
\underbrace{\begin{bmatrix}
\blkdiag({\Rthin}) &   \\
  &  \Qort^{\T} \matr{V}_{k} 
\end{bmatrix}}_{\widetilde{\matr{R}}}  + \O(\varepsilon). 
\]
Next, we pre-/post- multiply \eqref{eq:macalurin_matrix_nd}  by $\matr{E}^{-1}_n\Qfull^{\T}$ and its transpose, to get (as in \eqref{eq:1d_smooth_scaled_K}),
\begin{equation}\label{eq:nd_smooth_scaled_K}
\begin{split}\matr{E}^{-1}_n \Qfull^{\T}  \matr{K}_{\varepsilon} \Qfull
\matr{E}^{-1}_n = \widetilde{\matr{R}} {\matr{W}}\widetilde{\matr{R}}^{\T}  + \O(\varepsilon).
\end{split}
\end{equation}
Finally, we prove the statements of the theorem.
\begin{enumerate}
\item
From \eqref{eq:det_delta} we have that
\[
\varepsilon^{-2(M(k,d) - \ell k)} \det \matr{K}_{\varepsilon} = \det (\widetilde{\matr{R}} {\matr{W}}\widetilde{\matr{R}}^{\T}).
\]
Thus, if $n = \#\PP_k$, then we have
\[
 \det (\widetilde{\matr{R}} {\matr{W}}\widetilde{\matr{R}}^{\T}) = (\det(\widetilde{\matr{R}}))^2 \det(\matr{W}) =  \det( \matr{V}_{\le k})^2 \det(\matr{W}),
\]
where the last equality follows from the fact that
\[
\det(\widetilde{\matr{R}}) = \det(\blkdiag({\Rthin})) \det(\Qort^{\T} \matr{V}_{k} ) = \det(\Qfull^{\T} \matr{V}_{\le k})  = \det(\matr{V}_{\le k}),
\]
because $\Qfull^{\T} \matr{V}_{\le k}$ is block-triangular.
\item
For $n < \# \PP_k$, we note that
\[
\widetilde{\matr{R}} = \begin{bmatrix}
\blkdiag{\Rthin} &   \\
  &  \matr{I}_{\ell}
\end{bmatrix} \matr{Y},
\]
hence
\[
 \det (\widetilde{\matr{R}} {\matr{W}}\widetilde{\matr{R}}^{\T}) =  (\det(\Rthin))^2 \det(\matr{Y}\matr{W}\matr{Y}^{\T}) = \det(\matr{V}_{\le k-1}^{\T} \matr{V}_{\le k-1}) \det(\matr{Y}\matr{W}\matr{Y}^{\T}).
\]
\item
Finally, as in the proof of \Cref{thm:det_1d_smooth}, \eqref{eq:nd_smooth_scaled_K} implies that
\eqref{eq:block_s_smooth} holds as well  in the multivariate case; this, together with \Cref{lem:courant_fischer} completes the proof.
\end{enumerate}
\end{proof}

\subsection{Individual eigenvalues, eigenvectors, and finite smoothness}\label{sec:nd_proofs_ev}
\begin{proof}[Proof of {\Cref{thm:nd_smooth_ev}}]
\noindent 1--2.
Choose a subset  $\calY$ of $\mathcal{X}$ of size $m$, $\# \PP_{s-1} < m \le  \# \PP_{s}$, $s \le k$.
Then we have that 
\[
\matr{K}_{\varepsilon,\calY} = 
\varepsilon^{2(M-s(\# \PP_{s} - m))}  (\det(\matr{Y}\Wronsk{s} \matr{Y}^{\T}) \det(\matr{V}_{\le s-1,\calY}^{\T} \matr{V}_{\le s-1,\calY}) + \O(\varepsilon)),
\]
and 
\[
\matr{Y} = 
\begin{bmatrix}
\matr{I}_{\# \PP_{s-1}} &  \\ & \matr{R}_{s,\calY}
\end{bmatrix},
\]
where $\matr{R}_{s,\calY} = \QortSub{\calY} ^{\T} \matr{V}_s$.
In particular
\begin{equation*}
\begin{split}
\det(\matr{Y}\Wronsk{s} \matr{Y}^{\T}) & =  \det \begin{bmatrix}\Wronsk{s-1} & \matr{W}_{\urcorner}\matr{R}^{\T}_{s,\calY} \\ \matr{R}_{s,\calY}\matr{W}_{\llcorner}  &\matr{R}_{s,\calY}\matr{W}_{\lrcorner} \matr{R}^{\T}_{s,\calY}  \end{bmatrix}\\
& = \det\Wronsk{s-1} \det(\matr{R}_{s,\calY}\matr{W}_{\lrcorner} \matr{R}^{\T}_{s,\calY}  -  \matr{R}_{s,\calY}\matr{W}_{\llcorner}  (\Wronsk{s-1})^{-1} \matr{W}_{\urcorner}\matr{R}^{\T}_{s,\calY} )\\
& = \det\Wronsk{s-1} \det(\QortSub{\calY}^{\T} \matr{V}_{s}   \widetilde{\matr{W}} \matr{V}^{\T}_{s} \QortSub{\calY} ), \\
\end{split}
\end{equation*}
hence by \Cref{lem:saddle_point_mat}
\begin{equation*}
\begin{split}
&\det(\matr{Y}\Wronsk{s} \matr{Y}^{\T}) \det(\matr{V}_{\le s-1,\calY}^{\T} \matr{V}_{\le s-1,\calY}) \\
 &= \det\Wronsk{s-1} \coefAtMonomial{\det(\matr{V}_{s,\calY} \widetilde{\matr{W}} \matr{V}^{\T}_{s,\calY} + \gamma \matr{V}_{\le s-1,\calY} \matr{V}_{\le s-1,\calY}^{\T}  )}{\gamma^{\#\PP_{s-1}}} 
\end{split}
\end{equation*}
and therefore
\begin{align}
\widetilde{e}_m  &= \det\Wronsk{s-1} \coefAtMonomial{e_m(\matr{V}_{s} \widetilde{\matr{W}}  \matr{V}^{\T}_{s} + \gamma \matr{V}_{\le s-1} \matr{V}_{\le s-1}^{\T}  )}{\gamma^{\#\PP_{s-1}}}  \nonumber \\
 & = \det\Wronsk{s-1}  \det(\matr{V}_{\le s-1}^{\T} \matr{V}_{\le s-1}) e_{m- \#\PP_{s-1}}(\matr{Q}_{s:k}^{\T}\matr{V}_{s} \widetilde{\matr{W}} \matr{V}^{\T}_{s}  \matr{Q}_{s:k} )\nonumber\\
  & = \det\Wronsk{s-1}  \det(\matr{V}_{\le s-1}^{\T} \matr{V}_{\le s-1}) e_{m- \#\PP_{s-1}}(\matr{Q}_{s}^{\T}\matr{V}_{s} \widetilde{\matr{W}} \matr{V}^{\T}_{s}  \matr{Q}_{s} ), \label{eq:esp_nd_smooth}
\end{align}
where $\matr{Q}_{s:k} \eqdef \begin{bmatrix}\matr{Q}_{s} & \cdots &\matr{Q}_{k} \end{bmatrix}$, the penultimate equality follows from \Cref{lem:saddle_point_mat}, and the last equality follows from the fact that only the top block of $\matr{Q}^{\T}_{s:k}\matr{V}_{s}$ is nonzero.
The rest of the proof follows from \Cref{lem:main_terms_esp_group} to \eqref{eq:esp_nd_smooth}. 

\noindent 3. We repeat the same steps as in the proof of \Cref{cor:1d_smooth_ev_ratio} (for groups of limiting eigenvectors).
Since, from  the proof of \Cref{thm:det_nd_smooth}, for any $0 \le \ell \le s$, the order of $(\matr{Q}_{\ell+1:k})^{\T}\matr{K}_{\varepsilon} \matr{Q}_{\ell+1:k}$ is $\O(\varepsilon^{2(\ell+1)})$ (as in \eqref{eq:block_s_smooth}), and the order of eigenvalues in $\O(1),\ldots,\O(\varepsilon^{2s})$ groups is exact, all the eigenvectors
\[
\{\vect{p}_{\ell,1}, \ldots, \vect{p}_{\ell,\#\HH_\ell}\}_{0 \le \ell \le s},
\]
must be orthogonal to $\mspan(\matr{Q}_{\ell+1:k})$, which proves the first part of the statement.
If, moreover, $s = k-1$, then the last block of eigenvectors (corresponding to eigenvalues of order $\O(\varepsilon^{2k})$) must be contained in $\mspan(\matr{Q}_{k})$.
\end{proof}

\begin{proof}[{Proof of \Cref{thm:det_nd_finite_smoothness}}]
\noindent 1.
The proof repeats that of \Cref{thm:det_1d_finite_smoothness} with the following minor modifications (in order to take into account the multivariate case):
\begin{itemize}
\item the matrix ${\matr{W}_{\ultriangle}} \in \RR^{\#\PP_{2r-2}\times \#\PP_{2r-2}}$ is defined as 
\[
{W}_{\ultriangle,\va,\vect{\beta}} = 
\begin{cases} 
 \frac{K^{(\va,\vect{\beta})}(0,0)}{\va!\vect{\beta}!}, & |\va|+|\vect{\beta}|\le2r, \\
0, & |\va|+|\vect{\beta}| > 2r;
\end{cases}
\]
i.e., in the sum \eqref{eq:sum_wronskian_diagonals} the matrices $\matr{W}_{\diagup s}$ are defined according to \eqref{eq:wronskian_diag_nd};
\item the matrix $\matr{\Delta}_k(\varepsilon)$ is defined in \eqref{eq:diag_scaling_nd}, and \Cref{lem:scaling_nd}  is used instead of \Cref{lem:scaling_1d}; 
\item the extended diagonal scaling matrix  is
\[
\widetilde{\matr{\Delta}} = \begin{bmatrix} \matr{\Delta}_{r-1} & \\
& \varepsilon^{r-1} \matr{I}_{N} \end{bmatrix} \in \RR^{n\times n},
\]
where $N = n - \#\PP_{r-1}$;
\item the last displayed formula in the proof of \Cref{thm:det_1d_finite_smoothness} becomes
\begin{equation*}
\begin{split}
& \varepsilon^{-2M(r-1,d)-2N(r-1)} \det\matr{K}_{\varepsilon} = \det(\widetilde{\matr{\Delta}}^{-1}\Qfull^{\T}\matr{K}_{\varepsilon}\Qfull\widetilde{\matr{\Delta}}^{-1}) \\
& = \varepsilon^{N} ((\det\Rthin)^2 \det \matr{W}
\det( f_{2r-1} \Qort^{\T}\matr{D}_{(2r-1)} \Qort) + \O(\varepsilon)).\\
\end{split}
\end{equation*}
\end{itemize}

\noindent 2. Again, as in the proof of  \Cref{thm:det_1d_finite_smoothness}, the matrix $\matr{Q}^{\T}_{s:k}\matr{K}_{\varepsilon}\matr{Q}_{s:k}$ has order given in \eqref{eq:block_s_finite_smoothness}, hence the  orders of the eigenvalues follow from \Cref{lem:courant_fischer}.

\noindent 3. The proof of this statement repeats the proof of  \Cref{thm:1d_finite_smoothness_ev} without changes.

\noindent 4. The last statement  follows from combining \eqref{eq:block_s_finite_smoothness} with  \Cref{lem:courant_fischer}, and proceeding as in the proof of the corresponding statement in {\Cref{thm:nd_smooth_ev}}.
\end{proof}

\section{Perturbation theory: a summary for Hermitian matrices}\label{sec:kato}
 This section contains a summary of  facts from \cite[Ch.
II]{kato1995perturbation} to deal with analytic perturbations of self-adjoint operators in a finite-dimensional vector space (i.e., Hermitian matrices).
Formally, and in keeping with the notation used in \cite{kato1995perturbation}, we assume that we are given a matrix-valued function depending on $\varepsilon$ such that
\begin{equation}
\matr{T}(\varepsilon) = \matr{T}^{(0)} + \varepsilon\matr{T}^{(1)}  + \varepsilon^2 \matr{T}^{(2)} + \cdots,
\end{equation}
where we assume that the matrices $\matr{T}^{(k)} \in \CC^{n\times n}$ are  Hermitian. 
In a neighborhood of $0$, $0 \in \mathcal{D}_0 \subset \CC$, $\matr{T}(\varepsilon)$ has  $s \le n$ semi-simple eigenvalues genericaly (i.e., except a finite number of exceptional points).
For simplicity of presentation\footnote{We can also consider the general case if needed.}, we assume that $s=n$, which is the case if there exists $\varepsilon_1\in \mathcal{D}_0$ having all distinct eigenvalues.
The interesting case (considered in this paper) is when $\varepsilon = 0$ is an exceptional point, i.e. $\matr{T}(0) =  \matr{T}^{(0)}$ has multiple eigenvalues  (e.g., a low-rank matrix with a multiple eigenvalue $0$).

\subsection{Perturbation of  eigenvalues and group eigenprojectors}
Since all matrices are Hermitian, by \cite[Theorem 1.10, Ch. II]{kato1995perturbation} (see \Cref{lem:analytic_eigenvalues}),  the eigenvalues $\lambda_1(\varepsilon), \ldots, \lambda_n(\varepsilon)$ and the   rank-one projectors 
on the corresponding eigenspaces $\matr{P}_1(\varepsilon), \ldots, \matr{P}_n(\varepsilon)$ are holomorphic functions of $\varepsilon$ in a neighborhood of $0$, $0 \in \mathcal{D} \subset \CC$.

\newpage

\begin{remark}\label{rem:splitting}
If the matrix $\matr{T}^{(0)}$ has $d$ multiple eigenvalues $\{\mu_k\}_{k=1}^d$ with multiplicities $m_1, \ldots, m_d$, i.e., after proper reordering, at $\varepsilon = 0$
\[
(\lambda_1(0), \ldots, \lambda_n(0)) = 
(\underbrace{\mu_1, \ldots,\mu_1}_{m_1\text{ times}}, \underbrace{\mu_2, \ldots,\mu_2}_{m_2\text{ times}}, \ldots,
\underbrace{\mu_d, \ldots,\mu_d}_{m_d\text{ times}})
\]
then the projectors on the invariant subspaces associated with $\mu_1,\ldots,\mu_s$ are sums of the corresponding rank-one projectors on the eigenspaces:
\begin{equation*}
\begin{split}
\matr{P}_{\mu,1} & = \matr{P}_1(0) + \ldots + \matr{P}_{m_1}(0) \\
\matr{P}_{\mu,2} & = \matr{P}_{m_1+1}(0) + \ldots + \matr{P}_{m_1+m_2}(0) \\
& \vdots \\
\matr{P}_{\mu,d} & = \matr{P}_{m_1 + \cdots + m_{d-1}+1}(0) + \ldots + \matr{P}_{n}(0) \\
\end{split}
\end{equation*}
\end{remark}
In this paper, our aim is to obtain a limiting eigenstructure at $\varepsilon = 0$.
In case of multiple  eigenvalues, this information cannot be retrieved from the spectral decompositon of $\matr{T}^{(0)}$ alone (we can only retrieve  the group projectors $\matr{P}_{\mu,j}$ from the spectral decomposition of $\matr{T}^{(0)}$).
In what follows, we look in details at perturbation expansions in order to find individual $\matr{P}_k(0)$.

As shown in \cite[Ch. II]{kato1995perturbation}, we can consider perturbations of a   possibly multiple  eigenvalue.
Let $\lambda$ be an eigenvalue of $\matr{T}^{(0)}$ of multiplicity $m$, and $\matr{P}$ is the corresponding orthogonal projector on the $m$-dimensional eigenspace.
The projector on the perturbed $m$-dimensional invariant subspace is
an analytic matrix-valued function 
\[
\matr{P}(\varepsilon) = \sum\limits_{k=0}^{\infty} \varepsilon^{k} \matr{P}^{(k)},
\]
with the coefficients given by
\begin{align*}
\matr{P}^{(0)}  =&\; \matr{P}, \quad \matr{P}^{(1)}   = -\matr{P} \matr{T}^{(1)} \matr{S} -\matr{S} \matr{T}^{(1)} \matr{P}, \\
 \matr{P}^{(k)}  =& -\sum\limits_{p=1}^k (-1)^p
                    \sum\limits_{\begin{smallmatrix}\nu_1 + \cdots + \nu_p = k
                        \\k_1 + \cdots + k_{p+1} = p \\ \nu_j \ge 1, k_j \ge
                        0\end{smallmatrix}} \matr{S}^{(k_1)} \matr{T}^{(\nu_1)}
  \matr{S}^{(k_2)} \cdots  \matr{S}^{(k_p)}\matr{T}^{(\nu_p)}
  \matr{S}^{(k_{p+1})}, \numberthis \label{eq:perturbation-series-projectors}
\end{align*}
where $\matr{S} = (\matr{T} - \lambda \matr{I})^{\dagger}$, $\matr{S}^{(0)} = -\matr{P}$ and $\matr{S}^{(k)} = \matr{S}^{k}$.

\subsection{Reduction and splitting the groups}\label{sec:perturbations_expressions}
In order to find the individual projectors of the eigenspaces corresponding to a multiple $\lambda$, and the expansion of the corresponding eigenvalues,
the following reduction (or splitting) procedure \cite[Ch. II, \S 2.3]{kato1995perturbation} can be applied, which localises the matrix to the $m$-dimensional subspace corresponding to the perturbations of $\lambda$.

We first define the eigennilpotent  matrix $\matr{D}(\varepsilon)$ as
\[
\matr{D}(\varepsilon) = (\matr{T}(\varepsilon) - \lambda \matr{I}) \matr{P}(\varepsilon) = \matr{P}(\varepsilon) (\matr{T}(\varepsilon) - \lambda \matr{I})   =  \matr{P}(\varepsilon)  (\matr{T}(\varepsilon) - \lambda \matr{I})\matr{P}(\varepsilon),
\]
which from  \cite[Ch. II, \S 2.2]{kato1995perturbation}  has an expansion
\[
\matr{D}(\varepsilon)  = 0 + \sum\limits_{k=1}^{\infty} \varepsilon^k  \widetilde{\matr{T}}^{(k)},
\]
where the expressions for $\widetilde{\matr{T}}^{(k)}$ are as follows:
\begin{align}
\widetilde{\matr{T}}^{(1)}   =& \matr{P} \matr{T}^{(1)} \matr{P}, \nonumber \\
\widetilde{\matr{T}}^{(2)}   = &\matr{P} \matr{T}^{(2)} \matr{P} - \matr{P} \matr{T}^{(1)} \matr{P} \matr{T}^{(1)} \matr{S}- \matr{P} \matr{T}^{(1)} \matr{S} \matr{T}^{(1)} \matr{P}-\matr{S} \matr{T}^{(1)} \matr{P} \matr{T}^{(1)} \matr{P}, \nonumber\\
 \widetilde{\matr{T}}^{(k)}  =& -\sum\limits_{p=1}^k (-1)^p
                                \sum\limits_{\begin{smallmatrix}\nu_1 + \cdots +
                                    \nu_p = k \\k_1 + \cdots + k_{p+1} = p-1 \\
                                    \nu_j \ge 1, k_j \ge 0\end{smallmatrix}}
  \matr{S}^{(k_1)} \matr{T}^{(\nu_1)}  \matr{S}^{(k_2)} \cdots
  \matr{S}^{(k_p)}\matr{T}^{(\nu_p)} \matr{S}^{(k_{p+1})}. \label{eq:perturbation-series-nilpotent}
\end{align}
\begin{remark}\label{rem:selfadjoint_tilde}
Note that the matrices $\widetilde{\matr{T}}^{(k)}$ are selfadjoint, which follows from the fact that for real $\varepsilon$ the matrices $\matr{T}(\varepsilon)$ and  $\matr{P}(\varepsilon)$ (and hence $\matr{D}(\varepsilon)$) are selfadjoint.
\end{remark}
Next, we define  the matrix $\widetilde{\matr{T}} (\varepsilon)$  as
\[
\widetilde{\matr{T}} (\varepsilon) = \frac{1}{\varepsilon} \matr{D}(\varepsilon) = \sum\limits_{k=0}^{\infty} \varepsilon^k  \widetilde{\matr{T}}^{(k+1)},
\]
such that $\widetilde{\matr{T}} (0) = \widetilde{\matr{T}}^{(1)}$.
Note that by \Cref{rem:selfadjoint_tilde}, all the matrices $\widetilde{\matr{T}}^{(k+1)}$ are Hermitian and all the eigenvalues of $\widetilde{\matr{T}}^{(1)} (\varepsilon)$  are holomorphic functions of $\varepsilon$.
The idea of the reduction process is to apply the perturbation theory to the matrix $\widetilde{\matr{T}}^{(1)} (\varepsilon)$.

Let $\widetilde{\matr{T}}^{(1)}$  have $s$ eigenvalues $\widetilde{\lambda}_{1}, \ldots, \widetilde{\lambda}_{s}$ with multiplicities
\[
\ell_1 + \cdots+ \ell_s = m,
\]
where we take into account only the $m$ eigenvalues\footnote{The other $n-m$ eigenvalues are $0$.} in the  subspace spanned by $\matr{P}(\varepsilon)$.

Then  $\widetilde{\lambda}_{1}, \ldots, \widetilde{\lambda}_{s}$ determine the splitting of $\lambda$ in the following way.
\begin{lemma}[A summary of {\cite[Ch. II, \S 2.3]{kato1995perturbation}}]\label{lem:splitting}
Let 
\[
\underbrace{\widetilde{\lambda}_{1,1}(\varepsilon),\ldots,\widetilde{\lambda}_{1,\ell_1}(\varepsilon)}_{\text{pert.\ of }\widetilde{\lambda}_{1}}, \ldots, 
\underbrace{\widetilde{\lambda}_{s,1}(\varepsilon),\ldots,\widetilde{\lambda}_{s,\ell_s}(\varepsilon)}_{\text{pert.\ of }\widetilde{\lambda}_{s}},
\]
be the holomorphic functions for the perturbations of the eigenvalues $\widetilde{\lambda}_{1}, \ldots, \widetilde{\lambda}_{s}$ of $\widetilde{\matr{T}} (\varepsilon)$.
Then the holomorphic functions corresponding to perturbations of  the eigenvalue $\lambda$ of the original matrix $\matr{T}(\varepsilon)$ are given by
\[
\{\lambda_1(\varepsilon), \ldots,\lambda_m(\varepsilon) \} = 
\bigcup\limits_{k=1}^s
\{\lambda + \varepsilon \widetilde{\lambda}_{k,1}(\varepsilon), \ldots, \lambda + \varepsilon \widetilde{\lambda}_{k,\ell_k}(\varepsilon)\}.
\]
Moreover, the expansions $\widetilde{\matr{P}}_{k,j}(\varepsilon)$ of the projectors on the eigenspaces of $\widetilde{\matr{T}} (\varepsilon)$ (corresponding to $\widetilde{\lambda}_{k,j}(\varepsilon)$) give the expansions of the projectors on the  eigenspaces of $\matr{T}(\varepsilon)$ corresponding to $\lambda_1(\varepsilon), \ldots,\lambda_m(\varepsilon)$.
\end{lemma}
\Cref{lem:splitting} can be applied recursively: for each individual eigenvalue  $\widetilde{\lambda}_{k}$ (of multiplicity $\ell_k>1$) we  can  consider the corresponding reduced matrix
\[
 \tildeN{2}{\matr{T}}(\varepsilon) = 
\frac{1}{\varepsilon}   (\widetilde{\matr{T}}^{(1)}(\varepsilon) - \widetilde{\lambda}_{k} \matr{I})  \widetilde{\matr{P}}_k(\varepsilon),
\]
where $\widetilde{\matr{P}}_k(\varepsilon)$ is the perturbation of the total projection on the $\ell_k$-dimensional eigenspace corresponding to $\widetilde{\lambda}_{k}$ (which can be computed as in the previous subsection).
Depending on the eigenvalues of the main term of the reduced matrix, either the splitting will occur again, or there will be no splitting; in any case, after a finite number of steps, all the individual eigenvalues will be split into simple (multiplicity $1$) eigenvalues\footnote{This follows from our assumption that the eigenvalues are simple generically.}.

\section{Results on eigenvectors for finitely smooth kernels}\label{sec:eigenvectors}
In this section we are going to prove the following result for the multivariate case.
\begin{theorem}
  \label{thm:finite_smoothness_eigenvectors}
Let the radial kernel be as in \Cref{thm:det_nd_finite_smoothness},  with $\matr{K}_{\varepsilon}$ positive semidefinite on $[0,\varepsilon_0], \varepsilon_0>0$ and analytic in $\varepsilon$. 
If $\det \Wronsk{r-1} \neq 0$, $\matr{V}_{\le r-1}$ is full rank, and $\Qort^{\T} \bD_{(2r-1)}  \Qort$ is invertible,
then the eigenvectors  corresponding to the last group of $\O(\varepsilon^{2r-1})$ eigenvalues,  are given by the columns of 
\[
\Qort \matr{U},
\]
where the columns of $\matr{U}$ are eigenvectors of $\Qort^{\T}\matr{D}_{(2r-1)} \Qort$,
and the matrix $\Qort$ comes from the full QR factorization \eqref{eq:qr_full} of $\matr{V}_{\le r-1}$ (as in \Cref{thm:det_nd_finite_smoothness}).
\end{theorem}
We  conjecture that for finitely-smooth kernels as well, the individual eigenvectors for the  $\O(\varepsilon^{2s})$ groups, $0 \le s \le r-1$, can be obtained as in \Cref{conj:eigenvectors}.

\subsection{Block staircase matrices}
\label{sec:block-staircase}
We first need some facts about  a class of so-called block staircase  matrices. 
Let $\vect{n} = (n_0, \ldots, n_s)\in \mathbb{N}^{s+1}$ such that
\[
n = n_0 + \cdots + n_s
\]
and consider the following block partition of a matrix $\matr{A} \in \CC^{m \times m}$
\[
\matr{A}=
\begin{bmatrix}
\matr{A}^{(0,0)} & \cdots & \matr{A}^{(0,s)} \\
\vdots &   &\vdots \\
\matr{A}^{(s,0)} & \cdots & \matr{A}^{(s,s)} \\
\end{bmatrix}
\]
where the blocks are of size $\matr{A}^{(k,l)} \in \CC^{n\times n}$.
Assuming that the partition is fixed, we define the classes $\stair{p}$ of ``staircase'' matrices 
\[
\{\matr{0}\} =\stair{-1} \subset \stair{0} \subset \stair{1} \subset\stair{2} \subset \cdots \subset \stair{2s-1}  = \stair{2s} = \stair{2s+1} = \cdots = \CC^{n\times n}, 
\]
such that the matrices in  $\mathscr{S}_p$ have nonzero blocks only up to the $s$-th antidiagonal
\[
\mathscr{S}_p = \{ \matr{A} \in \CC^{m\times n} \,|\, \matr{A}^{(k,l)} = 0 \mbox{ for all } k + l > p \}.
\]
In \Cref{fig:staircase}, we illustrate the classes for $s=2$:

\begin{figure}[t!]
$\stair{0} = 
\left\{\raisebox{-0.42\height}{\begin{tikzpicture}
\path [fill = lightgray] (0,0.9) --(0.3,0.9) -- (0.3,0.6) -- (0.0,0.6) --(0,0.9);
\draw (0,0) rectangle (0.9,0.9); 
\end{tikzpicture}}\right\},
\;
\stair{1} = 
\left\{\raisebox{-0.42\height}{\begin{tikzpicture}
\path [fill = lightgray] (0,0.9) --(0.6,0.9) -- (0.6,0.6) --(0.3,0.6) --(0.3,0.3) --  (0.0,0.3) --(0,0.9);
\draw (0,0) rectangle (0.9,0.9); 
\end{tikzpicture}}\right\},
\;
\stair{2} = 
\left\{\raisebox{-0.42\height}{\begin{tikzpicture}
\path [fill = lightgray] (0,0.9) --(0.9,0.9) -- (0.9,0.6) --(0.6,0.6) -- (0.6,0.3) --(0.3,0.3) --(0.3,0.0) --  (0.0,0.0) --(0,0.9);
\draw (0,0) rectangle (0.9,0.9); 
\end{tikzpicture}}\right\},
\;
\stair{3} = 
\left\{\raisebox{-0.42\height}{\begin{tikzpicture}
\path [fill = lightgray] (0,0.9) --(0.9,0.9) -- (0.9,0.3) --(0.6,0.3) --(0.6,0.0) --  (0.0,0.0) --(0,0.9);
\draw (0,0) rectangle (0.9,0.9); 
\end{tikzpicture}}\right\},
\;
\stair{4} = 
\left\{\raisebox{-0.42\height}{\begin{tikzpicture}
\path [fill = lightgray] (0,0.9) --(0.9,0.9) -- (0.9,0.0) -- (0.0,0.0) --(0,0.9);
\draw (0,0) rectangle (0.9,0.9); 
\end{tikzpicture}}\right\}.$
\caption{Classes of staircase matrices. White color stands for zero blocks.}\label{fig:staircase}
\end{figure}

The following obvious property will be useful.
\begin{lemma}\label{lem:staircase}
\begin{enumerate}
\item $\matr{A} \in \stair{p},  \matr{B} \in \stair{q} \Rightarrow \matr{A}\matr{B} \in  \stair{p+q}$.
\item For any  $\matr{A} \in \stair{p}$  and upper triangular $\matr{R}$, it follows that $\matr{R}\matr{A}\matr{R}^{\T} \in \stair{p}$.
\item For any  $\matr{A} \in \stair{p}$ and block-diagonal matrix $\matr{\Lambda} $ it holds that $\matr{\Lambda} \matr{A},  \matr{A} \matr{\Lambda}  \in  \stair{p}$.
\end{enumerate} 
\end{lemma}
\begin{proof}
 The proof follows from straighforward verification. 
\end{proof}

\subsection{Proof of \Cref{thm:finite_smoothness_eigenvectors}}
\label{sec:proof-eigenvec-finite-smoothness}
\begin{proof}
The kernel matrix has an expansion with only even powers of $\varepsilon$ until $2r-2$
\[
\matr{K}_{\varepsilon} =  \sum_{j=0}^{r-1} \varepsilon^{2j} \matr{V}_{\le 2j} \matr{W}_{\diagup 2j} \matr{V}^{\T}_{\le 2j} +
\varepsilon^{2r-1} f_{2r-1} \matr{D}_{(2r-1)}+ 
 \O(\varepsilon^{2r}),
\]
since odd block diagonals $\matr{W}_{\diagup 2j-1}$, $j < r$ vanish.
We look at the transformed matrix
\[
\matr{T}({\varepsilon}) = \Qfull^{\T} \matr{K}_{\varepsilon}\Qfull  = \matr{T}_{0} + \varepsilon \matr{T}_{1} + \cdots +  \varepsilon^{2r-1} \matr{T}_{2r-1} + \O(\varepsilon^{2r}), 
\]
where $\Qfull$ is the matrix of the full QR decomposition of $\matr{V}_{\le 2r-2}$
\[
\matr{T}_{s} =
\begin{cases}  
 \Qfull^{\T} \matr{V}_{\le s} \matr{W}_{\diagup s} \matr{V}^{\T}_{\le s} \Qfull, & s \text{ even}, s < 2r-1, \\
0, & s \text{ odd}, s < 2r-1, \\
f_{2r-1} \Qfull^{\T} \matr{D}_{(2r-1)} \Qfull, &  s = 2r-1. \\
\end{cases}
\]
Due to the fact that that $\matr{W}_{\diagup s}$ is block antidiagonal, $\Qfull^{\T} \matr{V}_{\le s}$ is upper triangular and by \Cref{lem:staircase}, we have that $\matr{T}_{s}$ is block staircase for $s < 2r-1$, i.e. $\matr{T}_{\diagup s} \in \stair{s}$.

We proceed by a series of Kato's reductions, according to \Cref{lem:splitting}.
At each order of $\varepsilon$,  a multiple eigenvalue $0$ is split into a group of nonzero eigenvalues and  eigenvalue $0$ of smaller multiplicity. 
Formally, we  consider a sequence of reduced matrices 
\[
\tildeN{\!\!\!\!\! s+1}{\matr{T}} (\varepsilon) = \frac{1}{\varepsilon}  \tildeN{s}{\matr{T}} (\varepsilon) \matr{P}_s(\varepsilon), \quad \tildeN{0}{\matr{T}} (\varepsilon) = \matr{T}(\varepsilon)
\]
where $\matr{P}_k(\varepsilon)$ is the projector onto the perturbation of the  nullspace of $\tildeN{s}{\matr{T}} (0)$ (i.e., the invariant subspace associated with the eigenvalue $0$).
Its power series expansion 
\[
 \tildeN{s}{\matr{T}} (\varepsilon) = \sum\limits_{j=s}^{\infty} \varepsilon^{j-s} \cdot \tildeN{s}{\matr{T}}_{\!\!j}
\]
 can be computed according to \eqref{eq:perturbation-series-nilpotent}.
For each matrix we will be interested only in the first $2r-s$ terms,
as summarized below,
\[
\begin{array}{ccccccc}
1 & \varepsilon & \varepsilon^2 &  \cdots&\varepsilon^{2r-1}  \\\hline
\matr{T}_0 &  & &  & & & \\ 
\matr{T}_1 & \widetilde{\matr{T}}_{1} & &  & \\ 
\matr{T}_2 & \widetilde{\matr{T}}_{2} & \tildeN{2}{\matr{T}}_{\!\!2}   && \\ 
\vdots &   &    &\ddots & \\ 
\matr{T}_{2r-1} & \widetilde{\matr{T}}_{2r-1} & \tildeN{2}{\matr{T}}_{\!\!2r-1}   && \tildeN{\!\!\!\!\!2r-1}{\matr{T}}_{\!\!2r-1}     \end{array}
\]
since we are interested only in the terms
\[
\matr{T}_0 , \widetilde{\matr{T}}_{1}, \tildeN{2}{\matr{T}}_{\!\!2}, \ldots, \tildeN{\!\!\!2k}{\matr{T}}_{\!\!2k},
\]
whose eigenvectors give limiting eigenvectors for the original matrix $\matr{T}(\varepsilon)$.

Next, we will look in detail at the form of  coefficients of the reduced  matrices.
The projector on the image space of $\matr{T}_0$ is $\matr{\Pi}_0 = \vect{e}_0 \vect{e}_0 ^{\T}$ (where $\vect{e}_0 = (1,0,\ldots,0)^{\T}$), the projector on the nullspace is 
\[
\matr{P}_0 = 
\begin{bmatrix}
0  &\\
& \matr{I}_{n-1} \\
\end{bmatrix},
\]
and the matrix $\matr{S} = \matr{S}_0$ in \eqref{eq:perturbation-series-projectors}.
By examining the terms in \eqref{eq:perturbation-series-projectors}, we have that the coefficients of the reduced matrices preserve the staircase class, i.e. $\widetilde{\matr{T}}_s \in \stair{s}$.
This can be seen by verifying that if $\matr{A}_s \in \stair{s}$ and $\matr{S}^{(j)}$ are diagonal, then the products
\begin{equation}\label{eq:pert_expansion_one_term}
\matr{S}^{(k_1)} \matr{A}^{(\nu_1)}  \matr{S}^{(k_2)} \cdots  \matr{S}^{(k_p)}\matr{A}^{(\nu_p)} \matr{S}^{(k_{p+1})} \in \stair{s}
\end{equation}
if  $\nu_1 + \cdots + \nu_{p} = s$.
Next, we note that since ${\matr{T}}_1 \in \stair{1}$, then we have
\[
\widetilde{\matr{T}}_1 = \matr{P}_0 {\matr{T}}_1 \matr{P}_0 = 0.
\]
Hence we have that $\matr{P}_1 = \matr{I}$ and $\matr{S}_1 = 0$, and the second step of reduction does not change the matrices, i.e. $\tildeN{2}{\matr{T}}_{\!\!s} = \widetilde{\matr{T}}_{s}$.

Proceeding by induction,  at the step of reduction $(2s-2) \to (2s-1)$ the staircase order of the matrices is preserved due to \eqref{eq:pert_expansion_one_term}, and block diagonality of $\matr{P}_{2s-1}$ and $ \tildeN{\!\!\!\!\!2s-2}{\matr{T}}_{\!\!2s-2}$.
Since the reduction step $(2s-1) \to 2s$ does not change anything, we

\noindent get
\begin{equation}\label{eq:ev_reduction_P2s}
\matr{P}_{2s} = 
\begin{bmatrix}
0_{(\# \PP_{s}) \times (\# \PP_{s})} &  \\
 &\matr{I}_{(n - \# \PP_{s}) \times (n -\# \PP_{s})}\end{bmatrix},
 \end{equation}
which we know from  \Cref{thm:det_nd_finite_smoothness}, and 
\begin{equation}\label{eq:ev_reduction_T2s}
\tildeN{\!\!\!2s}{\matr{T}}_{\!\!2s}  = 
\begin{bmatrix}
 0_{(\# \PP_{s}) \times (\# \PP_{s})} &  & \\
 &*_{(\# \HH_{s+1}) \times (\# \HH_{s+1})}& \\
&& 0_{(n - \# \PP_{s+1}) \times (n -\# \PP_{s+1})}
\end{bmatrix},
 \end{equation}
The last reduction step is different, as we get $\tildeN{\!\!\!\!\!2r-1}{\matr{T}}_{\!\!2r-1}$ which is not equal to zero.
In order to obtain $\tildeN{\!\!\!\!\!2r-1}{\matr{T}}_{\!\!2r-1}$, we note the following: at the first step of the reduction  the matrices $\widetilde{\matr{T}}_{2j-1}$ defined by \eqref{eq:perturbation-series-nilpotent}, are the sums of the terms running over multi-indices
\[ 
\nu_1 + \cdots +  \nu_p = 2j-1,
\]
where at least one of $\nu_\ell$ should be odd and all $\nu_\ell \le 2j-1$. 
Therefore, we have  $\widetilde{\matr{T}}_{2j-1} = 0$ if $j < r$  and $\widetilde{\matr{T}}_{2r-1} = \matr{P}{\matr{T}}_{2r-1}\matr{P}$.
Proceeding by induction, we get that 
\[
\tildeN{\!\!\!\!\!s}{\matr{T}}_{2j-1} = 
\begin{cases}
0, & j < r \mbox{ and } s < 2r-1,\\
f_{2r-1}\matr{P}_{2r-2}\matr{T}_{2r-1}  \matr{P}_{2r-2}, & j = r \mbox{ and } s = 2r-1, \\
\end{cases}
\]
where $\matr{P}_{2r-2}$ is defined in \eqref{eq:ev_reduction_P2s}.

Thus we have that the limiting eigenvectors of $\Qfull\matr{T}_{\varepsilon}\Qfull^{\T}$  for the order $\O(\varepsilon^{2r-1})$ are the limiting eigenvectors (corresponding to non-zero eigenvalues) of the matrix 
\[ \Qort\Qort^{\T} \matr{D}_{(2r-1)} \Qort\Qort^{\T}, \]
where $\Qfull = \begin{bmatrix}\Qthin & \Qort\end{bmatrix}$ is the splitting
of $\Qfull$ such that $\Qort \in \RR^{n \times (n-\#\PP_{r-1})}$. 
\end{proof}

\section{Discussion}
\label{sec:discussion}

We have shown that kernel matrices become tractable in the flat limit, and
exhibit deep ties to orthogonal polynomials. We would like to add some remarks
and highlight some open problems. 

First, we expect our analysis to generalise in a mostly straightforward manner
to the ``continuous'' case, i.e., to kernel integral operators. This should make
it possible to examine a double asymptotic, in which $n \rightarrow \infty$ as
$\varepsilon \rightarrow 0$. One could then leverage recent  results on
the asymptotics of orthogonal polynomials, for instance
\cite{KrooLubinsky:ChristoffelFuncs}.

Second, our results may be used empirically to create preconditioners for kernel
matrices. There is already a vast literature on approximate kernel methods,
including in the flat limit (e.g., \cite{fornberg2011stable,le2019factorization}), and
future work should examine how effective polynomial preconditioners are compared
to other available methods.

Third, many interesting problems (e.g., spectral clustering
\cite{von2008consistency}) involve not the kernel matrix itself but some
rescaled variant. We expect that multiplicative perturbation theory could be
brought to bear here \cite{sosa2016first}. 

Finally, while multivariate polynomials are relatively well-understood objects,
our analysis also shows that in the finite smoothness case, a central role is
played by a different class of objects: namely, multivariate polynomials are
replaced by the eigenvectors of distance matrices of an odd power. To our
knowledge, very little has been proved about such objects but some literature from
statistical physics \cite{bogomolny2003spectral} points to a link to ``Anderson localization''. Anderson
localization is a well-known phenomenon in physics whereby eigenvectors
of certain operators are localised, in the sense of having fast decay over
space. This typically does not hold for orthogonal polynomials, which tend
rather to be localised in frequency. Thus, we conjecture that eigenvectors of
completely smooth kernels are localised in frequency, contrary to eigenvectors
of finitely smooth kernels, which (at low energies) are localised in space. The
results in \cite{bogomolny2003spectral} are enough to show that this holds for
the exponential kernel in $d=1$, but extending this to $d>1$ and higher
regularity orders is a fascinating and probably nontrivial problem.
\appendix

\section{Proofs for elementary symmetric polynomials}\label{sec:proofs_esp}
\begin{proof}[Proof of \Cref{lem:main_terms_esp_separated}]
\begin{enumerate}
\item By definition, the ESPs can be expanded as
\begin{align}
e_s(\varepsilon) & = \sum\limits_{1\le t_1 < \cdots < t_s \le n} \; \prod\limits_{j=1}^s \varepsilon^{L_{t_j}} (\widetilde{\lambda}_{t_j} + \O(\varepsilon)) \nonumber \\
&=
\varepsilon^{L_1 + \cdots + L_s} \Big(\underbrace{\sum\limits_{\substack{1\le t_1 < \cdots < t_s \le n \\ L_1 + \cdots + L_s = L_{t_1} + \cdots + L_{t_s}}} \;\prod\limits_{j=1}^s \widetilde{\lambda}_{t_j}}_{\widetilde{e}_s} + \O(\varepsilon)\Big),\label{eq:esp_orders_from_actual}
\end{align}
which follows from the fact that $L_{t_1} + \cdots + L_{t_s}$ is minimized at $t_j = j$.

\item The case $s=n$ is obvious, because there is only one possible tuple $(t_1,\ldots,t_s)$.
Consider the case $L_s < L_{s+1}$, $1 \le s < n$.
If $t_s > s$, then the sum is increased
\[
L_1 + \cdots + L_s < L_{t_1} + \cdots + L_{t_s},
\]
hence  $(1,\ldots,s)$  is the only possible choice for $(t_1,\ldots,t_s)$  in \eqref{eq:esp_orders_from_actual}.
\end{enumerate}
\end{proof}
\begin{proof}[Proof of \Cref{lem:main_terms_esp_group}]
When \eqref{eq:group_degrees} is satisfied, we need to find
\[
\widetilde{e}_{s+k} =  \sum\limits_{\substack{1\le t_1 < \cdots < t_{s+k} \le n \\ L_1 + \cdots + L_{s+k} = L_{t_1} + \cdots + L_{t_{s+k}}}} \;\prod\limits_{j=1}^{s+k} \widetilde{\lambda}_{t_j}
\]
where, as in the previous case, the minimum sum  $L_{t_1} + \cdots + L_{t_{s+k}}$ is achieved by
\[
L_{1} + \cdots + L_{s} +L_{s+1} + \cdots +  L_{{s+k}},
\]
and the sum increases if $t_{s} > s$ or  if $t_{s+k} > s+m$.
Therefore, $(t_1,\ldots,t_s) = (1,\ldots,s)$ and the main term becomes
\[
 \widetilde{e}_{s+k} =  \prod\limits_{i=1}^{s} \widetilde{\lambda}_{i} \left(\sum\limits_{\substack{s+1\le t_{s+1} < \cdots < t_{s+k} \le s+m}} \;\prod\limits_{j=1}^{k} \widetilde{\lambda}_{t_{s+j}}\right).
\]
\end{proof}
\section{Proofs for saddle point matrices}\label{sec:proofs_saddle}
\begin{proof}[{Proof of \Cref{lem:saddle_point_mat}}]
We  note that $(\det\Rthin)^2 = \det(\matr{V}^{\T} \matr{V})$ and 
\begin{align*}
&\det\begin{bmatrix}
\matr{A} &  \matr{V} \\
 \matr{V}^{\T} & 0 
\end{bmatrix} = 
\det \left(
\begin{bmatrix}\Qfull^{\T} & \\ & \matr{I}_r \end{bmatrix}
\begin{bmatrix}
\matr{A} &  \matr{V} \\
 \matr{V}^{\T} & 0 
\end{bmatrix} \begin{bmatrix}\Qfull & \\ & \matr{I}_r \end{bmatrix} \right) \\
&=
\det
\begin{bmatrix}
\Qthin^{\T}\matr{A} \Qthin^{\T}  &\Qthin^{\T}\matr{A} \Qort  &  \Rthin \\
\Qort^{\T}\matr{A} \Qthin  &\Qort^{\T}\matr{A} \Qort & \\
 \Rthin^{\T} & & 0 
\end{bmatrix} =  (-1)^r (\det\Rthin)^2 \det(\Qort^{\T}\matr{A} \Qort).
\end{align*}
\end{proof}

Before proving  \Cref{lem:esp_coef_perturbation}, we need a technical lemma first.
\begin{lemma}\label{lem:low-rank-update-det}
For $\matr{G} \in \RR^{r\times r}$ and $\matr{M} =       \begin{pmatrix}
        \matr{A} & \matr{B} \\
        \matr{C} & \matr{D} 
      \end{pmatrix} \in \RR^{n \times n}$, $\matr{A} \in \RR^{r \times r}$ it holds that 
\[
  \coefAtMonomial{\det \left( \matr{M} + t
      \begin{pmatrix}
        \matr{G} & 0 \\
        0 & 0
      \end{pmatrix}
    \right)}{t^r} = \det \matr{G} \det \matr{D}
  \]
\end{lemma}
\begin{proof}
From  \cite[Theorem 1]{bhatia2009higher}, we have that\footnote{Note that in the sum in \cite[eq. (7)]{bhatia2009higher} only one term is nonzero.}
\[
\frac{1}{r!} \frac{d^r}{dt^r}  \det \left( \matr{M} + t
      \begin{pmatrix}
        \matr{G} & 0 \\
        0 & 0
      \end{pmatrix}
    \right) = \det  \begin{pmatrix}
        \matr{G} & \matr{B} \\
       0 & \matr{D} 
      \end{pmatrix}  = \det \matr{G} \det \matr{D}.
\]
\end{proof}
\begin{proof}[{Proof of \Cref{lem:esp_coef_perturbation}}]
Due to invariance under similarity transformations of the elementary polynomials, we get
\begin{align*}
&\coefAtMonomial{e_k(\matr{A} + t \matr{V}\matr{V}^{\T})}{t^{r}} = 
\coefAtMonomial{e_k(\Qfull^{\T}(\matr{A} + t \matr{V}\matr{V}^{\T})\Qfull)}{t^{r}} \\
& = 
\coefAtMonomial{e_k\Big(\underbrace{\begin{bmatrix}\Qthin^{\T}\matr{A}\Qthin + t  \Rthin\Rthin^{\T} &\Qthin^{\T}\matr{A}\Qort \\ \Qort^{\T}\matr{A}\Qthin& \Qort^{\T}\matr{A}\Qort \end{bmatrix}}_{\matr{B}(t)}\Big)}{t^{r}} \\
& = \sum_{\substack{|\calJ| = k\\\calJ \subseteq \{1,\ldots, n\}}} \coefAtMonomial{\det \left(\matr{B}_{\calJ,\calJ}(t)\right)}{t^{r}} = \sum_{\substack{|\calJ| = k\\\{1,\ldots, r\} \subseteq \calJ \subseteq \{1,\ldots, n\}}} \coefAtMonomial{\det \left(\matr{B}_{\calJ,\calJ}(t)\right)}{t^{r}},
\end{align*}
where the last equality follows from the fact that the polynomial  $\det \left(\matr{B}_{\calJ,\calJ}(t)\right)$ has degree that is equal to the cardinality of the intersection $\calJ \cap \{1,\ldots,r\}$.
Any such $\calJ$ can be written as $\{1,\ldots,r\} \cup\calJ'$, where  $\calJ' \subseteq \{r+1,\ldots, n\}$.
Applying \Cref{lem:low-rank-update-det} to each term individually, we get 
\[
 \coefAtMonomial{\det  \left(\matr{B}_{\calJ,\calJ}(t)\right)}{t^r}
  = \det(\Rthin\Rthin^{\T})  \det \left((\Qort^{\T}\matr{A}\Qort )_{\calJ',\calJ'}\right), \\
\]
thus, summing over all $\calJ' \subseteq \{r+1,\ldots, n\}$ yields
\[
\det(\Rthin\Rthin^{\T})  \sum_{\substack{\calJ' \subseteq \{r+1,\ldots, n\}\\|\calJ'| = k-r}} \det \left((\Qort^{\T}\matr{A}\Qort )_{\calJ',\calJ'}\right)  = \det(\Rthin\Rthin^{\T}) e_{k-r}(\Qort^{\T}\matr{A}\Qort ).
\]
\end{proof}

\section*{Acknowledgments}
The authors would like to acknowledge Pierre Comon for his help at the beginning
of this project, and Cosme Louart and Malik Tiomoko for their help in improving
the presentation of the results. We would also like to thank two anonymous
referees for their careful reading of our work and many suggestions.

\bibliographystyle{siamplain}
\bibliography{flat_limit}

\end{document}